\documentclass[oneside,english]{amsart}
\usepackage[T1]{fontenc}
\usepackage[latin9]{inputenc}
\usepackage{babel}
\usepackage{mathrsfs}
\usepackage{mathtools}

\usepackage{bm}
\usepackage{amstext}
\usepackage{amsthm}
\usepackage{amssymb}
\usepackage{stmaryrd}
\usepackage[unicode=true,pdfusetitle,
 bookmarks=true,bookmarksnumbered=false,bookmarksopen=false,
 breaklinks=false,pdfborder={0 0 1},backref=false,colorlinks=false]
 {hyperref}
\usepackage{cleveref}
\usepackage{enumitem}
\crefname{equation}{}{}
\makeatletter
\numberwithin{equation}{section}
\numberwithin{figure}{section}
\theoremstyle{plain}
\newtheorem{thm}{\protect\theoremname}[section]
\theoremstyle{plain}
\newtheorem{assumption}[thm]{\protect\assumptionname}
\theoremstyle{definition}
\newtheorem{defn}[thm]{\protect\definitionname}
\theoremstyle{remark}
\newtheorem{rem}[thm]{\protect\remarkname}
\theoremstyle{plain}
\newtheorem{prop}[thm]{\protect\propositionname}
\theoremstyle{plain}
\newtheorem{lem}[thm]{\protect\lemmaname}
\theoremstyle{remark}
\newtheorem*{acknowledgement*}{\protect\acknowledgementname}

\allowdisplaybreaks[4]

\usepackage{geometry}
\let\std@footnotetext\@footnotetext
\usepackage{setspace}
\let\@footnotetext\std@footnotetext

\makeatother

\providecommand{\acknowledgementname}{Acknowledgement}
\providecommand{\assumptionname}{Assumption}
\providecommand{\definitionname}{Definition}
\providecommand{\lemmaname}{Lemma}
\providecommand{\propositionname}{Proposition}
\providecommand{\remarkname}{Remark}
\providecommand{\theoremname}{Theorem}

\begin{document}
\title[The Dirichlet problem for nonlinear diffusion equations]{Entropy solutions to the Dirichlet problem for nonlinear diffusion
equations with conservative noise}
\author{Kai Du}
\address[K. Du]{Shanghai Center for Mathematical Sciences, Fudan University, Shanghai 200438, China.}
\email{kdu@fudan.edu.cn}

\author{Ruoyang Liu}
\address[R.Liu]{Corresponding author. School of Mathematical Sciences, Fudan University, Shanghai 200433, China.}
\email{ruoyangliu18@fudan.edu.cn}
\author{Yuxing Wang}
\address[Y. Wang]{School of Mathematical Sciences, Fudan University, Shanghai 200433, China.} 
\email{wangyuxing@fudan.edu.cn}

\begin{abstract}
Motivated by porous medium equations with randomly perturbed velocity field, this paper considers a class of nonlinear degenerate diffusion equations with nonlinear conservative noise in bounded domains. 
The existence, uniqueness and $L_{1}$-stability
of non-negative entropy solutions under the homogeneous Dirichlet boundary condition are proved.
The approach combines Kruzhkov's doubling variables technique with
a revised strong entropy condition that is automatically satisfied 
by the solutions of approximate equations.
\end{abstract}

\maketitle

\section{Introduction\label{sec:Introduction}}

This paper is concerned with the Dirichlet problem for nonlinear diffusion
equations
\begin{equation}
\begin{aligned}\mathrm{d}u(t,x) & =\big(\Delta\Phi(u)+\nabla\cdot G(x,u)+F(x,u)\big)\mathrm{d}t\\
 & \quad+\big(\nabla\cdot\sigma^{k}(x,u)\big)\circ\mathrm{d}W^{k}(t),\quad(t,x)\in(0,T)\times D;\\
u(0,x) & =\xi(x),\quad x\in D;\\
u(t,x) & =0,\quad(t,x)\in[0,T]\times\partial D,
\end{aligned}
\label{eq:S-integral}
\end{equation}
where $D$ is a bounded domain in $\mathbb{R}^{d}$ with smooth boundary, and
$\Phi:\mathbb{R\rightarrow\mathbb{R}}$ is a monotone function. The noise $\{W^{k}\}_{k\in\mathbb{N}}$
is a sequence of independent standard Brownian motions defined on
a complete filtered probability space $(\Omega,\mathcal{F},(\mathcal{F}_{t})_{t\ge0},\mathbb{P})$,
and the stochastic integral is in the Stratonovich sense. The Einstein
summation convention is used throughout this paper.

A typical example of \cref{eq:S-integral} and also our primary motivation
of this paper is the porous medium equation in random environment,
which describes the flow of an ideal gas in a homogeneous porous medium:
let $u$ be the gas density that satisfies
\begin{equation}
\varepsilon\partial_{t}u+\nabla\cdot(u\bm{V})=0,\label{eq:mass balance}
\end{equation}
where $\varepsilon\in(0,1)$ is the porosity of the medium, and $\bm{V}$
is the randomly perturbed velocity field of the form
\begin{equation}
\bm{V}=\bm{V}_{0}+\sigma^{k}\circ\dot{W}^{k},\label{eq:random velocity field}
\end{equation}
where $\bm{V}_{0}$ is derived from Darcy's law 
\[
\mu\bm{V}_{0}=-k\nabla p=-k\nabla(p_{0}u^{m-1}).
\]
Then, we can informally derive a stochastic porous medium equation
\[
\mathrm{d}u(t,x)=c\Delta u^{m}\mathrm{d}t+\varepsilon^{-1}\nabla\cdot(u\sigma^{k})\circ\mathrm{d}W^{k}(t).
\]
For more details of the model, we refer to \cite{vazquez2007porous}
and references therein. Moreover, equations of type \cref{eq:S-integral}
also arise as limits of interacting particle systems driven by common
noise from mean field models \cite{lasry2006jeuxa,lasry2006jeux,lasry2007mean},
and as simplified models of fluctuations in non-equilibrium statistical
physics \cite{dirr2016entropic}; for more applications we refer to
\cite{fehrman2019well,dareiotis2020nonlinear} and the references
therein.

The well-posedness of the Cauchy problem for stochastic nonlinear diffusion equations with general noise has
been investigated in various frameworks, for example, the variational
approach in the space $H^{-1}$ (cf. \cite{barbu2008existence,barbu2016stochastic,ciotir2020stochastic}
etc.), the kinetic formulation (cf. \cite{debussche2016degenerate,gess2018well,fehrman2021path}
etc.), and the entropy formulation (cf. \cite{bauzet2015degenerate,dareiotis2019entropy,dareiotis2020ergodicity}
etc.).  The case of linear gradient noise,  say $\sigma(x,u)=h(x)u$ in \cref{eq:S-integral}, has been studied in \cite{dareiotis2019supremum,tolle2020stochastic,ciotir2020stochastic}, and the general case was discussed in \cite{fehrman2019well} by a kinetic approach with rough path techniques (cf. \cite{lions2013scalar,lions2014scalar,gess2015scalar,gess2017long,gess2017stochastic}),  requiring $\sigma\in C^\gamma_b(\mathbb{T}^d\times\mathbb{R})$ for some $\gamma>5$.  
The paper \cite{dareiotis2020nonlinear} relaxed the regularity assumption to $\sigma\in C^3_b(\mathbb{T}^d\times\mathbb{R})$ and treated general diffusion nonlinearity $\Phi$ under the entropy formulation.
The recent work \cite{fehrman2021well} introduced a concept of stochastic kinetic solution based on the concept of renormalized solutions, and proved the existence and uniqueness results under a very general setting that can even cover the case $\sigma(x,u)=f(x)\sqrt{u}$ with $f\in C^2_b(\mathbb{T}^d)$. 

In this paper, we adopt the entropy approach to study \cref{eq:S-integral}. The notion of entropy solution was proposed for nonlinear PDEs in
1970s to tackle the difficulty that the weak solution may not be unique
(see, for example, \cite{kruvzkov1970first,serre1999systems,andreianov2000l1,dafermos2005hyperbolic}
for scalar conservation laws and \cite{carrillo1999entropy} for degenerate
parabolic equations). The entropy solution discriminates the physical
admissible solution and maintains the uniqueness theoretically.

The entropy formulation has been naturally extended to nonlinear SPDEs.
Many works have been done for first order equations (see, for example,
\cite{kim2003stochastic} for additive noise and \cite{feng2008stochastic,vallet2009stochastic,chen2012nonlinear,bauzet2012cauchy,bauzet2014dirichlet}
for multiplicative noise). For second order equations, most works
considered the Cauchy problem in the whole space or a torus; for instance,
\cite{bauzet2015degenerate} studied the Cauchy problem of parabolic-hyperbolic
SPDEs with the noise term $\sigma(x,u)\mathrm{d}W$ and Lipschitz
functions $\Phi$. Stochastic nonlinear diffusion equations in a torus
are considered in \cite{dareiotis2019entropy} for the noise term $\sigma(x,u)\mathrm{d}W$
and in \cite{dareiotis2020nonlinear} for the conservative noise $\nabla\cdot\sigma(x,u)\circ\mathrm{d}W$.
The obstacle problem with the noise term $\sigma(x,u)\mathrm{d}W$
in a torus is introduced in \cite{liu2021obstacle}.

The results on the Dirichlet problem for stochastic nonlinear diffusion
equations are relatively few. The recent papers \cite{bavnas2020numerical,hensel2021finite}
adopted the variational approach from \cite{ren2007stochastic} with
the monotone condition and affine noise, thus not covering our setting;
\cite{dareiotis2020ergodicity} considered entropy solutions of stochastic
porous medium equations with the noise term $\sigma(x,u)\mathrm{d}W$.
A recent paper \cite{clini2023porous} is quite relevant to our paper,
studying kinetic solutions to the Dirichlet problem for porous medium
equations with nonlinear gradient noise driven by rough path. Thanks
to the entropy approach, our result is built on the same regularity
conditions with the existing work on the Cauchy problem in torus (cf. \cite{dareiotis2020nonlinear}),
and do not require extra technical assumptions like \cite[Condition~(2.4)]{clini2023porous}
that prevents the space characteristics from escaping the domain. 
According to the recent work \cite{fehrman2021well}, it is a very interesting question how to relax the regularity condition on $\sigma$ for the Dirichlet problem.

The strategy of the proof of our main result (\cref{thm:main theorem})
basically follows from \cite{dareiotis2019entropy,dareiotis2020ergodicity,dareiotis2020nonlinear},
combining the method of strong entropy condition (called the $(\star)$-property
in this paper) and Kruzhkov's doubling variables technique (cf. \cite{kruvzkov1970first}).
The notion of strong entropy condition was introduced by \cite{feng2008stochastic}
to tackle the uniqueness issue of stochastic scalar conservation laws.
The strategy can be summarized to two steps:
\begin{enumerate}
\item $L_{1}$-estimates: to derive an estimate for $\mathbb{E}\Vert u(t,\cdot)-\tilde{u}(t,\cdot)\Vert_{L_{1}(D)}$,
providing one of the entropy solutions $u$ and $\tilde{u}$ (with
different initial data) satisfies the $(\star)$-property.
\item Approximation: to construct a sequence of non-degenerate equations
whose solutions have the $(\star)$-property and converge to the entropy
solution of \cref{eq:S-integral}.
\end{enumerate}
Kruzhkov's doubling variables technique plays a key role in the proof
of the $L_{1}$-estimates.

The new difficulty arising in the Dirichlet problem, comparing to
the Cauchy problem (cf. \cite{dareiotis2020nonlinear}), is how to
deal with the boundary integral terms that may emerge when applying
the divergence theorem in the proof of the $L_{1}$-estimates. In the paper
\cite{dareiotis2020ergodicity} where the noise term is $\sigma(x,u)\mathrm{d}W$,
a weighted space with a weight function $w\in H_{0}^{1}$ satisfying
$\Delta w=-1$ is introduced to eliminate the boundary terms. However,
in the case of gradient noise, there will be a new ``trouble'' term
$|(u-\tilde{u})\nabla w|$ appearing in the estimate, which cannot
be dominated by any ``good'' terms like $|u-\tilde{u}|w$. 

There are three key points in our approach to overcome the above difficulty:
i) to expand the set of test functions in the definition of entropy
solution (see \cref{def:entropy-solution}), ii) to refine
the strong entropy condition, and iii) to construct subtly a pair
of the convex function $\eta$ and the test function $\phi$ to avoid
the appearance of boundary terms. Specifically, when applying Kruzhkov's
doubling variables technique to our problem, we have to estimate both
$(\tilde{u}(s,x)-u(t,y))^{+}$ and $(u(t,x)-\tilde{u}(s,y))^{+}$
rather than $|u(t,x)-\tilde{u}(s,y)|$ as in \cite{dareiotis2019entropy,dareiotis2020ergodicity,dareiotis2020nonlinear}.
Inspired by \cite{carrillo1999entropy,bauzet2014dirichlet}, we make
use of a partition of unity and shifted mollifiers to keep the test
function in $C_{c}^{\infty}(D)$ with respect to the variable $y$;
and for the variable $x$, we choose a sequence of smooth convex functions
to approach $\eta(r)=r^{+}$. Those carefully chosen functions along
our modified definition of entropy solution and the refined strong
entropy condition let the boundary terms vanish. 
More specific details are given at the beginning of \Cref{sec:-property-and--esitmate} and Remarks \ref{rem:adjust star-property}, \ref{rem:star property in Lemma of L1} and \ref{rem:different star property}.
It is worth noting that our approach
avoids the involvement of weight functions and considers the $L_{1}$-estimates
in standard Sobolev spaces, and our method may also apply to nonlinear
diffusion equations with the noise term $\sigma(x,u)\mathrm{d}W$
to obtain the $L_{1}$-estimates without weight.

This paper is organized as follows. \Cref{sec:Entropy-formulation} describes the entropy
formulation and presents the main theorem. \Cref{sec:-property-and--esitmate}, which is the
main part of this paper, introduces a refined strong entropy condition
and derives the $L_{1}$-estimates for the difference of two entropy solutions.
\Cref{sec:Approximation} constructs the approximate equations and proves the strong
entropy condition of their solutions as well as their solvability.
\Cref{sec:Existence-and-Uniqueness} completes the proof of the main theorem. Two auxiliary lemmas
are proved in the final section.

We conclude the introduction with some notation. Fix $T>0$. Define
$\Omega_{T}\coloneqq\Omega\times[0,T]$ and $D_{T}\coloneqq[0,T]\times D$.
Define $|D|$ and $\overline{D}$ as the volume and closure of $D$, respectively. $L_{p}$ and $H_{p}^{k}$ are
the usual Lebesgue spaces and Sobolev spaces. 
Denote by $H_{p,0}^{k}$
the closure of $C_{c}^{\infty}$ in $H_{p}^{k}$. When $p=2$, we
simplify the notation by $H^{k}\coloneqq H_{2}^{k}$ and $H_{0}^{k}\coloneqq H_{2,0}^{k}$.
Moreover, if a function space is given on $\Omega$ or $\Omega_{T}$,
we understand it to be defined with respect to $\mathcal{F}_{T}$
and the predictable $\sigma$-field, respectively. Let $E$ be a Banach
space and $U=D$ or $\mathbb{R}$. For a function $f:U\rightarrow E$,
we define
\begin{align*}
[f]_{C^{\alpha}(U;E)} & \coloneqq\sup_{x,y\in U,\ x\neq y}\frac{\Vert f(x)-f(y)\Vert_{E}}{|x-y|^{\alpha}},\quad\alpha\in(0,1],\\
\Vert f\Vert_{C^{\alpha}(U;E)} & \coloneqq[f]_{C^{\alpha}(U;E)}+\sup_{x\in U}\Vert f(x)\Vert_{E}.
\end{align*}
We define a non-negative smooth mollifier $\rho:\mathbb{R}\rightarrow\mathbb{R}$,
such that $\text{supp}\,\rho\subset(0,1)$, $\rho\leq2$ and $\int_{\mathbb{R}}\rho(r)\mathrm{d}r=1$.
For $\delta>0$, we set $\rho_{\delta}(r)\coloneqq\delta^{-1}\rho(\delta^{-1}r)$
as a sequence of mollifiers. 

\section{Entropy formulation and main results\label{sec:Entropy-formulation}}

First of all, we rewrite \cref{eq:S-integral} into an It\^{o} form
(the notation follows from \cite{dareiotis2020nonlinear}):
\begin{gather}
\begin{aligned}\mathrm{d}u(t,x) & =\big[\Delta\Phi(u)+\partial_{x_{i}}\big(a^{ij}(x,u)\partial_{x_{j}}u+b^{i}(x,u)+f^{i}(x,u)\big)+F(x,u)\big]\mathrm{d}t\\
 & \quad+\big(\nabla\cdot\sigma^{k}(x,u)\big)\mathrm{d}W^{k}(t),\quad(t,x)\in(0,T)\times D;\\
u(0,x) & =\xi(x),\quad x\in D;\\
u(t,x) & =0,\quad(t,x)\in[0,T]\times\partial D,
\end{aligned}
\label{eq:ito-integral}
\end{gather}
where $i,j=1,\dots,d$ and
\begin{align*}
a^{ij}(x,r) & =\frac{1}{2}\sigma_{r}^{ik}(x,r)\sigma_{r}^{jk}(x,r),\quad b^{i}(x,r)=\sigma_{r}^{ik}(x,r)\sigma_{x_{j}}^{jk}(x,r),\\
f^{i}(x,r) & =G^{i}(x,r)-\frac{1}{2}b^{i}(x,r).
\end{align*}
We denote by $\Pi(\Phi,\xi)$ the Dirichlet problem \cref{eq:ito-integral}
with given $\Phi$ and $\xi$.

Throughout this paper, we denote 
\[
\mathfrak{a}(r)\coloneqq\sqrt{\Phi^{\prime}(r)};
\]
for a function $g:D\times\mathbb{R}\rightarrow\mathbb{R}$, we use
the notation
\[
\llbracket g\rrbracket(x,r)\coloneqq\int_{0}^{r}g(x,s)\mathrm{d}s,
\]
and drop $x$ in the above if $g$ does not depend on $x\in D$. Our
condition on the nonlinearity $\Phi$ is the same with \cite{dareiotis2019entropy,dareiotis2020ergodicity,dareiotis2020nonlinear}.
\begin{assumption}
\label{assu:=00005CPhi}$\Phi:\mathbb{R}\rightarrow\mathbb{R}$ is
 differentiable, strictly increasing and satisfying $\Phi(0)=0$. With $\mathfrak{a}(r)=\sqrt{\Phi^{\prime}(r)}$,
there exist constants $m>1$ and $K>0$ such that
\begin{equation}
\begin{aligned} & |\mathfrak{a}(0)|\leq K,\quad|\mathfrak{a}^{\prime}(r)|\leq K|r|^{\frac{m-3}{2}}\bm{1}_{r\neq0},\quad\mathfrak{a}(r)\geq K^{-1}\bm{1}_{|r|\geq1},\\
 & |\llbracket\mathfrak{a}\rrbracket(r)-\llbracket\mathfrak{a}\rrbracket(s)|\geq\begin{cases}
K^{-1}|r-s|, & \mbox{if }\ |r|\lor|s|\geq1,\\
K^{-1}|r-s|^{\frac{m+1}{2}}, & \mbox{if }\ |r|\lor|s|<1.
\end{cases}
\end{aligned}
\label{eq:a K}
\end{equation}
\end{assumption}

The following definition of entropy solution is based on the formulation
in \cite{dareiotis2020ergodicity,dareiotis2020nonlinear} with a slight
but significant modification inspired by \cite[Definition 1]{bauzet2014dirichlet}.
Define two sets
\begin{align*}
\mathcal{E} & \coloneqq\{\eta\in C^{2}(\mathbb{R}):\eta^{\prime\prime}\geq0,\ \text{supp}\,\eta^{\prime\prime}\text{ is compact}\},\\
\mathcal{E}_{0} & \coloneqq\{\eta\in\mathcal{E}:\eta^{\prime}(0)=0\}.
\end{align*}

\begin{defn}
\label{def:entropy-solution}An entropy solution of \cref{eq:ito-integral}
is a predictable stochastic process $u:\Omega_{T}\rightarrow L_{1}(D)$
such that
\begin{enumerate}[label=(\roman*)]
\item $u\in L_{m+1}(\Omega_{T};L_{m+1}(D))$;

\item For all $f\in C_{b}(\mathbb{R})$, we have $\llbracket\mathfrak{a}f\rrbracket(u)\in L_{2}(\Omega_{T};H_{0}^{1}(D))$
and
\[
\partial_{x_{i}}\llbracket\mathfrak{a}f\rrbracket(u)=f(u)\partial_{x_{i}}\llbracket\mathfrak{a}\rrbracket(u);
\]

\item For all 
\[
(\eta,\varphi,\varrho)\in\big(\mathcal{E}\times C_{c}^{\infty}[0,T)\times C_{c}^{\infty}(D)\big)\cup\big(\mathcal{E}_{0}\times C_{c}^{\infty}[0,T)\times C^{\infty}(\overline{D})\big)
\]
such that $\phi\coloneqq\varphi\times\varrho\geq0$, we have almost
surely
\begin{align}
 & -\int_{0}^{T}\int_{D}\eta(u)\partial_{t}\phi\mathrm{d}x\mathrm{d}t\label{eq:entropy formulation}\\
 & \leq\int_{D}\eta(\xi)\phi(0)\mathrm{d}x+\int_{0}^{T}\int_{D}\Big(\llbracket\mathfrak{a}^{2}\eta^{\prime}\rrbracket(u)\Delta\phi+\llbracket a^{ij}\eta^{\prime}\rrbracket(x,u)\phi_{x_{i}x_{j}}\Big)\mathrm{d}x\mathrm{d}t\nonumber \\
 & +\int_{0}^{T}\int_{D}\Big(\llbracket a_{x_{j}}^{ij}\eta^{\prime}-f_{r}^{i}\eta^{\prime}\rrbracket(x,u)-\eta^{\prime}(u)b^{i}(x,u)\Big)\phi_{x_{i}}\mathrm{d}x\mathrm{d}t\nonumber \\
 & +\int_{0}^{T}\int_{D}\Big(\eta^{\prime}(u)f_{x_{i}}^{i}(x,u)-\llbracket f_{rx_{i}}^{i}\eta^{\prime}\rrbracket(x,u)+\eta^{\prime}(u)F(x,u)\Big)\phi\mathrm{d}x\mathrm{d}t\nonumber \\
 & +\int_{0}^{T}\int_{D}\Big(\frac{1}{2}\eta^{\prime\prime}(u)\sum_{k=1}^{\infty}|\sigma_{x_{i}}^{ik}(x,u)|^{2}-\eta^{\prime\prime}(u)|\nabla\llbracket\mathfrak{a}\rrbracket(u)|^{2}\Big)\phi\mathrm{d}x\mathrm{d}t\nonumber \\
 & +\int_{0}^{T}\int_{D}\Big(\eta^{\prime}(u)\phi\sigma_{x_{i}}^{ik}(x,u)-\llbracket\sigma_{rx_{i}}^{ik}\eta^{\prime}\rrbracket(x,u)\phi-\llbracket\sigma_{r}^{ik}\eta^{\prime}\rrbracket(x,u)\phi_{x_{i}}\Big)\mathrm{d}x\mathrm{d}W^{k}(t).\nonumber 
\end{align}
\end{enumerate}
\end{defn}

\begin{rem}
Comparing with \cite[Definition 2.4]{dareiotis2020ergodicity}, we
expand the set of test functions $(\eta,\varphi,\varrho)$ in order
to give a more precise characterization of the behavior of solutions
near the boundary. This is a critical point in our proof of the $L_{1}$-estimates.
 Moreover, the Dirichlet boundary condition is satisfied
implicitly according to \cref{def:entropy-solution} (ii)
and (iii).
\end{rem}

The regularity assumption on coefficients coincides with \cite[Assumption 2.3]{dareiotis2020nonlinear}
for the Cauchy problem in a torus.
\begin{assumption}
\label{assu:Coefficients}Let $G^{i}:D\times\mathbb{R}\rightarrow\mathbb{R}$
and $\sigma^{i}:D\times\mathbb{R}\rightarrow l^{2}$ for $i\in\{1,\ldots,d\}$
and $F:D\times\mathbb{R}\rightarrow\mathbb{R}$ are all continuous.
For all $i,l\in\{1,\dots,d\}$, $q\in\{1,2\}$ and all multi-indices
$\gamma\in\mathbb{N}^{d}$ with $q+|\gamma|\leq3$, the derivatives
$\partial_{r}G^{i}$, $\partial_{x_{l}}G^{i}$, $\partial_{rx_{l}}G^{i}$
and $\partial_{r}^{q}\partial_{x}^{\gamma}\sigma^{i}$ exist and are
continuous on $D\times\mathbb{R}$. Moreover, there exist $\bar{\kappa}\in((m\land2)^{-1},1]$,
$\beta\in((2\bar{\kappa})^{-1},1]$, $N_{0}>0$ and $\tilde{\beta}\in(0,1)$
such that for all $i,j,l\in\{1,\dots,d\}$ and $r\in\mathbb{R}$,
we have:
\begin{align}
 & \sup_{r}\Vert\sigma_{r}^{i}(\cdot,r)\Vert_{H_{\infty}^{2}(D;l^{2})}+\Vert\sigma_{x_{i}}^{i}(\cdot,0)\Vert_{C^{\bar{\kappa}}(D,l^{2})}\leq N_{0},\label{eq:sigma x}\\
 & \sup_{x\in D}\Big([\sigma(x,\cdot)]_{C^{\beta}(\mathbb{R},l^{2})}+\Vert\sigma_{rx_{l}}^{i}(x,\cdot)\Vert_{H_{\infty}^{1}(\mathbb{R};l^{2})}\Big)\leq N_{0},\label{eq:sigma r}\\
 & \|\partial_{r}(\sigma_{x_{j}}^{jk}\sigma_{rx_{l}}^{ik})\|_{L_{\infty}}\leq N_{0},\label{eq:interchange}\\
 & \sup_{x\in D}\Big(\Vert G_{r}^{i}(x,\cdot)\Vert_{C^{\beta}(\mathbb{R})}+\Vert\partial_{r}(\sigma_{r}^{ik}\sigma_{x_{j}}^{jk})(x,\cdot)\Vert_{C^{\beta}(\mathbb{R})}\Big)\leq N_{0},\label{eq:f r}\\
 & [G_{x_{l}}^{i}(\cdot,r)]_{C^{\tilde{\beta}}(D)}+[\partial_{x_{l}}(\sigma_{r}^{ik}(\cdot,r)\sigma_{x_{j}}^{jk}(\cdot,r))]_{C^{\tilde{\beta}}(D)}\leq N_{0}(1+|r|),\label{eq:f x}\\
 & \|\partial_{x_{l}}\partial_{r}(\sigma_{r}^{ik}\sigma_{x_{j}}^{jk})\|_{L_{\infty}}+\|G_{x_{l}r}^{i}\|_{L_{\infty}}\leq N_{0},\label{eq:f r xl}\\
 & \sup_{x\in D}[F(x,\cdot)]_{C^{1}(\mathbb{R})}+\sup_{r}\Vert F(\cdot,r)\Vert_{C^{\tilde{\beta}}(D)}\leq N_{0}.\label{eq:F}
\end{align}
\end{assumption}

Due to the physical background, we focus on non-negative solutions
to the concerned problem, for which we need the following natural
condition.
\begin{assumption}
\label{assu:zero condition}The coefficients satisfy
\begin{equation}
\nabla_{x}\cdot G(x,0)+F(x,0)\geq0,\quad\nabla_{x}\cdot\sigma(x,0)=\{0\}_{k=1}^{\infty},\quad x\in D.\label{eq:zero condition}
\end{equation}
\end{assumption}

The main result of this paper is stated as follows.
\begin{thm}
\label{thm:main theorem}Under Assumptions \ref{assu:=00005CPhi}, \ref{assu:Coefficients} and \ref{assu:zero condition},
 for all non-negative initial function $\xi\in L_{m+1}(\Omega,\mathcal{F}_{0};L_{m+1}(D))$,
we have that
\begin{enumerate}[label=(\roman*)]
\item the problem $\Pi(\Phi,\xi)$ has a unique entropy solution
$u$; 

\item $u\geq0$ for almost all $(\omega,t,x)\in\Omega\times D$;

\item if $\tilde{u}$ is the entropy solution to $\Pi(\Phi,\tilde{\xi})$
with $0\leq\tilde{\xi}\in L_{m+1}(\Omega,\mathcal{F}_{0};L_{m+1}(D))$,
then
\begin{equation}
\underset{t\in[0,T]}{\mathrm{ess\,sup}}\,\mathbb{E}\Vert(u(t,\cdot)-\tilde{u}(t,\cdot))^{+}\Vert_{L_{1}(D)}\leq C\mathbb{E}\Vert(\xi-\tilde{\xi})^{+}\Vert_{L_{1}(D)},\label{eq:L1-est}
\end{equation}
where the constant $C$ depends only on $N_{0}$, $K$, $d$, $T$
and $|D|$.
\end{enumerate}
\end{thm}

We conclude this section by proving the non-negativity of entropy solutions
under our assumptions.
\begin{prop}
\label{prop:non-negative}Under the condition of \cref{thm:main theorem},
each entropy solution to $\Pi(\Phi,\xi)$ with $0\leq\xi\in L_{m+1}(\Omega,\mathcal{F}_{0};L_{m+1}(D))$
is non-negative for almost all $(\omega,t,x)\in\Omega_{T}\times D$.
\end{prop}

\begin{proof}
For a small $\delta>0$, we introduce a function $\eta_{\delta}\in C^{2}(\mathbb{R})$
defined by 
\[
\eta_{\delta}(0)=\eta_{\delta}^{\prime}(0)=0,\quad\eta_{\delta}^{\prime\prime}(r)=\rho_{\delta}(r).
\]
Applying the entropy inequality \cref{eq:entropy formulation} with
$\eta(\cdot)=\eta_{\delta}(-\cdot)$ and $\phi$ independent of $x$,
using the non-negativity of $\xi$, we have
\begin{align}
 & -\mathbb{E}\int_{0}^{T}\int_{D}\eta_{\delta}(-u)\partial_{t}\phi\mathrm{d}x\mathrm{d}t\label{eq:entropy for positive}\\
 & \leq\mathbb{E}\int_{0}^{T}\int_{D}\Big(\llbracket f_{rx_{i}}^{i}\eta_{\delta}^{\prime}(-\cdot)\rrbracket(x,u)-\eta_{\delta}^{\prime}(-u)\big(f_{x_{i}}^{i}(x,u)-G_{x_{i}}^{i}(x,0)\big)\Big)\phi\mathrm{d}x\mathrm{d}t\nonumber \\
 & \quad-\mathbb{E}\int_{0}^{T}\int_{D}\eta{}_{\delta}^{\prime}(-u)\big(F(x,u)+G_{x_{i}}^{i}(x,0)\big)\phi\mathrm{d}x\mathrm{d}t\nonumber \\
 & \quad+\mathbb{E}\int_{0}^{T}\int_{D}\Big(\frac{1}{2}\eta_{\delta}^{\prime\prime}(-u)\sum_{k=1}^{\infty}|\sigma_{x_{i}}^{ik}(x,u)|^{2}\phi-\eta_{\delta}^{\prime\prime}(-u)|\nabla\llbracket\mathfrak{a}\rrbracket(u)|^{2}\phi\Big)\mathrm{d}x\mathrm{d}t.\nonumber 
\end{align}
With \cref{assu:zero condition} and \cref{eq:F} in 
\cref{assu:Coefficients}, we have
\begin{equation}
\sup_{x}\big(F(x,r)+G_{x_{i}}^{i}(x,0)\big)\geq\sup_{x}\big(F(x,0)+G_{x_{i}}^{i}(x,0)\big)-C|r|\geq-C|r|\label{eq:growth of F,G}
\end{equation}
for a positive constant $C$. Moreover, using \cref{assu:zero condition},
the definition of $f^{i}$ and \cref{eq:sigma r} and \cref{eq:f r xl}
in \cref{assu:Coefficients}, we have
\begin{equation}
\sup_{x}|\sigma_{x_{i}}^{i}(x,r)|_{l_{2}}+\sup_{x}|f_{x_{i}}^{i}(x,r)-G_{x_{i}}^{i}(x,0)|\leq C|r|.\label{eq:growth of sigma f}
\end{equation}
Combining \cref{eq:entropy for positive}-\cref{eq:growth of sigma f}
with \cref{eq:f r xl} in \cref{assu:Coefficients}, we
have
\begin{align*}
-\mathbb{E}\int_{0}^{T}\int_{D}\eta_{\delta}(-u)\partial_{t}\phi\mathrm{d}x\mathrm{d}t & \leq C\mathbb{E}\int_{0}^{T}\int_{D}(-u)^{+}\phi\mathrm{d}x\mathrm{d}t+C\delta.
\end{align*}
Since $|\eta_{\delta}(r)-r^{+}|\leq\delta$, taking $\delta\rightarrow0^{+}$,
we have
\begin{equation}
-\mathbb{E}\int_{0}^{T}\int_{D}(-u)^{+}\partial_{t}\phi\mathrm{d}x\mathrm{d}t\leq C\mathbb{E}\int_{0}^{T}\int_{D}(-u)^{+}\phi\mathrm{d}x\mathrm{d}t.\label{eq:non-negative with phi}
\end{equation}
Let $0<s<\tau<T$ be Lebesgue points of the function 
\[
t\mapsto\mathbb{E}\int_{D}(-u(t,x))^{+}\mathrm{d}x.
\]
Fix a constant $\gamma\in(0,(\tau-s)\lor(T-\tau))$. We choose a sequence
of functions $\{\phi_{n}\}_{n\in\mathbb{N}}$ satisfying $\phi_{n}\in C_{c}^{\infty}((0,T))$
and $\Vert\phi_{n}\Vert_{L_{\infty}(0,T)}\lor\Vert\partial_{t}\phi_{n}\Vert_{L_{1}(0,T)}\leq1$,
such that
\[
\lim_{n\rightarrow\infty}\Vert\phi_{n}-V_{(\gamma)}\Vert_{H_{0}^{1}(0,T)}=0,
\]
where $V_{(\gamma)}:[0,T]\rightarrow\mathbb{R}$ satisfies $V_{(\gamma)}(0)=0$
and $V_{(\gamma)}^{\prime}=\gamma^{-1}\mathbf{1}_{[s,s+\gamma]}-\gamma^{-1}\mathbf{1}_{[\tau,\tau+\gamma]}$.
Taking $\phi=\phi_{n}$ in \cref{eq:non-negative with phi} and passing to 
the limit $n\rightarrow\infty$, we have
\begin{align*}
  \frac{1}{\gamma}\mathbb{E}\int_{\tau}^{\tau+\gamma}\int_{D}(-u)^{+}\mathrm{d}x\mathrm{d}t
 \leq C\mathbb{E}\int_{0}^{\tau+\gamma}\int_{D}(-u)^{+}\mathrm{d}x\mathrm{d}t+\frac{1}{\gamma}\mathbb{E}\int_{s}^{s+\gamma}\int_{D}(-u)^{+}\mathrm{d}x\mathrm{d}t.
\end{align*}
Let $\gamma\rightarrow0^{+}$, we have
\[
\mathbb{E}\int_{D}(-u(\tau,x))^{+}\mathrm{d}x\leq C\mathbb{E}\int_{0}^{\tau}\int_{D}(-u)^{+}\mathrm{d}x\mathrm{d}t+\mathbb{E}\int_{D}(-u(s,x))^{+}\mathrm{d}x
\]
holds for almost all $s\in(0,\tau)$. Then, for each $\tilde{\gamma}\in(0,\tau)$,
by averaging over $s\in(0,\tilde{\gamma})$, we have
\begin{align*}
\mathbb{E}\int_{D}(-u(\tau,x))^{+}\mathrm{d}x
\leq C\mathbb{E}\int_{0}^{\tau}\int_{D}(-u)^{+}\mathrm{d}x\mathrm{d}t+\frac{1}{\tilde{\gamma}}\mathbb{E}\int_{0}^{\tilde{\gamma}}\int_{D}(-u)^{+}\mathrm{d}x\mathrm{d}s.
\end{align*}
Taking the limit $\tilde{\gamma}\rightarrow0^{+}$ and using 
\cref{lem:initial value}, the non-negativity of $\xi$ and Gronwall's
inequality, we have $u\geq0$ for almost all $(\omega,t,x)\in\Omega_{T}\times D$.
\end{proof}

\section{Strong entropy condition and $L_{1}$-estimates\label{sec:-property-and--esitmate}}

The uniqueness for entropy solutions of stochastic partial differential equations is usually a challenging problem. The seminal work \cite{feng2008stochastic} introduced a notion of strong entropy condition, called the $(\star)$-property in what follows, to deal with the stochastic integral in the $L_1$-estimates
and proved the uniqueness of the strong entropy solution (namely, an entropy solution that satisfies the $(\star)$-property) for the Cauchy
problem of stochastic scalar conservation laws.    
Recently, a series of papers \cite{dareiotis2019entropy,dareiotis2020ergodicity,dareiotis2020nonlinear} improved this technique and proved the uniqueness of entropy solutions to stochastic porous medium equations. 
The basic idea is to estimate the $L_1$-difference between an entropy solution and a strong entropy solution, which leads to the uniqueness of entropy solutions by proving that the entropy solutions constructed from approximation always satisfy the $(\star)$-property. 
The key technique in the proof of the $L_1$-estimates is Kruzhkov's doubling variables method \cite{kruvzkov1970first}, which is a classical method to study the uniqueness problem for deterministic conservation laws. 

From the Cauchy problem to the Dirichlet problem, the new technical difficulties mainly lie in how to handle the boundary terms. 
Considering the noise form $\sigma(x,u)\mathrm{d}W$, the paper
\cite{dareiotis2020ergodicity} introduced a weighted space to overcome the difficulties, where the weight function $w\in H_{0}^{1}$ satisfying
$\Delta w=-1$ is used to remove the boundary terms in the estimate. 
However, in the case of \cref{eq:S-integral}, there will be a new ``trouble'' term
$|(u-\tilde{u})\nabla w|$ appearing in the $L_1$-estimates, which cannot
be dominated by any ``good'' terms like $|u-\tilde{u}|w$.

In order to deal with the boundary terms in our case, we do not introduce the weight function but modify the definition of the $(\star)$-property by subtly choosing the test functions. 
Since the modified definition is highly intricate,
we give a heuristic explanation before the exact formulation.

We begin by addressing the problem of obtaining the $L_1$-difference between entropy solutions $u$ and $\tilde{u}$ of \cref{eq:S-integral}.
Using Kruzhkov's doubling variables technique (see, for example,  \cite{kruvzkov1970first,li2012homogeneous,dareiotis2019entropy,dareiotis2020ergodicity,dareiotis2020nonlinear}), we wish to estimate the term
\begin{equation}\label{eq:formalKruzhkov}
\mathbb{E}\int_{D_T}\int_{D_T}|u(t,x)-\tilde{u}(s,y)|\rho_\theta(s-t)\varphi_\varepsilon(x,y) \mathrm{d}x\mathrm{d}t\mathrm{d}y\mathrm{d}s,
\end{equation}
where $\rho_\theta(\cdot)\coloneqq\theta^{-1}\rho(\theta^{-1}\cdot)$ defined before \Cref{sec:Entropy-formulation} is a time mollifier, and $\varphi_\varepsilon\in C^\infty(\overline{D}\times \overline{D})$ is a spatial mollifier which satisfies $\lim_{\varepsilon\rightarrow 0^+}\varphi_\varepsilon(x,y)=\delta_0(x-y)$ for all $(x,y)\in \overline{D}\times 
\overline{D}$. When $\theta,\varepsilon\rightarrow 0^+$, we have the estimate for $\mathbb{E}\int_{D_T}|u(t,x)-\tilde{u}(t,x)|\mathrm{d}x\mathrm{d}t$.

How to select a suitable spatial mollifier is important but quite subtle: this is standard when $x$ is an interior point (see, e.g., \cite[proof of Proposition 4.2]{dareiotis2020ergodicity}), but much more complicated when $x$ is on the boundary.
Following the idea from \cite{carrillo1999entropy,bauzet2014dirichlet}, we introduce a partition of unity $\sum_{i=1}^N\psi_i\equiv1$ on $\overline{D}$ as the set of localization functions, and then choose corresponding mollifiers $\varrho_{\varepsilon,i}\in C^\infty(\mathbb{R}^d)$ such that $\text{supp}\,\varrho_{\varepsilon,i}(x-\cdot)\subset D$ for all $x\in\text{supp}\,\psi_i$ and sufficiently small $\varepsilon$. 
Then, the spatial mollifier $\varphi_\varepsilon(x,y)\coloneqq\sum^N_{i=1} \psi_i(x)\varrho_{\varepsilon,i}(x-y)$ satisfies
\[
\lim_{\varepsilon\rightarrow 0^+}\varphi_\varepsilon(x,y)=\lim_{\varepsilon\rightarrow 0^+}\sum^N_{i=1} \psi_i(x)\varrho_{\varepsilon,i}(x-y)=\delta_0(x-y),\quad\forall x,y\in \overline{D}.
\]
It is worth noting that this mollifier is asymmetric in the spatial variables, specifically, for all sufficiently small $\varepsilon$, 
\[
\varphi_\varepsilon(x,\cdot)\in C_c^\infty(D),\  \forall x\in\overline{D},\quad  \text{but} \quad \varphi_\varepsilon(\cdot,y)\in C^\infty(\overline{D}),\ \forall y\in\overline{D}.
\]

Consequently, this asymmetry makes it difficult to estimate \cref{eq:formalKruzhkov}. Instead, we respectively estimate both
\begin{align}
&\mathbb{E}\int_{D_T}\int_{D_T}\big(u(t,x)-\tilde{u}(s,y)\big)^+\rho_\theta(s-t)\varphi_\varepsilon(x,y) \mathrm{d}t\mathrm{d}x\mathrm{d}s\mathrm{d}y,\quad \text{and}\label{eq:formalKruzhkovpositivepart}\\
&\mathbb{E}\int_{D_T}\int_{D_T}\big(\tilde{u}(t,x)-u(s,y)\big)^+\rho_\theta(s-t)\varphi_\varepsilon(x,y) \mathrm{d}t\mathrm{d}x\mathrm{d}s\mathrm{d}y.\nonumber
\end{align} 
Actually, when directly applying the methods in \cite{dareiotis2019entropy,dareiotis2020ergodicity,dareiotis2020nonlinear}, the estimate of \cref{eq:formalKruzhkov} contains the following terms
\begin{align}
&\mathbb{E}\int_0^T\int_{D}\int_{D}\partial_{x_{i}x_{j}}\varphi_\varepsilon(x,y)\int_0^{u(t,x)} a^{ij}(x,r)\text{sgn}(r-\tilde{u}(t,y))\mathrm{d}r\mathrm{d}x\mathrm{d}y\mathrm{d}t\label{eq:integrate by part term}\\
&\quad+\mathbb{E}\int_0^T\int_{D}\int_{D}\partial_{x_{i}}\varphi_\varepsilon(x,y)\int_0^{u(t,x)} a_{x_j}^{ij}(x,r)\text{sgn}(r-\tilde{u}(t,y))\mathrm{d}r\mathrm{d}x\mathrm{d}y\mathrm{d}t,\nonumber
\end{align}
which may not be equal to
\begin{align}
&\mathbb{E}\int_0^T\int_{D}\int_{D}\partial_{x_{i}x_{j}}\varphi_\varepsilon(x,y)\int_{\tilde{u}(t,y)}^{u(t,x)} a^{ij}(x,r)\text{sgn}(r-\tilde{u}(t,y))\mathrm{d}r\mathrm{d}x\mathrm{d}y\mathrm{d}t\label{eq:integrate by parts term after}\\
&\quad+\mathbb{E}\int_0^T\int_{D}\int_{D}\partial_{x_{i}}\varphi_\varepsilon(x,y)\int_{\tilde{u}(t,y)}^{u(t,x)} a_{x_j}^{ij}(x,r)\text{sgn}(r-\tilde{u}(t,y))\mathrm{d}r\mathrm{d}x\mathrm{d}y\mathrm{d}t.\nonumber
\end{align}
The reason for this discrepancy is that when formally applying the divergence theorem in $x$ to the first term in \cref{eq:integrate by part term}, the boundary term does not vanish due to the asymmetry of the spatial mollifier, and estimating it when $\varepsilon\rightarrow 0^+$ is challenging due to the potential absence of a trace for $u$ caused by its low regularity. While for the case of \cref{eq:formalKruzhkovpositivepart}, this equivalence is guaranteed by the support of $\eta(\cdot)=(\cdot-\tilde{u}(s,y))^+$.

Note that the time mollifier is also asymmetric and always requires $s>t$.
For this, we have an important observation: we only need the $(\star)$-property of the function at the larger time variable $s$. However, when applying this observation to  \cref{eq:formalKruzhkovpositivepart}, it necessitates the $(\star)$-property of both $u$ and $\tilde{u}$, which is not what we desire. 
Therefore, we turn to estimate
\begin{align}
&\mathbb{E}\int_{D_T}\int_{D_T}\big(u(t,x)-\tilde{u}(s,y)\big)^+\rho_\theta(s-t)\varphi_\varepsilon(x,y) \mathrm{d}t\mathrm{d}x\mathrm{d}s\mathrm{d}y,\quad\text{and}\label{eq:formalKruzhkovpositivepartchangetime}\\
&\mathbb{E}\int_{D_T}\int_{D_T}\big(\tilde{u}(s,x)-u(t,y)\big)^+\rho_\theta(s-t)\varphi_\varepsilon(x,y) \mathrm{d}t\mathrm{d}x\mathrm{d}s\mathrm{d}y,\nonumber
\end{align}
which only use $\tilde{u}$ at time $s$ and thus require the $(\star)$-property of $\tilde{u}$. 
Since these two terms are asymmetric, we adapt the definition of the $(\star)$-property for both formulations in \cref{eq:formalKruzhkovpositivepartchangetime} (see the definition of the test functions at the beginning of the \Cref{sec:star-property}). Fortunately, the entropy solutions constructed from the vanishing viscosity approximation satisfy this modified $(\star)$-property, given a stronger integrability condition on the initial data in $\omega$.

\subsection*{Construction of the spatial mollifier}
Now, we give the specific construction of the spatial mollifier. Define
\[
\text{dist}(A_{1},A_{2})\coloneqq\inf_{x_{1}\in A_{1},\ x_{2}\in A_{2}}|x_{1}-x_{2}|,\quad A_{1},A_{2}\subset\mathbb{R}^{d}.
\]
Fix an open covering of $\partial D$ by balls $\{B^\prime_j\}_{j=1}^{N^\prime}$ which satisfy that $B_{j}^{\prime}\cap\partial D$
is part of a Lipschitz graph for each $j =1,\ldots,N^\prime$.
Choose an open covering of $\overline{D}$ by $B_{0}$ and balls
$\{B_{i}\}_{i=1}^{N}$ satisfying $\text{dist}(B_{0},\partial D)>0$,
and for each $i>0$, there exists $j\in\{1,\ldots,N^\prime\}$ such that $B_{i}\subset B_{j}^{\prime}$ and $\text{dist}(B_{i},\partial B_{j}^{\prime})>0$.

From \cite[Lemma 9.3]{brezis2011functional}, we know that there exist
functions $\psi_{0},\psi_{1},\ldots,\psi_{N}\in C^{\infty}(\mathbb{R}^{d})$
such that $0\leq\psi_{i}\leq1$, $\text{supp}\, \psi_{i}\subseteq B_{i}$
for $i=0,1,\ldots,N$, and 
\[
\sum_{i=0}^{N}\psi_{i}(x)\equiv1,\quad x\in\overline{D}.
\]
Without loss of generality, we assume $\text{dist}(\text{supp}\,\psi_{i},\partial B_{i})>0$ for all $i=0,1,\ldots,N$. Otherwise, we
consider larger open domains, also denoted as $B_{i}$ for convenience, satisfying $\text{dist}(B_{0},\partial D)>0$ 
and $\text{dist}(B_{i},\partial B_{j}^{\prime})>0$ for at least one $j\in\{1,\ldots,N^\prime\}$. 

For $x=(x_{1},x_{2},\ldots,x_{d})\in\mathbb{R}^{d}$, define the function
$\tilde{\varrho}(x)\coloneqq\prod_{i=1}^{d}\rho(x_{i}-1/2)$ and the
mollifier $\tilde{\varrho}_{\varepsilon}(\cdot)\coloneqq\tilde{\varrho}(\cdot/\varepsilon)/(\varepsilon^{d})$.
Similar to \cite[Section 3.2.1]{bauzet2014dirichlet}, for each $i=1,\ldots,N$,
there exist a constant $\varepsilon_{i}$ and a vector $\tilde{\eta}_{i}\in\mathbb{R}^{d}$,
such that the translated sequence of mollifiers $\varrho_{\varepsilon,i}(\cdot)\coloneqq\tilde{\varrho}_{\varepsilon}(\cdot-\tilde{\eta}_{i})$
satisfying $y\mapsto\varrho_{\varepsilon,i}(x-\cdot)\in C_{c}^{\infty}(D)$
for all $(x,\varepsilon)\in(B_{i}\cap\overline{D})\times(0,\varepsilon_{i})$.
The vector $\tilde{\eta}_{i}$ depends only on the local representation
of the boundary of $D$ in $B_{i}^{\prime}$ as the graph of a Lipschitz
function. We also define $\varrho_{\varepsilon,0}(\cdot)\coloneqq\tilde{\varrho}_{\varepsilon}(\cdot)$.
\begin{rem}
\label{rem:support of rho}From the construction of the function $\varrho_{\varepsilon,i}$,
there exists a constant $\tilde{K}$ depending on the maximum norm of $\eta_{i}$,
$i=1,\ldots,N$, such that $\text{supp}\,\varrho_{\varepsilon,i}\subset\{x\in\mathbb{R}^{d}:|x|<\tilde{K}\varepsilon\}$
holds for all $\varepsilon\in(0,\varepsilon_{i})$ and $i=0,1,\ldots,N$.
\end{rem}

\begin{rem}
\label{rem:support of rho y}For each $i=1,\ldots,N$, since $\text{dist}(\text{supp}\, \psi_{i},\partial B_{i})>0$,
there exists a $\bar{\varepsilon}_{i}\in(0,\varepsilon_{i})$ such
that $\text{supp}\, \varrho_{\varepsilon,i}(x-\cdot)\subset B_{i}\cap D$
for all $(x,\varepsilon)\in(\text{supp}\, \psi_{i}\cap\overline{D})\times(0,\bar{\varepsilon}_{i})$.
Moreover, there exists a $\varepsilon_{0}\in(0,1)$ such that $\text{supp}\, \varrho_{\varepsilon,0}(x-\cdot)\subset D$
for all $(x,\varepsilon)\in B_{0}\times(0,{\varepsilon}_{0})$. In the rest of this article, we define 
\[
\bar{\varepsilon}\coloneqq\min\{\varepsilon_{0},\varepsilon_{1},\bar{\varepsilon}_{1},\varepsilon_{2},\bar{\varepsilon}_{2},\ldots,\varepsilon_{N},\bar{\varepsilon}_{N}\}.
\]
\end{rem}

\subsection{$(\star)$-property\label{sec:star-property}}
For $i \in\{1,\ldots,N\}$, define the sets
\begin{align*}
\mathcal{C}^{-}&\coloneqq\{f\in C^{\infty}(\mathbb{R}):f^{\prime}\in C_{c}(\mathbb{R}),\ \text{supp}\, f\subseteq(-\infty,0]\},\\
\mathcal{C}^{+}&\coloneqq\{f\in C^{\infty}(\mathbb{R}):f^{\prime}\in C_{c}(\mathbb{R}),\ \text{supp}\, f\subseteq(0,\infty)\},\\
\Gamma_{B_i}^{-} &\coloneqq\{f\in C^{\infty}(\overline{D}\times\overline{D}):\text{supp}\, f\subset(B_i\cap\overline{D})\times(B_i\cap\overline{D}),\\
&\quad\ \ f(x,\cdot)\in C_{c}^{\infty}(D),\ \forall x\in B_i\cap\overline{D}\},\\
\Gamma_{B_i}^{+} &\coloneqq\{f\in C^{\infty}(\overline{D}\times\overline{D}):\text{supp}\, f\subset(B_i\cap\overline{D})\times(B_i\cap\overline{D}),\\
&\quad\ \ f(\cdot,y)\in C_{c}^{\infty}(D),\ \forall y\in B_i\cap\overline{D}\}.
\end{align*}
Let $(g,\varphi,u,h)\in(\Gamma_{B_i}^{-}\cup\Gamma_{B_i}^{+})\times C_{c}^{\infty}((0,T))\times L_{m+1}(\Omega_{T};L_{m+1}(D))\times(\mathcal{C}^{-}\cup\mathcal{C}^{+})$.
For $\theta>0$, we introduce
\begin{align*}
\phi_{\theta}(t,s,x,y)\coloneqq g(x,y)\rho_{\theta}(s-t)\varphi(\frac{t+s}{2}),
\end{align*}
and
\begin{align*}
H_{\theta}(s,x,z) & \coloneqq\int_{0}^{T}\int_{y}\bigg(h(u(t,y)-z)\sigma_{y_{i}}^{ik}(y,u(t,y))\phi_{\theta}(t,s,x,y)\\
 & \quad-\int_{0}^{u(t,y)}h(r-z)\sigma_{ry_{i}}^{ik}(y,r)\mathrm{d}r\phi_{\theta}(t,s,x,y)\\
 & \quad-\int_{0}^{u(t,y)}h(r-z)\sigma_{r}^{ik}(y,r)\mathrm{d}r\partial_{y_{i}}\phi_{\theta}(t,s,x,y)\bigg)\mathrm{d}W^{k}(t),
\end{align*}
\begin{align*}
\mathcal{E}(u,w,\theta) & \coloneqq-\mathbb{E}\int_{t,s,x,y}\partial_{x_{j}y_{i}}\phi_{\theta}(t,s,x,y)\int_{u}^{w}\int_{\tilde{r}}^{u}h^{\prime}(r-\tilde{r})\sigma_{r}^{ik}(y,r)\sigma_{r}^{jk}(x,\tilde{r})\mathrm{d}r\mathrm{d}\tilde{r}\\
 & \quad-\mathbb{E}\int_{t,s,x,y}\partial_{y_{i}}\phi_{\theta}(t,s,x,y)\int_{u}^{w}\int_{\tilde{r}}^{u}h^{\prime}(r-\tilde{r})\sigma_{r}^{ik}(y,r)\sigma_{rx_{j}}^{jk}(x,\tilde{r})\mathrm{d}r\mathrm{d}\tilde{r}\\
 & \quad+\mathbb{E}\int_{t,s,x,y}\partial_{y_{i}}\phi_{\theta}(t,s,x,y)\int_{w}^{u}h^{\prime}(r-w)\sigma_{r}^{ik}(y,r)\sigma_{x_{j}}^{jk}(x,w)\mathrm{d}r\\
 & \quad-\mathbb{E}\int_{t,s,x,y}\partial_{x_{j}}\phi_{\theta}(t,s,x,y)\int_{u}^{w}\int_{\tilde{r}}^{u}h^{\prime}(r-\tilde{r})\sigma_{ry_{i}}^{ik}(y,r)\sigma_{r}^{jk}(x,\tilde{r})\mathrm{d}r\mathrm{d}\tilde{r}\\
 & \quad-\mathbb{E}\int_{t,s,x,y}\phi_{\theta}(t,s,x,y)\int_{u}^{w}\int_{\tilde{r}}^{u}h^{\prime}(r-\tilde{r})\sigma_{ry_{i}}^{ik}(y,r)\sigma_{rx_{j}}^{jk}(x,\tilde{r})\mathrm{d}r\mathrm{d}\tilde{r}\\
 & \quad+\mathbb{E}\int_{t,s,x,y}\phi_{\theta}(t,s,x,y)\int_{w}^{u}h^{\prime}(r-w)\sigma_{ry_{i}}^{ik}(y,r)\sigma_{x_{j}}^{jk}(x,w)\mathrm{d}r\\
 & \quad+\mathbb{E}\int_{t,s,x,y}\partial_{x_{j}}\phi_{\theta}(t,s,x,y)\int_{u}^{w}h^{\prime}(u-\tilde{r})\sigma_{y_{i}}^{ik}(y,u)\sigma_{r}^{jk}(x,\tilde{r})\mathrm{d}\tilde{r}\\
 & \quad+\mathbb{E}\int_{t,s,x,y}\phi_{\theta}(t,s,x,y)\int_{u}^{w}h^{\prime}(u-\tilde{r})\sigma_{y_{i}}^{ik}(y,u)\sigma_{rx_{j}}^{jk}(x,\tilde{r})\mathrm{d}\tilde{r}\\
 & \quad-\mathbb{E}\int_{t,s,x,y}\phi_{\theta}(t,s,x,y)h^{\prime}(u-w)\sigma_{y_{i}}^{ik}(y,u)\sigma_{x_{j}}^{jk}(x,w),
\end{align*}
where in the integrand $u=u(t,y)$ and $w=w(t,x)$. For simplicity,
here and below we write $\int_{t}$ in place of $\int_{0}^{T}\cdot\mathrm{d}t$
(and similarly for $\int_{s}$), $\int_{x}$ in place of $\int_{D}\cdot\mathrm{d}x$
(and similarly for $\int_{y}$), and $\int_{z}$ in place of $\int_{\mathbb{R}}\cdot\mathrm{d}z$.
However, to avoid confusion, we use the usual notation if the integral
is taken on a different domain or is a stochastic integral. 
\begin{rem}
\label{rem:Htheta format}Since $\text{supp}\,\varphi\subset(0,T)$,
for a sufficiently small $\theta$, the function $H_{\theta}$ has a
similar form as \cite[(3.7)]{dareiotis2019entropy}. The definition of the sets $\Gamma_{B_i}^{+}$ and $\Gamma_{B_i}^{-}$ indicates $\text{supp}\,H_{\theta}(s,\cdot,z)\subset B_i\cap\overline{D}$
for all $(s,z)\in[0,T]\times\mathbb{R}$. Furthermore, there exists
a modification of $H_{\theta}$ which is smooth in $(s,x,z)$ (see
\cite[Exercise 3.15]{kunita1997stochastic}). Throughout this paper,
we will use this smooth version and still denote it by $H_{\theta}$.
\end{rem}

Fix a constant $\mu\coloneqq(3m+5)/(4m+4)$, which is chosen so that $\mu\in((m+3)/(2m+2),1)$.
\begin{defn}
\label{def:star property}We say that a function $w\in L_{m+1}(\Omega_{T};L_{m+1}(D))$
has the $(\star)$-property, if 
\begin{enumerate}[label=(\roman*)]
\item For each $i\in\{0,1,\ldots,N\}$ and all 
\[(g,\varphi,u,h)\in\Gamma_{B_i}^{-}\times C_{c}^{\infty}((0,T))\times L_{m+1}(\Omega_{T};L_{m+1}(D))\times\mathcal{C}^{-}\]
satisfying $u\geq0$ for almost all $(\omega,t,x)\in\Omega_{T}\times D$,
and for all sufficiently small $\theta\in(0,1)$, we have $H_{\theta}(\cdot,\cdot,w(\cdot,\cdot))\in L_{1}(\Omega_{T}\times D)$
and
\[
\mathbb{E}\int_{s,x}H_{\theta}(s,x,w(s,x))\leq C\theta^{1-\mu}+\mathcal{E}(u,w,\theta)
\]
for a constant $C$ independent of $\theta$.

\item For each $i\in\{0,1,\ldots,N\}$ and all \[(g,\varphi,u,h)\in\Gamma_{B_i}^{+}\times C_{c}^{\infty}((0,T))\times L_{m+1}(\Omega_{T};L_{m+1}(D))\times\mathcal{C}^{+}\]
and sufficiently small $\theta\in(0,1)$, we have $H_{\theta}(\cdot,\cdot,w(\cdot,\cdot))\in L_{1}(\Omega_{T}\times D)$
and
\begin{equation*}
\mathbb{E}\int_{s,x}H_{\theta}(s,x,w(s,x))\leq C\theta^{1-\mu}+\mathcal{E}(u,w,\theta)
\end{equation*}
for a constant $C$ independent of $\theta$.
\end{enumerate}
\end{defn}

\begin{rem}\label{rem:adjust star-property}
Differing from the $(\star)$-property in \cite{dareiotis2019entropy,dareiotis2020ergodicity,dareiotis2020nonlinear},
we adjust test functions $g$, $u$ and $h$ to apply the divergence
theorem with the Dirichlet boundary condition (see the proof of
\cref{prop:u_n star property}). \Cref{rem:star property in Lemma of L1}
 provides the justification for our distinct consideration of assertions (i) and (ii) in \cref{def:star property}.
\end{rem}

\begin{rem}
If $g\in \Gamma_{B_i}^{+}$, define $\tilde{g}(x,y)\coloneqq g(y,x)$, then we have $\tilde{g}\in \Gamma_{B_i}^{-}$. Moreover, if we relabel $x\leftrightarrow y$ in assertion (ii) of \cref{def:star property}, it becomes an estimate to
\begin{align*}
&\mathbb{E}\int_{s,y}H_{\theta}(s,y,w(s,y))\\
 &= \mathbb{E}\int_{s,y}\Bigg[\int_{0}^{T}\int_{x}\bigg(h(u(t,x)-z)\sigma_{x_{i}}^{ik}(x,u(t,x))\tilde{g}(x,y)\rho_{\theta}(s-t)\varphi(\frac{t+s}{2})\\
 & \quad-\int_{0}^{u(t,x)}h(r-z)\sigma_{rx_{i}}^{ik}(x,r)\mathrm{d}r\tilde{g}(x,y)\rho_{\theta}(s-t)\varphi(\frac{t+s}{2})\\
 & \quad-\int_{0}^{u(t,x)}h(r-z)\sigma_{r}^{ik}(x,r)\mathrm{d}r\partial_{x_{i}}\tilde{g}(x,y)\rho_{\theta}(s-t)\varphi(\frac{t+s}{2})\bigg)\mathrm{d}W^{k}(t)\Bigg]_{z=w(s,y)}.
\end{align*}
This is used in the proof of \cref{lem:u-=00005Ctildeu} for the case that $u$ has the $(\star)$-property.

\end{rem}

The following lemmas are introduced from \cite{dareiotis2019entropy,dareiotis2020nonlinear},
and the proofs are similar no matter the space is $D$ or $\mathbb{T}^{d}$.
We omit the proofs here.
\begin{lem}
\label{lem:partial H}Under \cref{eq:sigma x}-\cref{eq:sigma r}
in \cref{assu:Coefficients}, for all $\lambda\in((m+3)/(2m+2),1)$,
$\bar{k}\in\mathbb{N}$ and sufficiently small $\theta\in(0,1)$,
we have 
\begin{gather*}
\mathbb{E}\Vert\partial_{z}H_{\theta}\Vert_{L_{\infty}([0,T];H_{m+1}^{\bar{k}}(D\times\mathbb{R}))}^{m+1}\leq C\theta^{-\lambda(m+1)}\mathcal{N}_{m}(u),
\end{gather*}
where
\[
\mathcal{N}_{m}(u)\coloneqq\mathbb{E}\int_{0}^{T}\Big(1+\Vert u(t)\Vert_{L_{\frac{m+1}{2}}(D)}^{m+1}+\Vert u(t)\Vert_{L_{2}(D)}^{m+1}\Big)\mathrm{d}t,
\]
and the constant $C=C(N_{0},k,d,T,\lambda,|D|,m,h,g,\varphi)$ is
independent of $\theta$. In particular, we have
\begin{align*}
&\mathbb{E}\Vert\partial_{z}H_{\theta}\Vert_{L_{\infty}([0,T];H_{m+1}^{\bar{k}}(D\times\mathbb{R}))}^{m+1}\leq C\theta^{-\lambda(m+1)}\Big(1+\Vert u\Vert_{L_{m+1}(D_{T})}^{m+1}\Big).
\end{align*}
\end{lem}

\begin{lem}
\label{lem:appro for Htheta}Let $w\in L_{2}(\Omega_{T}\times D)$.
Then, for all sufficiently small $\theta\in(0,1)$, we have 
\begin{gather*}
\mathbb{E}\int_{s,x}H_{\theta}(s,x,w(s,x))=\lim_{\lambda\rightarrow0}\mathbb{E}\int_{s,x,z}H_{\theta}(s,x,z)\rho_{\lambda}(w(s,x)-z).
\end{gather*}
\end{lem}

\begin{lem}
\label{lem:Limit star property}Let $\{w_{n}\}_{n\in\mathbb{N}}$
be a sequence bounded in $L_{m+1}(\Omega_{T}\times D)$ satisfying
the $(\star)$-property uniformly in $n$, which means the constant
$C$ in \cref{def:star property} (i) and (ii) are independent
of $n$. Suppose that $w_{n}$ converges to a function $w$ for almost
all $(\omega,t,x)\in\Omega_{T}\times D$. Then, $w$ has the $(\star)$-property.
\end{lem}

\subsection{$L_{1}$-estimates\label{sec:L1+-estimate}}

Now, we give the $L_{1}$-estimates of two entropy solutions $u$ and
$\tilde{u}$ using Kruzhkov's doubling variables technique (cf.
\cite{kruvzkov1970first}), which can eliminate the coupling effect
of $u-\tilde{u}$. One key point is to select the appropriate function
$\eta$ such that $\eta(\cdot-z)\in\mathcal{E}_{0}$ for any $z\geq0$
to approximate the positive part function. Using the convolution,
the non-negative $z$ can be replaced by another entropy solution
under suitable conditions.

In this subsection, we fix $i\in\{0,1,\ldots,N\}$. 
For the sake of brevity, we define $B\coloneqq B_{i}$,
$\psi\coloneqq\psi_{i}$ and $\varrho_{\varepsilon}(x-y)\coloneqq\varrho_{\varepsilon,i}(x-y)$ which are introduced in
the definition of the spatial mollifier in \Cref{sec:-property-and--esitmate}. 
For a non-negative $\varphi\in C_{c}^{\infty}((0,T))$ satisfying
$\Vert\varphi\Vert_{L_{\infty}(0,T)}\lor\Vert\partial_{t}\varphi\Vert_{L_{1}(0,T)}\leq1$,
we define the non-negative test functions
\[
\phi_{\varepsilon}(t,x,y)\coloneqq\varrho_{\varepsilon}(x-y)\varphi(t)\psi(x)\mathbf{1}_{\overline{D}}(x),\quad\phi_{\theta,\varepsilon}(t,x,s,y)\coloneqq\rho_{\theta}(s-t)\phi_{\varepsilon}(\frac{s+t}{2},x,y).
\]
\begin{rem}
Compare to the test functions $\phi_{\theta}$ in \Cref{sec:-property-and--esitmate}.
We want to take $g(x,y)=\mathbf{1}_{\overline{D}}(x)\psi(x)\varrho_{\varepsilon}(x-y)$.
In this case, with $\varepsilon\in(0,\bar{\varepsilon})$ (see 
\cref{rem:support of rho y} for the definition of $\bar{\varepsilon}$),
we have $g\in\Gamma_{B}^{-}$ and the function $\tilde{g}(x,y)\coloneqq g(y,x)$ is in $\Gamma_{B}^+$. 

To derive the $L_{1}$-estimates, we need the following lemma.
\end{rem}

\begin{lem}
\label{lem:u-=00005Ctildeu}Let $0\leq\xi,\tilde{\xi}\in L_{m+1}(\Omega,\mathcal{F}_{0};L_{m+1}(D))$.
Suppose that $u$ and $\tilde{u}$ are the entropy solutions to the
Dirichlet problems $\Pi(\Phi,\xi)$ and $\Pi(\tilde{\Phi},\tilde{\xi})$,
respectively. Let Assumptions \ref{assu:=00005CPhi}, \ref{assu:Coefficients} and \ref{assu:zero condition} hold for both $\Phi$ and $\tilde{\Phi}$.
If $\tilde{u}$ or $u$ has the $(\star)$-property, for $\delta\in(0,1)$, $\varepsilon\in(0,\bar{\varepsilon})$, $\lambda\in[0,1]$,
$\alpha\in(0,1\land(m/2))$, and every non-negative $\varphi\in C_{c}^{\infty}((0,T))$
such that 
\[
\Vert\varphi\Vert_{L_{\infty}(0,T)}\lor\Vert\partial_{t}\varphi\Vert_{L_{1}(0,T)}\leq1,
\]
we have 
\begin{align}
&-\mathbb{E}\int_{t,x,y}(\tilde{u}(t,x)-u(t,y))^{+}\varrho_{\varepsilon}(x-y)\partial_{t}\varphi(t)\psi(x)\label{eq:intermitineq-1}\\
 & \leq\mathbb{E}\int_{t,x}\varphi(t)\Delta_{x}\psi(x)(\tilde{\Phi}(\tilde{u}(t,x))-\Phi(u(t,x)))^{+}\nonumber\\
 &\quad +C\varepsilon^{\frac{1}{m+1}}\mathbb{E}\Vert\nabla\llbracket\mathfrak{a}\rrbracket(u)\Vert_{L_{2}(D_{T})}^{2}\nonumber\\
 & \quad+C\mathfrak{C}(\varepsilon,\delta,\lambda,\alpha)\mathbb{E}(1+\Vert u\Vert_{L_{m+1}(D_{T})}^{m+1}+\Vert\tilde{u}\Vert_{L_{m+1}(D_{T})}^{m+1})\nonumber\\
 & \quad+C\varepsilon^{-2}\mathbb{E}\Big(\big\Vert\mathbf{1}_{|u|\geq R_{\lambda}}(1+u)\big\Vert_{L_{m}(D_{T})}^{m}+\big\Vert\mathbf{1}_{|\tilde{u}|\geq R_{\lambda}}(1+\tilde{u})\big\Vert_{L_{m}(D_{T})}^{m}\Big)\nonumber\\
 & \quad+C\mathbb{E}\int_{t,x,y}\mathbf{1}_{B\cap\overline{D}}(x)\varphi(t)\Big(\varepsilon^{2}\sum_{i,j}|\partial_{x_{i}y_{j}}\varrho_{\varepsilon}(x-y)|\nonumber\\
 & \quad+\varepsilon\sum_{i}|\partial_{x_{i}}\varrho_{\varepsilon}(x-y)|+\varrho_{\varepsilon}(x-y)\Big)(\tilde{u}(t,x)-u(t,y))^{+}.\nonumber
\end{align}
where 
\begin{align*}
R_{\lambda} & \coloneqq\sup\big\{ R\in[0,\infty]:|\mathfrak{a}(r)-\tilde{\mathfrak{a}}(r)|\leq\lambda,\ \forall|r|<R\big\},	\\
\mathfrak{C}(\varepsilon,\delta,\lambda,\alpha) & \coloneqq\varepsilon^{-2}\delta^{2\beta}+\delta^{\beta}\varepsilon^{-1}+\varepsilon^{\bar{\kappa}}+\varepsilon^{\tilde{\beta}}+\varepsilon^{\frac{1}{m+1}}+\varepsilon^{2\bar{\kappa}}\delta^{-1}\\
 & \quad+\varepsilon^{-2}\delta^{2\alpha}+\varepsilon^{-2}\lambda^{2}+\varepsilon^{-1}\lambda,
\end{align*}
and the constant $C$ depends only on $N_{0}$, $K$, $d$, $T$, $|D|$ and $\alpha$.
\end{lem}

\begin{rem}
\label{rem:star property in Lemma of L1}Estimate \cref{eq:intermitineq-1} is not symmetric for $u$ and $\tilde{u}$
(See \cref{rem:different star property} for the reason we focus
on such terms) due to the $(\star)$-property of $\tilde{u}$ or $u$ and
the different status of $x$ and $y$ in test function $\phi_{\varepsilon}(t,x,y)$.
Therefore, it is necessary to consider the $(\star)$-property separately
as in \cref{def:star property}. This is a key point of
our proof, which is different from \cite{dareiotis2019entropy,dareiotis2020ergodicity,dareiotis2020nonlinear}.
\end{rem}

\begin{proof}
We first prove the case that $\tilde{u}$ has the $(\star)$-property. With \cref{prop:non-negative},
we have $u,\tilde{u}\geq0$ for almost all $(\omega,t,x)\in\Omega_{T}\times D$.
For each $\delta>0$, define the function $\eta_{\delta}\in C^{2}(\mathbb{R})$
by
\[
\eta_{\delta}(0)=\eta_{\delta}^{\prime}(0)=0,\quad\eta_{\delta}^{\prime\prime}(r)=\rho_{\delta}(r).
\]
Thus, we have
\[
|\eta_{\delta}(r)-r^{+}|\leq\delta,\quad\mathrm{supp}\,\eta_{\delta}^{\prime\prime}\subset[0,\delta],\quad\int_{\mathbb{R}}|\eta_{\delta}^{\prime\prime}(r)|\mathrm{d}r\leq2,\quad|\eta_{\delta}^{\prime\prime}|\leq2\delta^{-1}.
\]

Fix $(z,t,y)\in[0,\infty)\times D_{T}$. Since $\tilde{u}$ is the
entropy solution to $\Pi(\tilde{\Phi},\tilde{\xi})$ and 
\[
\big(\eta_{\delta}(\cdot-z),\rho_{\theta}(\cdot-t)\varphi(\frac{\cdot+t}{2}),\varrho_{\varepsilon}(\cdot-y)\psi(\cdot)\mathbf{1}_{\overline{D}}(\cdot)\big)\in\mathcal{E}_{0}\times C_{c}^{\infty}((0,T))\times C^{\infty}(\overline{D})
\]
for $\varepsilon\in(0,\bar{\varepsilon})$ and a sufficiently small $\theta$,
using the entropy inequality \cref{eq:entropy formulation} of $\tilde{u}$
with $\big(\eta_{\delta}(r-z),\phi_{\theta,\varepsilon}(t,\cdot,\cdot,y)\big)$
instead of $\big(\eta(r),\phi\big)$, we have
\begin{align*}
 & -\int_{s,x}\eta_{\delta}(\tilde{u}-z)\partial_{s}\phi_{\theta,\varepsilon}\\
 & \leq\int_{s,x}\llbracket\tilde{\mathfrak{a}}^{2}\eta_{\delta}^{\prime}(\cdot-z)\rrbracket(\tilde{u})\Delta_{x}\phi_{\theta,\varepsilon}+\int_{s,x}\llbracket a^{ij}\eta_{\delta}^{\prime}(\cdot-z)\rrbracket(x,\tilde{u})\partial_{x_{i}x_{j}}\phi_{\theta,\varepsilon}\\
 & \quad+\int_{s,x}\llbracket a_{x_{j}}^{ij}\eta_{\delta}^{\prime}(\cdot-z)-f_{r}^{i}\eta_{\delta}^{\prime}(\cdot-z)\rrbracket(x,\tilde{u})\partial_{x_{i}}\phi_{\theta,\varepsilon}\\
 & \quad-\int_{s,x}\eta_{\delta}^{\prime}(\tilde{u}-z)b^{i}(x,\tilde{u})\partial_{x_{i}}\phi_{\theta,\varepsilon}+\int_{s,x}\eta_{\delta}^{\prime}(\tilde{u}-z)F(x,\tilde{u})\phi_{\theta,\varepsilon}\\
 & \quad+\int_{s,x}\eta_{\delta}^{\prime}(\tilde{u}-z)f_{x_{i}}^{i}(x,\tilde{u})\phi_{\theta,\varepsilon}-\int_{s,x}\llbracket f_{rx_{i}}^{i}\eta_{\delta}^{\prime}(\cdot-z)\rrbracket(x,\tilde{u})\phi_{\theta,\varepsilon}\\
 & \quad+\int_{s,x}\frac{1}{2}\eta_{\delta}^{\prime\prime}(\tilde{u}-z)\sum_{k=1}^{\infty}|\sigma_{x_{i}}^{ik}(x,\tilde{u})|^{2}\phi_{\theta,\varepsilon}-\int_{s,x}\eta_{\delta}^{\prime\prime}(\tilde{u}-z)|\nabla_{x}\llbracket\tilde{\mathfrak{a}}\rrbracket(\tilde{u})|^{2}\phi_{\theta,\varepsilon}\\
 & \quad+\int_{0}^{T}\int_{x}\Big(\eta_{\delta}^{\prime}(\tilde{u}-z)\phi_{\theta,\varepsilon}\sigma_{x_{i}}^{ik}(x,\tilde{u})-\llbracket\sigma_{rx_{i}}^{ik}\eta_{\delta}^{\prime}(\cdot-z)\rrbracket(x,\tilde{u})\phi_{\theta,\varepsilon}\\
 & \quad-\llbracket\sigma_{r}^{ik}\eta_{\delta}^{\prime}(\cdot-z)\rrbracket(x,\tilde{u})\partial_{x_{i}}\phi_{\theta,\varepsilon}\Big)\mathrm{d}W^{k}(s),
\end{align*}
where $\tilde{u}=\tilde{u}(s,x)$. Notice that all the expressions
are continuous in $(z,t,y)$. We take $z=u(t,y)$ by convolution and
integrate over $(t,y)\in D_{T}$. By taking expectations, we have
\begin{align}
 & -\mathbb{E}\int_{t,x,s,y}\eta_{\delta}(\tilde{u}-u)\partial_{s}\phi_{\theta,\varepsilon}\label{eq:entropy for deltau}\\
 & \leq\mathbb{E}\int_{t,x,s,y}\llbracket\tilde{\mathfrak{a}}^{2}\eta_{\delta}^{\prime}(\cdot-u)\rrbracket(\tilde{u})\Delta_{x}\phi_{\theta,\varepsilon}+\mathbb{E}\int_{t,x,s,y}\llbracket a^{ij}\eta_{\delta}^{\prime}(\cdot-u)\rrbracket(x,\tilde{u})\partial_{x_{i}x_{j}}\phi_{\theta,\varepsilon}\nonumber \\
 & \quad+\mathbb{E}\int_{t,x,s,y}\llbracket a_{x_{j}}^{ij}\eta_{\delta}^{\prime}(\cdot-u)-f_{r}^{i}\eta_{\delta}^{\prime}(\cdot-u)\rrbracket(x,\tilde{u})\partial_{x_{i}}\phi_{\theta,\varepsilon}\nonumber \\
 & \quad-\mathbb{E}\int_{t,x,s,y}\eta_{\delta}^{\prime}(\tilde{u}-u)b^{i}(x,\tilde{u})\partial_{x_{i}}\phi_{\theta,\varepsilon}+\mathbb{E}\int_{t,x,s,y}\eta_{\delta}^{\prime}(\tilde{u}-u)F(x,\tilde{u})\phi_{\theta,\varepsilon}\nonumber \\
 & \quad+\mathbb{E}\int_{t,x,s,y}\eta_{\delta}^{\prime}(\tilde{u}-u)f_{x_{i}}^{i}(x,\tilde{u})\phi_{\theta,\varepsilon}-\mathbb{E}\int_{t,x,s,y}\llbracket f_{rx_{i}}^{i}\eta_{\delta}^{\prime}(\cdot-u)\rrbracket(x,\tilde{u})\phi_{\theta,\varepsilon}\nonumber \\
 & \quad+\mathbb{E}\int_{t,x,s,y}\frac{1}{2}\eta_{\delta}^{\prime\prime}(\tilde{u}-u)\sum_{k=1}^{\infty}|\sigma_{x_{i}}^{ik}(x,\tilde{u})|^{2}\phi_{\theta,\varepsilon}-\mathbb{E}\int_{t,x,s,y}\eta_{\delta}^{\prime\prime}(\tilde{u}-u)|\nabla_{x}\llbracket\tilde{\mathfrak{a}}\rrbracket(\tilde{u})|^{2}\phi_{\theta,\varepsilon}\nonumber \\
 & \quad+\int_{t,y}\mathbb{E}\Bigg[\int_{0}^{T}\int_{x}\Big(\eta_{\delta}^{\prime}(\tilde{u}-z)\phi_{\theta,\varepsilon}\sigma_{x_{i}}^{ik}(x,\tilde{u})-\llbracket\sigma_{rx_{i}}^{ik}\eta_{\delta}^{\prime}(\cdot-z)\rrbracket(x,\tilde{u})\phi_{\theta,\varepsilon}\nonumber \\
 & \quad-\llbracket\sigma_{r}^{ik}\eta_{\delta}^{\prime}(\cdot-z)\rrbracket(x,\tilde{u})\partial_{x_{i}}\phi_{\theta,\varepsilon}\Big)\mathrm{d}W^{k}(s)\Bigg]_{z=u},\nonumber 
\end{align}
where $u=u(t,y)$ and $\tilde{u}=\tilde{u}(s,x)$. 

Similarly, for each $(z,s,x)\in[0,\infty)\times D_{T}$, since
\[
\big(\eta_{\delta}(z-\cdot),\rho_{\theta}(s-\cdot)\varphi(\frac{s+\cdot}{2}),\varrho_{\varepsilon}(x-\cdot)\psi(x)\mathbf{1}_{\overline{D}}(x)\big)\in\mathcal{E}\times C_{c}^{\infty}((0,T))\times C_{c}^{\infty}(D)
\]
for $\varepsilon\in(0,\bar{\varepsilon})$ and a sufficiently small $\theta$,
we apply the entropy inequality of $u$ with $\eta(r)\coloneqq\eta_{\delta}(z-r)$
and $\phi(t,y)\coloneqq\phi_{\theta,\varepsilon}(t,x,s,y)$. After
substituting $z=\tilde{u}(s,x)$ by convolution, integrating over
$(s,x)\in D_{T}$ and taking expectations, we have
\begin{align}
 & -\mathbb{E}\int_{t,x,s,y}\eta_{\delta}(\tilde{u}-u)\partial_{t}\phi_{\theta,\varepsilon}\label{eq:entropy for u}\\
 & \leq-\mathbb{E}\int_{t,x,s,y}\llbracket\mathfrak{a}^{2}\eta_{\delta}^{\prime}(\tilde{u}-\cdot)\rrbracket(u)\Delta_{y}\phi_{\theta,\varepsilon}-\mathbb{E}\int_{t,x,s,y}\llbracket a^{ij}\eta_{\delta}^{\prime}(\tilde{u}-\cdot)\rrbracket(y,u)\partial_{y_{i}y_{j}}\phi_{\theta,\varepsilon}\nonumber \\
 & \quad-\mathbb{E}\int_{t,x,s,y}\llbracket a_{y_{j}}^{ij}\eta_{\delta}^{\prime}(\tilde{u}-\cdot)-f_{r}^{i}\eta_{\delta}^{\prime}(\tilde{u}-\cdot)\rrbracket(y,u)\partial_{y_{i}}\phi_{\theta,\varepsilon}\nonumber \\
 & \quad+\mathbb{E}\int_{t,x,s,y}\eta_{\delta}^{\prime}(\tilde{u}-u)b^{i}(y,u)\partial_{y_{i}}\phi_{\theta,\varepsilon}-\mathbb{E}\int_{t,x,s,y}\eta_{\delta}^{\prime}(\tilde{u}-u)F(y,u)\phi_{\theta,\varepsilon}\nonumber \\
 & \quad-\mathbb{E}\int_{t,x,s,y}\eta_{\delta}^{\prime}(\tilde{u}-u)f_{y_{i}}^{i}(y,u)\phi_{\theta,\varepsilon}+\mathbb{E}\int_{t,x,s,y}\llbracket f_{ry_{i}}^{i}\eta_{\delta}^{\prime}(\tilde{u}-\cdot)\rrbracket(y,u)\phi_{\theta,\varepsilon}\nonumber \\
 & \quad+\mathbb{E}\int_{t,x,s,y}\frac{1}{2}\eta_{\delta}^{\prime\prime}(\tilde{u}-u)\sum_{k=1}^{\infty}|\sigma_{y_{i}}^{ik}(y,u)|^{2}\phi_{\theta,\varepsilon}-\mathbb{E}\int_{t,x,s,y}\eta_{\delta}^{\prime\prime}(\tilde{u}-u)|\nabla_{y}\llbracket\mathfrak{a}\rrbracket(u)|^{2}\phi_{\theta,\varepsilon}\nonumber \\
 & \quad-\int_{s,x}\mathbb{E}\Bigg[\int_{0}^{T}\int_{y}\Big(\eta_{\delta}^{\prime}(z-u)\phi_{\theta,\varepsilon}\sigma_{y_{i}}^{ik}(y,u)-\llbracket\sigma_{ry_{i}}^{ik}\eta_{\delta}^{\prime}(z-\cdot)\rrbracket(y,u)\phi_{\theta,\varepsilon}\nonumber \\
 & \quad-\llbracket\sigma_{r}^{ik}\eta_{\delta}^{\prime}(z-\cdot)\rrbracket(y,u)\partial_{y_{i}}\phi_{\theta,\varepsilon}\Big)\mathrm{d}W^{k}(t)\Bigg]_{z=\tilde{u}},\nonumber 
\end{align}
where $u=u(t,y)$ and $\tilde{u}=\tilde{u}(s,x)$. Note that the stochastic
integral term in \cref{eq:entropy for deltau} is zero due to the
support of $\rho_{\theta}$. Meanwhile, we use the $(\star)$-property
(i) in \cref{def:star property} of $\tilde{u}$ for the last term of \cref{eq:entropy for u}
with $h(r)=-\eta_{\delta}^{\prime}(-r)$ (then $h\in\mathcal{C}^{-}$)
and $g(x,y)=\mathbf{1}_{\overline{D}}(x)\psi(x)\varrho_{\varepsilon}(x-y)$.
Adding inequalities \cref{eq:entropy for deltau}-\cref{eq:entropy for u}
and taking the limit $\theta\rightarrow0^{+}$, we have
\begin{equation}
-\mathbb{E}\int_{t,x,y}\eta_{\delta}(\tilde{u}-u)\partial_{t}\phi_{\varepsilon}\leq I+\mathcal{A}+\mathcal{B}+\mathcal{E},\label{eq:L1_mediem}
\end{equation}
where
\begin{align*}
I&\coloneqq\mathbb{E}\int_{t,x,y}\llbracket\tilde{\mathfrak{a}}^{2}\eta_{\delta}^{\prime}(\cdot-u)\rrbracket(\tilde{u})\Delta_{x}\phi_{\varepsilon}-\mathbb{E}\int_{t,x,y}\llbracket\mathfrak{a}^{2}\eta_{\delta}^{\prime}(\tilde{u}-\cdot)\rrbracket(u)\Delta_{y}\phi_{\varepsilon}\\
&\quad-\mathbb{E}\int_{t,x,y}\eta_{\delta}^{\prime\prime}(\tilde{u}-u)\big(|\nabla_{x}\llbracket\tilde{\mathfrak{a}}\rrbracket(\tilde{u})|^{2}+|\nabla_{y}\llbracket\mathfrak{a}\rrbracket(u)|^{2}\big)\phi_{\varepsilon},\\
\mathcal{A} & \coloneqq\mathbb{E}\int_{t,x,y}\llbracket a^{ij}\eta_{\delta}^{\prime}(\cdot-u)\rrbracket(x,\tilde{u})\partial_{x_{i}x_{j}}\phi_{\varepsilon}-\mathbb{E}\int_{t,x,y}\llbracket a^{ij}\eta_{\delta}^{\prime}(\tilde{u}-\cdot)\rrbracket(y,u)\partial_{y_{i}y_{j}}\phi_{\varepsilon}\\
 & \quad+\mathbb{E}\int_{t,x,y}\llbracket a_{x_{j}}^{ij}\eta_{\delta}^{\prime}(\cdot-u)\rrbracket(x,\tilde{u})\partial_{x_{i}}\phi_{\varepsilon}-\mathbb{E}\int_{t,x,y}\llbracket a_{y_{j}}^{ij}\eta_{\delta}^{\prime}(\tilde{u}-\cdot)\rrbracket(y,u)\partial_{y_{i}}\phi_{\varepsilon}\\
 & \quad-\mathbb{E}\int_{t,x,y}\eta_{\delta}^{\prime}(\tilde{u}-u)b^{i}(x,\tilde{u})\partial_{x_{i}}\phi_{\varepsilon}+\mathbb{E}\int_{t,x,y}\eta_{\delta}^{\prime}(\tilde{u}-u)b^{i}(y,u)\partial_{y_{i}}\phi_{\varepsilon}
\\
&\quad+\mathbb{E}\int_{t,x,y}\frac{1}{2}\eta_{\delta}^{\prime\prime}(\tilde{u}-u)\Big(\sum_{k=1}^{\infty}|\sigma_{x_{i}}^{ik}(x,\tilde{u})|^{2}+\sum_{k=1}^{\infty}|\sigma_{y_{i}}^{ik}(y,u)|^{2}\Big)\phi_{\varepsilon},
\\
\mathcal{B} & \coloneqq-\mathbb{E}\int_{t,x,y}\llbracket f_{r}^{i}\eta_{\delta}^{\prime}(\cdot-u)\rrbracket(x,\tilde{u})\partial_{x_{i}}\phi_{\varepsilon}+\mathbb{E}\int_{t,x,y}\llbracket f_{r}^{i}\eta_{\delta}^{\prime}(\tilde{u}-\cdot)\rrbracket(y,u)\partial_{y_{i}}\phi_{\varepsilon}\\
 & \quad+\mathbb{E}\int_{t,x,y}\eta_{\delta}^{\prime}(\tilde{u}-u)\big(f_{x_{i}}^{i}(x,\tilde{u})-f_{y_{i}}^{i}(y,u)\big)\phi_{\varepsilon}\\
 & \quad-\mathbb{E}\int_{t,x,y}\big(\llbracket f_{rx_{i}}^{i}\eta_{\delta}^{\prime}(\cdot-u)\rrbracket(x,\tilde{u})-\llbracket f_{ry_{i}}^{i}\eta_{\delta}^{\prime}(\tilde{u}-\cdot)\rrbracket(y,u)\big)\phi_{\varepsilon}\\
 &\quad+ \mathbb{E}\int_{t,x,y}\eta_{\delta}^{\prime}(\tilde{u}-u)\big(F(x,\tilde{u})-F(y,u)\big)\phi_{\varepsilon},
\end{align*}
and
\begin{align*}
\mathcal{E}\coloneqq\sum_{i=1}^{9}\mathcal{E}_{i}&\coloneqq-\mathbb{E}\int_{t,x,y}\partial_{x_{j}y_{i}}\phi_{\varepsilon}\int_{u}^{\tilde{u}}\int_{\tilde{r}}^{u}\eta_{\delta}^{\prime\prime}(\tilde{r}-r)\sigma_{r}^{ik}(y,r)\sigma_{r}^{jk}(x,\tilde{r})\mathrm{d}r\mathrm{d}\tilde{r}
\\
&\quad-\mathbb{E}\int_{t,x,y}\partial_{y_{i}}\phi_{\varepsilon}\int_{u}^{\tilde{u}}\int_{\tilde{r}}^{u}\eta_{\delta}^{\prime\prime}(\tilde{r}-r)\sigma_{r}^{ik}(y,r)\sigma_{rx_{j}}^{jk}(x,\tilde{r})\mathrm{d}r\mathrm{d}\tilde{r}
\\
&\quad+\mathbb{E}\int_{t,x,y}\partial_{y_{i}}\phi_{\varepsilon}\int_{\tilde{u}}^{u}\eta_{\delta}^{\prime\prime}(\tilde{u}-r)\sigma_{r}^{ik}(y,r)\sigma_{x_{j}}^{jk}(x,\tilde{u})\mathrm{d}r
\\
&\quad-\mathbb{E}\int_{t,x,y}\partial_{x_{j}}\phi_{\varepsilon}\int_{u}^{\tilde{u}}\int_{\tilde{r}}^{u}\eta_{\delta}^{\prime\prime}(\tilde{r}-r)\sigma_{ry_{i}}^{ik}(y,r)\sigma_{r}^{jk}(x,\tilde{r})\mathrm{d}r\mathrm{d}\tilde{r}
\\
&\quad-\mathbb{E}\int_{t,x,y}\phi_{\varepsilon}\int_{u}^{\tilde{u}}\int_{\tilde{r}}^{u}\eta_{\delta}^{\prime\prime}(\tilde{r}-r)\sigma_{ry_{i}}^{ik}(y,r)\sigma_{rx_{j}}^{jk}(x,\tilde{r})\mathrm{d}r\mathrm{d}\tilde{r}
\\
&\quad+\mathbb{E}\int_{t,x,y}\phi_{\varepsilon}\int_{\tilde{u}}^{u}\eta_{\delta}^{\prime\prime}(\tilde{u}-r)\sigma_{ry_{i}}^{ik}(y,r)\sigma_{x_{j}}^{jk}(x,\tilde{u})\mathrm{d}r
\\
&\quad+\mathbb{E}\int_{t,x,y}\partial_{x_{j}}\phi_{\varepsilon}\int_{u}^{\tilde{u}}\eta_{\delta}^{\prime\prime}(\tilde{r}-u)\sigma_{y_{i}}^{ik}(y,u)\sigma_{r}^{jk}(x,\tilde{r})\mathrm{d}\tilde{r}
\\
&\quad+\mathbb{E}\int_{t,x,y}\phi_{\varepsilon}\int_{u}^{\tilde{u}}\eta_{\delta}^{\prime\prime}(\tilde{r}-u)\sigma_{y_{i}}^{ik}(y,u)\sigma_{rx_{j}}^{jk}(x,\tilde{r})\mathrm{d}\tilde{r}
\\
&\quad-\mathbb{E}\int_{t,x,y}\phi_{\varepsilon}\eta_{\delta}^{\prime\prime}(\tilde{u}-u)\sigma_{y_{i}}^{ik}(y,u)\sigma_{x_{j}}^{jk}(x,\tilde{u}),
\end{align*}
and $u=u(t,y)$ and $\tilde{u}=\tilde{u}(t,x)$ in the integrand.
The term $I$ contains $\mathfrak{a}$, the term $\mathcal{A}$ involves $\sigma$, and the term $\mathcal{B}$ encompasses either $f$ or $F$. Term $\mathcal{E}$ is derived from the $(\star)$-property.
 Now, we estimate these terms. The following estimates are similar
to the proof of \cite[Theorem 4.1]{dareiotis2019entropy} and \cite[Theorem 4.1]{dareiotis2020nonlinear}.
The differences are caused by introducing the function $\psi$ and the
Dirichlet boundary condition. Therefore, we only focus on the application
of the divergence theorem and the term about $\psi$. 
First, we estimate the term $I$. From
\[
\partial_{y_{i}}\int_{0}^{\tilde{u}}\mathfrak{a}^{2}(r)\eta_{\delta}^{\prime}(\tilde{u}-r)\mathrm{d}r=0,
\]
the support of $\eta_{\delta}^{\prime}$ and $\varrho_{\varepsilon}(x-\cdot)\in C_{c}^{\infty}(D)$
for all $(x,\varepsilon)\in(B\cap\overline{D})\times(0,\bar{\varepsilon})$,
we have
\begin{align*}
 I& =-\mathbb{E}\int_{t,x,y}\mathbf{1}_{u\leq\tilde{u}}\partial_{x_{i}y_{i}}\phi_{\varepsilon}\int_{u}^{\tilde{u}}\int_{u}^{\tilde{u}}\mathbf{1}_{r\leq\tilde{r}}\mathfrak{\tilde{\mathfrak{a}}}^{2}(\tilde{r})\eta_{\delta}^{\prime\prime}(\tilde{r}-r)\mathrm{d}\tilde{r}\mathrm{d}r\\
 & \quad-\mathbb{E}\int_{t,x,y}\mathbf{1}_{u\leq\tilde{u}}\partial_{x_{i}y_{i}}\phi_{\varepsilon}\int_{u}^{\tilde{u}}\int_{u}^{\tilde{u}}\mathbf{1}_{r\leq\tilde{r}}\mathfrak{a}^{2}(r)\eta_{\delta}^{\prime\prime}(\tilde{r}-r)\mathrm{d}\tilde{r}\mathrm{d}r\\
 & \quad+\mathbb{E}\int_{t,x,y}\varphi\partial_{x_{i}}(\varrho_{\varepsilon}\partial_{x_{i}}\psi)\int_{u}^{\tilde{u}}\mathfrak{\tilde{\mathfrak{a}}}^{2}(\tilde{r})\eta_{\delta}^{\prime}(\tilde{r}-u)\mathrm{d}\tilde{r}\\
 & \quad+\mathbb{E}\int_{t,x,y}\text{\ensuremath{\varphi}}\partial_{y_{i}}(\varrho_{\varepsilon}\partial_{x_{i}}\psi)\int_{u}^{\tilde{u}}\mathfrak{a}^{2}(r)\eta_{\delta}^{\prime}(\tilde{u}-r)\mathrm{d}r\\
 &\quad-\mathbb{E}\int_{t,x,y}\eta_{\delta}^{\prime\prime}(\tilde{u}-u)\big(|\nabla_{x}\llbracket\tilde{\mathfrak{a}}\rrbracket(\tilde{u})|^{2}+|\nabla_{y}\llbracket\mathfrak{a}\rrbracket(u)|^{2}\big)\phi_{\varepsilon} \eqqcolon\sum_{i=1}^{5}I_{i}.
\end{align*}
Terms $I_{3}$ and $I_{4}$ are arisen from introducing $\psi$.
The term $I_{3}$ can be written as 
\begin{align*}
&I_{3}\\
 & =\mathbb{E}\int_{t,x,y}\varphi(t)\varrho_{\varepsilon}(x-y)\Delta_{x}\psi(x)\int_{u}^{\tilde{u}}\mathfrak{\tilde{\mathfrak{a}}}^{2}(\tilde{r})\mbox{sgn}^{+}(\tilde{r}-u)\mathrm{d}\tilde{r}\\
 & \quad+\mathbb{E}\int_{t,x,y}\varphi(t)\varrho_{\varepsilon}(x-y)\Delta_{x}\psi(x)\int_{u}^{\tilde{u}}\mathbf{1}_{0\leq\tilde{r}-u\leq\delta}\mathfrak{\tilde{\mathfrak{a}}}^{2}(\tilde{r})\big(\eta_{\delta}^{\prime}(\tilde{r}-u)-\mbox{sgn}^{+}(\tilde{r}-u)\big)\mathrm{d}\tilde{r}\\
 & \quad+\mathbb{E}\int_{t,x,y}\varphi(t)\partial_{x_{i}}\varrho_{\varepsilon}(x-y)\partial_{x_{i}}\psi(x)\int_{u}^{\tilde{u}}\mathfrak{\tilde{\mathfrak{a}}}^{2}(\tilde{r})\mbox{sgn}^{+}(\tilde{r}-u)\mathrm{d}\tilde{r}\\
 & \quad+\mathbb{E}\int_{t,x,y}\varphi(t)\partial_{x_{i}}\varrho_{\varepsilon}(x-y)\partial_{x_{i}}\psi(x)\int_{u}^{\tilde{u}}\mathbf{1}_{0\leq\tilde{r}-u\leq\delta}\mathfrak{\tilde{\mathfrak{a}}}^{2}(\tilde{r})\big(\eta_{\delta}^{\prime}(\tilde{r}-u)-\mbox{sgn}^{+}(\tilde{r}-u)\big)\mathrm{d}\tilde{r}\\
 & \eqqcolon\sum_{i=1}^{4}I_{3,i},
\end{align*}
where
\[
\mbox{sgn}^{+}(x)\coloneqq\begin{cases}
1, & x>0;\\
0, & \text{otherwise}.
\end{cases}
\]
Combining the boundness of $\Delta\psi$ and $\varphi$, the definition
of $\varrho_{\varepsilon}$ and \cref{assu:=00005CPhi},
we obtain that
\begin{align*}
|I_{3,2}|+|I_{3,4}| & \lesssim\delta\mathbb{E}\int_{t,x,y}\big(\varrho_{\varepsilon}(x-y)+|\partial_{x_{i}}\varrho_{\varepsilon}(x-y)|\big)\sup_{0\leq\tilde{r}-u\leq\delta}\mathfrak{\tilde{\mathfrak{a}}}^{2}(\tilde{r})\\
 & \lesssim\delta(1+\varepsilon^{-1})\mathbb{E}\big(1+\Vert u\Vert_{L_{m}(D_{T})}^{m}\big).
\end{align*}
Therefore, noticing $\varepsilon<1$, we have
\begin{align*}
I_{3} &\leq I_{3,1}+\mathbb{E}\int_{t,x,y}\varphi(t)\partial_{x_{i}}\varrho_{\varepsilon}(x-y)\partial_{x_{i}}\psi(x)(\tilde{\Phi}(\tilde{u})-\tilde{\Phi}(u))^{+}\\
 & \quad+C\delta\varepsilon^{-1}\mathbb{E}\big(1+\Vert u\Vert_{L_{m}(D_{T})}^{m}\big).
\end{align*}
With the same method, we also have
\begin{align*}
I_{4} & \leq-\mathbb{E}\int_{t,x,y}\text{\ensuremath{\varphi}}(t)\partial_{x_{i}}\varrho_{\varepsilon}(x-y)\partial_{x_{i}}\psi(x)(\Phi(\tilde{u})-\Phi(u))^{+}+C\delta\varepsilon^{-1}\mathbb{E}\big(1+\Vert\tilde{u}\Vert_{L_{m}(D_{T})}^{m}\big).
\end{align*}
Using the triangle inequality, the definition of $\varrho_{\varepsilon}$,
the boundness of $\partial_{x_{i}}\psi$ and the fact
\[
|\Phi(r)-\tilde{\Phi}(r)|\lesssim\lambda|r|^{\frac{m+1}{2}}+\mathbf{1}_{|r|\geq R_{\lambda}}|r|^{m},\quad\forall r\in\mathbb{R},
\]
we have
\begin{align*}
 & \mathbb{E}\int_{t,x,y}\text{\ensuremath{\varphi}}(t)\partial_{x_{i}}\varrho_{\varepsilon}(x-y)\partial_{x_{i}}\psi(x)\big((\tilde{\Phi}(\tilde{u})-\tilde{\Phi}(u))^{+}-(\Phi(\tilde{u})-\Phi(u))^{+}\big)\\
 & \leq\mathbb{E}\int_{t,x,y}\text{\ensuremath{\varphi}}(t)\partial_{x_{i}}\varrho_{\varepsilon}(x-y)\partial_{x_{i}}\psi(x)\big((\Phi(u)-\tilde{\Phi}(u))^{+}+(\tilde{\Phi}(\tilde{u})-\Phi(\tilde{u}))^{+}\big)\\
 & \lesssim\varepsilon^{-1}\lambda\mathbb{E}\big(\Vert u\Vert_{L_{m+1}(D_{T})}^{m+1}+\Vert\tilde{u}\Vert_{L_{m+1}(D_{T})}^{m+1}\big)\\
 & \quad +\varepsilon^{-1}\mathbb{E}\big(\big\Vert\mathbf{1}_{|u|\geq R_{\lambda}}u\big\Vert_{L_{m}(D_{T})}^{m}+\big\Vert\mathbf{1}_{|\tilde{u}|\geq R_{\lambda}}\tilde{u}\big\Vert_{L_{m}(D_{T})}^{m}\big).
\end{align*}
Similarly, for $I_{3,1}$, we have
\begin{align}
 & \mathbb{E}\int_{t,x,y}\varphi(t)\varrho_{\varepsilon}(x-y)\Delta_{x}\psi(x)\big((\tilde{\Phi}(\tilde{u})-\tilde{\Phi}(u))^{+}-(\tilde{\Phi}(\tilde{u})-\Phi(u))^{+}\big)\label{eq:=00005Cphi-Phi diff}\\
 & \lesssim\lambda\Vert u\Vert_{L_{m+1}(D_{T})}^{m+1}+\mathbb{E}\big\Vert\mathbf{1}_{|u|\geq R_{\lambda}}u\big\Vert_{L_{m}(D_{T})}^{m}.\nonumber 
\end{align}
To estimate the new term in \cref{eq:=00005Cphi-Phi diff}, using
 \cref{eq:approximation in Phi} in \cref{lem:approximation},
we have
\begin{align*}
 & \mathbb{E}\int_{t,x,y}\varphi(t)\varrho_{\varepsilon}(x-y)\Delta_{x}\psi(x)(\tilde{\Phi}(\tilde{u}(t,x))-\Phi(u(t,y)))^{+}\\
 & \leq\mathbb{E}\int_{t,x}\varphi(t)\Delta_{x}\psi(x)(\tilde{\Phi}(\tilde{u}(t,x))-\Phi(u(t,x)))^{+}\\
 & \quad+C\varepsilon^{\frac{1}{m+1}}\mathbb{E}(1+\Vert u\Vert_{L_{m+1}(D_{T})}^{m+1}+\Vert\nabla\llbracket\mathfrak{a}\rrbracket(u)\Vert_{L_{2}(D_{T})}^{2}).
\end{align*}
For $I_{5}$, using \cref{def:entropy-solution} (ii) and
\cite[Remark 3.1]{dareiotis2019entropy}, we have
\begin{align}
I_{5} & \leq-2\int_{t,x,y}\eta_{\delta}^{\prime\prime}(\tilde{u}-u)\nabla_{x}\llbracket\tilde{\mathfrak{a}}\rrbracket(\tilde{u})\cdot\nabla_{y}\llbracket\mathfrak{a}\rrbracket(u)\phi_{\varepsilon}\label{eq:I3 x}\\
 & =2\int_{t,x,y}\phi_{\varepsilon}\partial_{x_{i}}\llbracket\tilde{\mathfrak{a}}\rrbracket(\tilde{u})\cdot\partial_{y_{i}}\int_{u}^{\tilde{u}}\eta_{\delta}^{\prime\prime}(\tilde{u}-r)\mathfrak{a}(r)\mathrm{d}r\nonumber \\
 & =2\int_{t,x,y}\partial_{x_{i}y_{i}}\phi_{\varepsilon}\int_{0}^{\tilde{u}}\int_{u}^{\tilde{r}}\eta_{\delta}^{\prime\prime}(\tilde{r}-r)\tilde{\mathfrak{a}}(\tilde{r})\mathfrak{a}(r)\mathrm{d}\tilde{r}\mathrm{d}r\nonumber \\
 & =2\int_{t,x,y}\mathbf{1}_{u\leq\tilde{u}}\partial_{x_{i}y_{i}}\phi_{\varepsilon}\int_{u}^{\tilde{u}}\int_{u}^{\tilde{u}}\eta_{\delta}^{\prime\prime}(\tilde{r}-r)\tilde{\mathfrak{a}}(\tilde{r})\mathfrak{a}(r)\mathrm{d}\tilde{r}\mathrm{d}r.\nonumber 
\end{align}
Then, we have
\begin{align*}
I_{1}+I_{2}+I_{5} & \leq2\int_{t,x,y}\mathbf{1}_{u\leq\tilde{u}}|\partial_{x_{i}y_{i}}\phi_{\varepsilon}|\int_{u}^{\tilde{u}}\int_{u}^{\tilde{u}}\eta_{\delta}^{\prime\prime}(\tilde{r}-r)|\mathfrak{a}(r)-\mathfrak{\tilde{\mathfrak{a}}}(\tilde{r})|^{2}\mathrm{d}\tilde{r}\mathrm{d}r.
\end{align*}
Based on the estimates of $|\mathfrak{a}(r)-\mathfrak{\tilde{\mathfrak{a}}}(\tilde{r})|^{2}$
in the proof of \cite[Theorem 4.1]{dareiotis2019entropy} which using
 \cref{assu:=00005CPhi}, combining the proceeding estimates
and noticing $\varepsilon<1$, we have
\begin{align}
 & I\leq C(\delta\varepsilon^{-1}+\varepsilon^{\frac{1}{m+1}}+\varepsilon^{-2}\lambda^{2}+\varepsilon^{-1}\lambda+\varepsilon^{-2}\delta^{2\alpha})\mathbb{E}\big(1+\Vert u\Vert_{L_{m+1}(D_{T})}^{m+1}+\Vert\tilde{u}\Vert_{L_{m+1}(D_{T})}^{m+1}\big)\label{eq:Phi}\\
 & \quad+C\varepsilon^{-2}\mathbb{E}\Big(\big\Vert\mathbf{1}_{|u|\geq R_{\lambda}}(1+u)\big\Vert_{L_{m}(D_{T})}^{m}+\big\Vert\mathbf{1}_{|\tilde{u}|\geq R_{\lambda}}(1+\tilde{u})\big\Vert_{L_{m}(D_{T})}^{m}\Big)\nonumber \\
 & \quad+\mathbb{E}\int_{t,x}\varphi(t)\Delta_{x}\psi(x)(\tilde{\Phi}(\tilde{u}(t,x))-\Phi(u(t,x)))^{+}+C\varepsilon^{\frac{1}{m+1}}\mathbb{E}\Vert\nabla\llbracket\mathfrak{a}\rrbracket(u)\Vert_{L_{2}(D_{T})}^{2},\nonumber 
\end{align}
where $\alpha\in(0,1\land(m/2))$. Now, we focus on the terms $\mathcal{A}$.
With the fact that $\phi_{\varepsilon}(t,x,\cdot)\in C_{c}^{\infty}(D)$
for any $(t,x,\varepsilon)\in D_{T}\times(0,\bar{\varepsilon})$,
using the divergence theorem in $y$, we have
\[
\mathbb{E}\int_{t,x,y}\partial_{y_{i}y_{j}}\phi_{\varepsilon}\int_{0}^{\tilde{u}}\eta_{\delta}^{\prime}(\tilde{u}-r)a^{ij}(y,r)\mathrm{d}r+\mathbb{E}\int_{t,x,y}\partial_{y_{i}}\phi_{\varepsilon}\int_{0}^{\tilde{u}}\eta_{\delta}^{\prime}(\tilde{u}-r)a_{y_{j}}^{ij}(y,r)\mathrm{d}r=0.
\]
Then, from the support of $\eta_{\delta}^{\prime}$, we have
\begin{align*}
\mathcal{A} & =-\mathbb{E}\int_{t,x,y}\mathbf{1}_{u\leq\tilde{u}}\partial_{x_{i}y_{j}}\phi_{\varepsilon}\int_{u}^{\tilde{u}}\int_{u}^{\tilde{u}}\eta_{\delta}^{\prime\prime}(\tilde{r}-r)\big(a^{ij}(y,r)+a^{ij}(x,\tilde{r})\big)\mathrm{d}\tilde{r}\mathrm{d}r\\
 & \quad-\mathbb{E}\int_{t,x,y}\partial_{x_{i}}\phi_{\varepsilon}\eta_{\delta}^{\prime}(\tilde{u}-u)b^{i}(x,\tilde{u})+\mathbb{E}\int_{t,x,y}\partial_{y_{i}}\phi_{\varepsilon}\eta_{\delta}^{\prime}(\tilde{u}-u)b^{i}(y,u)\\
 & \quad+\mathbb{E}\int_{t,x,y}\partial_{x_{i}}\phi_{\varepsilon}\int_{u}^{\tilde{u}}a_{x_{j}}^{ij}(x,\tilde{r})\eta_{\delta}^{\prime}(\tilde{r}-u)\mathrm{d}\tilde{r}+\mathbb{E}\int_{t,x,y}\partial_{y_{i}}\phi_{\varepsilon}\int_{u}^{\tilde{u}}a_{y_{j}}^{ij}(y,r)\eta_{\delta}^{\prime}(\tilde{u}-r)\mathrm{d}r\\
 & \quad+\mathbb{E}\int_{t,x,y}\mathbf{1}_{u\leq\tilde{u}}\varphi(t)\partial_{x_{i}}(\varrho_{\varepsilon}(x-y)\partial_{x_{j}}\psi(x))\int_{u}^{\tilde{u}}\eta_{\delta}^{\prime}(\tilde{r}-u)a^{ij}(x,\tilde{r})\mathrm{d}\tilde{r}\\
 & \quad+\mathbb{E}\int_{t,x,y}\mathbf{1}_{u\leq\tilde{u}}\varphi(t)\partial_{y_{j}}\varrho_{\varepsilon}(x-y)\partial_{x_{i}}\psi(x)\int_{u}^{\tilde{u}}\int_{u}^{\tilde{u}}\eta_{\delta}^{\prime\prime}(\tilde{r}-r)a^{ij}(y,r)\mathrm{d}\tilde{r}\mathrm{d}r\\
 &\quad+\mathbb{E}\int_{t,x,y}\frac{1}{2}\eta_{\delta}^{\prime\prime}(\tilde{u}-u)\Big(\sum_{k=1}^{\infty}|\sigma_{x_{i}}^{ik}(x,\tilde{u})|^{2}+\sum_{k=1}^{\infty}|\sigma_{y_{i}}^{ik}(y,u)|^{2}\Big)\phi_{\varepsilon}\\
 & \eqqcolon\sum_{i=1}^{6}\mathcal{A}_{i}.
\end{align*}
Using
\cref{eq:sigma x}-\cref{eq:sigma r} in \cref{assu:Coefficients},
we have
\begin{align}
\mathcal{E}_{9}+\mathcal{A}_{6} & =\mathbb{E}\int_{t,x,y}\frac{1}{2}\eta_{\delta}^{\prime\prime}(\tilde{u}-u)\sum_{k=1}^{\infty}\big(\sigma_{x_{j}}^{jk}(x,\tilde{u})-\sigma_{y_{i}}^{ik}(y,u)\big)^{2}\phi_{\varepsilon}\label{eq:E9}\\
 & \lesssim\varepsilon^{2\bar{\kappa}}\delta^{-1}\mathbb{E}(1+\Vert u\Vert_{L_{m+1}(D_{T})}^{m+1})+\delta^{2}.\nonumber 
\end{align}
Next, with \cref{eq:sigma x}-\cref{eq:sigma r} in
\cref{assu:Coefficients}, we have
\begin{align*}
& \mathcal{A}_{4}+\mathcal{A}_{5}\\
 & =\mathbb{E}\int_{t,x,y}\mathbf{1}_{u\leq\tilde{u}}\varphi\varrho_{\varepsilon}\partial_{x_{i}x_{j}}\psi\int_{u}^{\tilde{u}}\eta_{\delta}^{\prime}(\tilde{r}-u)a^{ij}(x,\tilde{r})\mathrm{d}\tilde{r}\nonumber \\
 & \quad+\mathbb{E}\int_{t,x,y}\mathbf{1}_{u\leq\tilde{u}}\varphi(\partial_{x_{j}}\varrho_{\varepsilon}\partial_{x_{i}}\psi)\int_{u}^{\tilde{u}}\int_{u}^{\tilde{u}}\eta_{\delta}^{\prime\prime}(\tilde{r}-r)(a^{ij}(x,\tilde{r})-a^{ij}(y,r))\mathrm{d}\tilde{r}\mathrm{d}r\nonumber \\
 & \lesssim\mathbb{E}\int_{t,x,y}\mathbf{1}_{B\cap\overline{D}}(x)\varphi(t)\Big(\varepsilon\sum_{i}|\partial_{x_{i}}\varrho_{\varepsilon}(x-y)|+\varrho_{\varepsilon}(x-y)\Big)\cdot(\tilde{u}-u)^{+}\nonumber \\
 & \quad+\delta^{\beta}\varepsilon^{-1}\mathbb{E}\Vert\tilde{u}\Vert_{L_{1}(D_{T})}.\nonumber 
\end{align*}
Moreover, if $|r-\tilde{r}|\leq\delta$, from the definition of $a^{ij}$
and \cref{eq:sigma r} in \cref{assu:Coefficients}, we
have
\begin{align*}
 & \partial_{x_{i}y_{j}}\phi_{\varepsilon}\big(a^{ij}(x,\tilde{r})+a^{ij}(y,r)-\sigma_{r}^{ik}(x,\tilde{r})\sigma_{r}^{jk}(y,r)\big)\\
 & =\frac{1}{2}\varphi\psi\partial_{x_{i}y_{j}}\varrho_{\varepsilon}(\sigma_{r}^{ik}(x,\tilde{r})-\sigma_{r}^{ik}(y,r))(\sigma_{r}^{jk}(x,\tilde{r})-\sigma_{r}^{jk}(y,r))\\
 & \quad+\frac{1}{2}\varphi\partial_{y_{j}}\varrho_{\varepsilon}\partial_{x_{i}}\psi\Big(\sigma_{r}^{ik}(x,\tilde{r})(\sigma_{r}^{jk}(x,\tilde{r})-\sigma_{r}^{jk}(y,r))-\sigma_{r}^{jk}(y,r)(\sigma_{r}^{ik}(x,\tilde{r})-\sigma_{r}^{ik}(y,r))\Big)\\
 & \lesssim \mathbf{1}_{B\cap\overline{D}}(x)\varphi(t)\Big[(\varepsilon^{2}+\delta^{2\beta})\sum_{i,j}|\partial_{x_{i}y_{j}}\varrho_{\varepsilon}(x-y)|+(\varepsilon+\delta^{\beta})\sum_{j}|\partial_{y_{j}}\varrho_{\varepsilon}(x-y)|\Big].
\end{align*}
Therefore,
\begin{align}
 & \mathcal{A}_{1}+\mathcal{E}_{1}\label{eq:A}\\
 & \lesssim\mathbb{E}\int_{t,x,y}\mathbf{1}_{B\cap\overline{D}}(x)\varphi(t)\Big(\varepsilon^{2}\sum_{i,j}|\partial_{x_{i}y_{j}}\varrho_{\varepsilon}(x-y)|+\varepsilon\sum_{i}|\partial_{x_{i}}\varrho_{\varepsilon}(x-y)|\Big)(\tilde{u}-u)^{+}\nonumber \\
 & \quad+(\delta^{2\beta}\varepsilon^{-2}+\delta^{\beta}\varepsilon^{-1})\mathbb{E}\Vert\tilde{u}\Vert_{L_{1}(D_{T})}.\nonumber 
\end{align}
Similarly, we have 
\begin{align*}
 & \mathcal{A}_{3}+\mathcal{E}_{2}+\mathcal{E}_{4}\\
 & =\mathbb{E}\int_{t,x,y}\mathbf{1}_{u\leq\tilde{u}}\varphi\partial_{x_{i}}\varrho_{\varepsilon}\psi\int_{u}^{\tilde{u}}\int_{u}^{\tilde{u}}\eta_{\delta}^{\prime\prime}(\tilde{r}-r)(a_{x_{j}}^{ij}(x,\tilde{r})-a_{y_{j}}^{ij}(y,r))\mathrm{d}r\mathrm{d}\tilde{r}\\
 & \quad+\mathbb{E}\int_{t,x,y}\varphi\varrho_{\varepsilon}\partial_{x_{i}}\psi\int_{u}^{\tilde{u}}a_{x_{j}}^{ij}(x,\tilde{r})\eta_{\delta}^{\prime}(\tilde{r}-u)\mathrm{d}\tilde{r}\\
 & \quad+\mathbb{E}\int_{t,x,y}\Big(\mathbf{1}_{u\leq\tilde{u}}\varphi\partial_{x_{i}}\varrho_{\varepsilon}\psi\\
 &\quad \cdot\int_{u}^{\tilde{u}}\int_{u}^{\tilde{u}}\eta_{\delta}^{\prime\prime}(\tilde{r}-r)\big(\sigma_{r}^{ik}(x,\tilde{r})\sigma_{ry_{j}}^{jk}(y,r)-\sigma_{r}^{ik}(y,r)\sigma_{rx_{j}}^{jk}(x,\tilde{r})\big)\mathrm{d}r\mathrm{d}\tilde{r}\Big)\\
 & \quad+\mathbb{E}\int_{t,x,y}\mathbf{1}_{u\leq\tilde{u}}\varphi\varrho_{\varepsilon}\partial_{x_{i}}\psi\int_{u}^{\tilde{u}}\int_{u}^{\tilde{u}}\eta_{\delta}^{\prime\prime}(\tilde{r}-r)\sigma_{ry_{j}}^{jk}(y,r)\sigma_{r}^{ik}(x,\tilde{r})\mathrm{d}r\mathrm{d}\tilde{r},
\end{align*}
with relabeling $i\leftrightarrow j$ in $\mathcal{E}_{4}$. From
\cref{eq:sigma x}-\cref{eq:sigma r} in \cref{assu:Coefficients},
we have
\begin{align}
 & \mathcal{A}_{3}+\mathcal{E}_{2}+\mathcal{E}_{4}\label{eq:B2B4}\\
 & \lesssim\delta^{\beta}\varepsilon^{-1}\mathbb{E}(\Vert u\Vert_{L_{1}(D_{T})}+\Vert\tilde{u}\Vert_{L_{1}(D_{T})})\nonumber \\
 & \quad+\mathbb{E}\int_{t,x,y}\mathbf{1}_{B\cap\overline{D}}(x)\varphi(t)\Big(\varepsilon\sum_{i}|\partial_{x_{i}}\varrho_{\varepsilon}(x-y)|+\varrho_{\varepsilon}(x-y)\Big)(\tilde{u}-u)^{+}.\nonumber 
\end{align}
To estimate $\mathcal{A}_{2}+\mathcal{E}_{3}+\mathcal{E}_{7}$, define 
\begin{align*}
\mathcal{A}_{2} & =-\mathbb{E}\int_{t,x,y}\partial_{x_{i}}\phi_{\varepsilon}\eta_{\delta}^{\prime}(\tilde{u}-u)b^{i}(x,\tilde{u})\\
 & \quad+\mathbb{E}\int_{t,x,y}\partial_{y_{i}}\phi_{\varepsilon}\eta_{\delta}^{\prime}(\tilde{u}-u)b^{i}(y,u)\eqqcolon \mathcal{A}_{2,1}+\mathcal{A}_{2,2}.
\end{align*}
Using the definition of $b^{i}$ and \cref{eq:sigma x}-\cref{eq:sigma r}
in \cref{assu:Coefficients} and relabeling $i\leftrightarrow j$
in $\mathcal{E}_{7}$, we have 
\begin{align*}
 \mathcal{A}_{2,2}+\mathcal{E}_{7}
 & =\mathbb{E}\int_{t,x,y}\varphi\partial_{y_{i}}\varrho_{\varepsilon}\psi\int_{u}^{\tilde{u}}\eta_{\delta}^{\prime\prime}(\tilde{r}-u)\sigma_{y_{j}}^{jk}(y,u)\big(\sigma_{r}^{ik}(y,u)-\sigma_{r}^{ik}(x,u)\big)\mathrm{d}\tilde{r}\\
 & \quad+\mathbb{E}\int_{t,x,y}\varphi\partial_{y_{i}}\varrho_{\varepsilon}\psi\int_{u}^{\tilde{u}}\eta_{\delta}^{\prime\prime}(\tilde{r}-u)\sigma_{y_{j}}^{jk}(y,u)\big(\sigma_{r}^{ik}(x,u)-\sigma_{r}^{ik}(x,\tilde{r})\big)\mathrm{d}\tilde{r}\\
 & \quad+\mathbb{E}\int_{t,x,y}\varphi\varrho_{\varepsilon}\partial_{x_{i}}\psi\int_{u}^{\tilde{u}}\eta_{\delta}^{\prime\prime}(\tilde{r}-u)\sigma_{y_{j}}^{jk}(y,u)\sigma_{r}^{ik}(y,u)\mathrm{d}\tilde{r}\\
 & \quad-\mathbb{E}\int_{t,x,y}\varphi\varrho_{\varepsilon}\partial_{x_{i}}\psi\int_{u}^{\tilde{u}}\eta_{\delta}^{\prime\prime}(\tilde{r}-u)\sigma_{y_{j}}^{jk}(y,u)\big(\sigma_{r}^{ik}(y,u)-\sigma_{r}^{ik}(x,\tilde{r})\big)\mathrm{d}\tilde{r}\\
 & \leq\mathbb{E}\int_{t,x,y}\varphi\partial_{y_{i}}\varrho_{\varepsilon}\psi\eta_{\delta}^{\prime}(\tilde{u}-u)\sigma_{y_{j}}^{jk}(y,u)(y_{l}-x_{l})\int_{0}^{1}\sigma_{rx_{l}}^{ik}(x+\theta(y-x),u)\mathrm{d}\theta\\
 & \quad+C\delta^{\beta}\varepsilon^{-1}(1+\mathbb{E}\Vert u\Vert_{L_{1}(D_{T})})\\
 & \quad+\mathbb{E}\int_{t,x,y}\varphi\varrho_{\varepsilon}\partial_{x_{i}}\psi\eta_{\delta}^{\prime}(\tilde{u}-u)b^{i}(y,u)+C(\delta^{\beta}+\varepsilon)(1+\mathbb{E}\Vert u\Vert_{L_{1}(D_{T})}).
\end{align*}
Similarly, we have
\begin{align*}
 &\mathcal{A}_{2,1}+\mathcal{E}_{3}\\
 & =-\mathbb{E}\int_{t,x,y}\varphi\partial_{y_{i}}\varrho_{\varepsilon}\psi\int_{u}^{\tilde{u}}\eta_{\delta}^{\prime\prime}(\tilde{u}-r)\sigma_{x_{j}}^{jk}(x,\tilde{u})\big(\sigma_{r}^{ik}(y,r)-\sigma_{r}^{ik}(x,\tilde{u})\big)\mathrm{d}r\\
 & \quad-\mathbb{E}\int_{t,x,y}\varphi\varrho_{\varepsilon}\partial_{x_{i}}\psi\int_{u}^{\tilde{u}}\eta_{\delta}^{\prime\prime}(\tilde{u}-r)\sigma_{r}^{ik}(x,\tilde{u})\sigma_{x_{j}}^{jk}(x,\tilde{u})\mathrm{d}r\\
 & \leq-\mathbb{E}\int_{t,x,y}\varphi\partial_{y_{i}}\varrho_{\varepsilon}\psi\eta_{\delta}^{\prime}(\tilde{u}-u)\sigma_{x_{j}}^{jk}(x,\tilde{u})(y_{l}-x_{l})\int_{0}^{1}\sigma_{rx_{l}}^{ik}(x+\theta(y-x),\tilde{u})\mathrm{d}\theta\\
 & \quad+C\delta^{\beta}\varepsilon^{-1}(1+\mathbb{E}\Vert\tilde{u}\Vert_{L_{1}(D_{T})})-\mathbb{E}\int_{t,x,y}\varphi\varrho_{\varepsilon}\partial_{x_{i}}\psi\eta_{\delta}^{\prime}(\tilde{u}-u)b^{i}(x,\tilde{u}).
\end{align*}
From \cref{eq:sigma x}-\cref{eq:interchange} in 
\cref{assu:Coefficients}, we have 
\begin{align*}
 & \Big|\sigma_{y_{j}}^{jk}(y,u)\int_{0}^{1}\sigma_{rx_{l}}^{ik}(x+\theta(y-x),u)\mathrm{d}\theta-\sigma_{x_{j}}^{jk}(x,\tilde{u})\int_{0}^{1}\sigma_{rx_{l}}^{ik}(x+\theta(y-x),\tilde{u})\mathrm{d}\theta\Big|\\
 & \lesssim|x-y|(1+|u|+|\tilde{u}|)+|\sigma_{y_{j}}^{jk}(y,u)\sigma_{ry_{l}}^{ik}(y,u)-\sigma_{x_{j}}^{jk}(x,\tilde{u})\sigma_{rx_{l}}^{ik}(x,\tilde{u})|\\
 & \lesssim(|x-y|+|x-y|^{\bar{\kappa}})(1+|u|+|\tilde{u}|)+|u-\tilde{u}|,
\end{align*}
and 
\[
|b^{i}(y,u)-b^{i}(x,\tilde{u})|\lesssim|u-\tilde{u}|+|x-y|^{\bar{\kappa}}+|x-y|(1+|u|).
\]
Therefore, using the support of $\eta_{\delta}^{\prime}$, we have
\begin{align}
 & \mathcal{A}_{2}+\mathcal{E}_{3}+\mathcal{E}_{7}\label{eq:A21A23}\\
 & \lesssim(\delta^{\beta}\varepsilon^{-1}+\varepsilon^{\bar{\kappa}})\mathbb{E}(1+\Vert u\Vert_{L_{1}(D_{T})}+\Vert\tilde{u}\Vert_{L_{1}(D_{T})})\nonumber \\
 & \quad+\mathbb{E}\int_{t,x,y}\mathbf{1}_{B\cap\overline{D}}(x)\Big(\varepsilon\sum_{i}|\partial_{x_{i}}\varrho_{\varepsilon}(x-y)|+\varrho_{\varepsilon}(x-y)\Big)\varphi(t)(\tilde{u}-u)^{+}.\nonumber 
\end{align}
The remaining terms in \cref{eq:L1_mediem} are $\mathcal{B}$,
$\mathcal{E}_{5}$, $\mathcal{E}_{6}$ and $\mathcal{E}_{8}$. Using
\cref{eq:sigma r} in \cref{assu:Coefficients} and the
support of $\eta_{\delta}^{\prime}$, we have
\begin{align}
\mathcal{\mathcal{E}}_{5}  & \lesssim\mathbb{E}\int_{t,x,y}\mathbf{1}_{B\cap\overline{D}}(x)\varphi(t)\varrho_{\varepsilon}(x-y)(\tilde{u}-u)^{+},\nonumber 
\end{align}
and
\begin{align}
 & \mathcal{E}_{6}+\mathcal{E}_{8}\label{eq:E6E8}\\
 & \leq\mathbb{E}\int_{t,x,y}\phi_{\varepsilon}\int_{\tilde{u}}^{u}\eta_{\delta}^{\prime\prime}(\tilde{u}-r)\sigma_{ry_{i}}^{ik}(y,\tilde{u})\sigma_{x_{j}}^{jk}(x,\tilde{u})\mathrm{d}r+C\delta\mathbb{E}(1+\Vert\tilde{u}\Vert_{L_{1}(D_{T})})\nonumber \\
 & \quad+\mathbb{E}\int_{t,x,y}\phi_{\varepsilon}\int_{u}^{\tilde{u}}\eta_{\delta}^{\prime\prime}(\tilde{r}-u)\sigma_{y_{i}}^{ik}(y,u)\sigma_{rx_{j}}^{jk}(x,u)\mathrm{d}\tilde{r}+C\delta\mathbb{E}(1+\Vert u\Vert_{L_{1}(D_{T})})\nonumber \\
 & \lesssim\mathbb{E}\int_{t,x,y}\phi_{\varepsilon}\eta_{\delta}^{\prime}(\tilde{u}-u)\big|\sigma_{y_{i}}^{ik}(y,u)\sigma_{rx_{j}}^{jk}(x,u)-\sigma_{ry_{i}}^{ik}(y,\tilde{u})\sigma_{x_{j}}^{jk}(x,\tilde{u})\big|\nonumber \\
 & \quad+\delta\mathbb{E}(1+\Vert u\Vert_{L_{1}(D_{T})}+\Vert\tilde{u}\Vert_{L_{1}(D_{T})})\nonumber \\
 & \lesssim(\delta+\varepsilon^{\bar{\kappa}})\mathbb{E}(1+\Vert u\Vert_{L_{1}(D_{T})}+\Vert\tilde{u}\Vert_{L_{1}(D_{T})})+\mathbb{E}\int_{t,x,y}\mathbf{1}_{B\cap\overline{D}}(x)\varphi(t)\varrho_{\varepsilon}(x-y)(\tilde{u}-u)^{+}.\nonumber 
\end{align}
From \cref{eq:f r}-\cref{eq:F} in \cref{assu:Coefficients},
we have
\begin{align}
\mathcal{B} & =\mathbb{E}\int_{t,x,y}\mathbf{1}_{u\leq\tilde{u}}\varphi(t)\partial_{x_{i}}\varrho_{\varepsilon}(x-y)\psi(x)\int_{u}^{\tilde{u}}\int_{u}^{\tilde{u}}\eta_{\delta}^{\prime\prime}(\tilde{r}-r)\big(f_{r}^{i}(y,r)-f_{r}^{i}(x,\tilde{r})\big)\mathrm{d}\tilde{r}\mathrm{d}r\label{eq:L1 f}\\
 & \quad-\mathbb{E}\int_{t,x,y}\varphi(t)\varrho_{\varepsilon}(x-y)\partial_{x_{i}}\psi(x)\int_{u}^{\tilde{u}}f_{r}^{i}(x,\tilde{r})\eta_{\delta}^{\prime}(\tilde{r}-u)\mathrm{d}\tilde{r}\nonumber \\
 & \quad+\mathbb{E}\int_{t,x,y}\eta_{\delta}^{\prime}(\tilde{u}-u)\big(f_{x_{i}}^{i}(x,\tilde{u})-f_{y_{i}}^{i}(y,u)\big)\phi_{\varepsilon}\nonumber \\
 & \quad-\mathbb{E}\int_{t,x,y}\phi_{\varepsilon}\Big(\int_{u}^{\tilde{u}}f_{rx_{i}}^{i}(x,\tilde{r})\eta_{\delta}^{\prime}(\tilde{r}-u)\mathrm{d}\tilde{r}+\int_{u}^{\tilde{u}}f_{ry_{i}}^{i}(y,r)\eta_{\delta}^{\prime}(\tilde{u}-r)\mathrm{d}r\Big)\nonumber \\
 &\quad+ \mathbb{E}\int_{t,x,y}\eta_{\delta}^{\prime}(\tilde{u}-u)\big(F(x,\tilde{u})-F(y,u)\big)\phi_{\varepsilon}\nonumber \\
 & \lesssim(\varepsilon^{\tilde{\beta}}+\delta^{\beta}\varepsilon^{-1})\mathbb{E}(1+\Vert u\Vert_{L_{1}(D_{T})}+\Vert\tilde{u}\Vert_{L_{1}(D_{T})})\nonumber \\
 & \quad+\mathbb{E}\int_{t,x,y}\mathbf{1}_{B\cap\overline{D}}(x)\varphi(t)\Big(\varepsilon|\partial_{x_{i}}\varrho_{\varepsilon}(x-y)|+\varrho_{\varepsilon}(x-y)\Big)(\tilde{u}-u)^{+}.\nonumber
\end{align}
Combining \cref{eq:L1_mediem} with $|\eta_{\delta}(r)-r^{+}|\leq\delta$,
 and \cref{eq:Phi}-\cref{eq:L1 f}, one has \cref{eq:intermitineq-1}.

For the case $u$ has the $(\star)$-property, the proof is a similar procedure but
using a different assertion of the $(\star)$-property. Specifically, for each $(z,t,x)\in[0,\infty)\times D_{T}$,
since
\[
\big(\eta_{\delta}(z-\cdot),\rho_{\theta}(\cdot-t)\varphi(\frac{\cdot+t}{2}),\varrho_{\varepsilon}(x-\cdot)\psi(x)\mathbf{1}_{\overline{D}}(x)\big)\in\mathcal{E}\times C_{c}((0,T))\times C_{c}^{\infty}(D)
\]
for $\varepsilon\in(0,\bar{\varepsilon})$ and a sufficiently small $\theta$,
we apply the entropy inequality \cref{eq:entropy formulation} of
$u(s,y)$ with $\big(\eta_{\delta}(z-r),\phi_{\theta,\varepsilon}(t,x,\cdot,\cdot)\big)$
instead of $\big(\eta(r),\phi\big)$. Taking $z=\tilde{u}(t,x)$ by convolution,
integrating over $(t,x)\in D_{T}$ and taking expectations, we 
acquire an estimate of $-\mathbb{E}\int_{t,x,s,y}\eta_{\delta}(\tilde{u}(t,x)-u(s,y))\partial_{t}\phi_{\theta,\varepsilon}$. 

Similarly, fix $(z,s,y)\in[0,\infty)\times D_{T}$. Since
\[
\big(\eta_{\delta}(\cdot-z),\rho_{\theta}(s-\cdot)\varphi(\frac{s+\cdot}{2}),\varrho_{\varepsilon}(\cdot-y)\psi(\cdot)\mathbf{1}_{\overline{D}}(\cdot)\big)\in\mathcal{E}_{0}\times C_{c}((0,T))\times C^{\infty}(\overline{D})
\]
for all sufficiently small $\theta$, we use \cref{eq:entropy formulation}
of $\tilde{u}(t,x)$ with $\big(\eta_{\delta}(r-z),\phi_{\theta,\varepsilon}(\cdot,\cdot,s,y)\big)$
instead of $\big(\eta(r),\phi\big)$. Then, taking $z=u(s,y)$
by convolution, integrating over $(s,y)\in D_{T}$ and taking expectations,
we obtain an estimate of $-\mathbb{E}\int_{t,x,s,y}\eta_{\delta}(\tilde{u}(t,x)-u(s,y))\partial_{s}\phi_{\theta,\varepsilon}$. 

Note that $\phi_{\theta,\varepsilon}$ only acts on the set $\{s,t\in[0,T]:s>t\}$,
the stochastic integral term in the estimate of $-\mathbb{E}\int_{t,x,s,y}\eta_{\delta}(\tilde{u}(t,x)-u(s,y))\partial_{t}\phi_{\theta,\varepsilon}$
vanishes, while the one in the estimate of $-\mathbb{E}\int_{t,x,s,y}\eta_{\delta}(\tilde{u}(t,x)-u(s,y))\partial_{s}\phi_{\theta,\varepsilon}$
may not be zeros. Therefore, we apply assertion (ii) of the $(\star)$-property of $u$ with $h(r)=\eta_{\delta}^{\prime}(r)$ (then $h\in\mathcal{C}^{+}$)
, $g(x,y)=\mathbf{1}_{\overline{D}}(y)\psi(y)\varrho_{\varepsilon}(y-x)$ (then $g\in \Gamma_{B}^+$) and relabel $x\leftrightarrow y$ in the integral,
then the stochastic integral terms is controlled by $C\theta^{1-\mu}+\mathcal{E}(\tilde{u},u,\theta)$, in the integrand of which $u=u(t,y)$ and $\tilde{u}=\tilde{u}(t,x)$.
After combining two estimates and taking the limit $\theta\rightarrow0^{+}$,
there is only one time variable in the integrand.
The remaining terms can be estimated as above.
Hence, this lemma is proved.
\end{proof}
\begin{rem}
In the whole proof of \cref{lem:u-=00005Ctildeu}, the divergence
theorem is applied mostly in $y$ with $\varrho_{\varepsilon}(x-\cdot)\in C_{c}^{\infty}(D)$
for all $(x,\varepsilon)\in(B\cap\overline{D})\times(0,\bar{\varepsilon})$,
except \cref{eq:I3 x}, in which the divergence theorem in $x$ is
used with the zero boundary condition of $\tilde{u}$. 
\end{rem}

\begin{thm}
\label{thm:the first part}($L_{1}$-estimates)Let $0\leq\xi,\tilde{\xi}\in L_{m+1}(\Omega,\mathcal{F}_{0};L_{m+1}(D))$.
Suppose $u$ and $\tilde{u}$ are the entropy solutions to the Dirichlet
problems $\Pi(\Phi,\xi)$ and $\Pi(\tilde{\Phi},\tilde{\xi})$, respectively.
Let Assumptions \ref{assu:=00005CPhi}, \ref{assu:Coefficients} and \ref{assu:zero condition} hold for both $\Phi$ and $\tilde{\Phi}$.
If $u$ or $\tilde{u}$ has the $(\star)$-property, then,
\begin{enumerate}[label=(\roman*)]
\item if furthermore $\Phi=\tilde{\Phi}$, then
\begin{equation}
\underset{t\in[0,T]}{\mathrm{ess\,sup}}\,\mathbb{E}\Vert(\tilde{u}(t,\cdot)-u(t,\cdot))^{+}\Vert_{L_{1}(D)}\leq C\mathbb{E}\Vert(\tilde{\xi}-\xi)^{+}\Vert_{L_{1}(D)},\label{eq:The first part of esssup}
\end{equation}
where the constant $C$ depends only on $N_{0}$, $K$, $d$, $T$
and $|D|$.
\item for all $\delta\in(0,1)$, $\varepsilon\in(0,\bar{\varepsilon})$,
$\lambda\in[0,1]$ and $\alpha\in(0,1\land(m/2))$, we have
\begin{align}
 & \mathbb{E}\int_{0}^{T}\int_{x}(\tilde{u}(\tau,x)-u(\tau,x))^{+}\mathrm{d}\tau\label{eq:the first part estimation}\\
 & \leq C\mathbb{E}\int_{x}(\tilde{\xi}(x)-\xi(x))^{+}+C\sup_{|h|\leq2\varepsilon}\mathbb{E}\big\Vert\xi(\cdot)-\bar{\xi}(\cdot+h)\big\Vert_{L_{1}(D)}\nonumber \\
 & \quad+C\sup_{|h|\leq2\varepsilon}\mathbb{E}\Big\Vert\tilde{\xi}(\cdot)-\bar{\tilde{\xi}}(\cdot+h)\Big\Vert_{L_{1}(D)}\nonumber \\
 & \quad+C\varepsilon^{\frac{1}{m+1}}\mathbb{E}\big(\Vert\nabla\llbracket\mathfrak{a}\rrbracket(u)\Vert_{L_{2}(D_{T})}^{2}+\Vert\nabla\llbracket\tilde{\mathfrak{a}}\rrbracket(\tilde{u})\Vert_{L_{2}(D_{T})}^{2}\big)\nonumber \\
 & \quad+C\mathfrak{C}(\varepsilon,\delta,\lambda,\alpha)\mathbb{E}(1+\Vert u\Vert_{L_{m+1}(D_{T})}^{m+1}+\Vert\tilde{u}\Vert_{L_{m+1}(D_{T})}^{m+1})\nonumber \\
 & \quad+C\varepsilon^{-2}\mathbb{E}\Big(\big\Vert\mathbf{1}_{|u|\geq R_{\lambda}}(1+u)\big\Vert_{L_{m}(D_{T})}^{m}+\big\Vert\mathbf{1}_{|\tilde{u}|\geq R_{\lambda}}(1+\tilde{u})\big\Vert_{L_{m}(D_{T})}^{m}\Big),\nonumber 
\end{align}
where $\mathfrak{C}(\varepsilon,\delta,\lambda,\alpha)$ and $R_{\lambda}$
are introduced in \cref{lem:u-=00005Ctildeu}, 
\begin{align}
\bar{\xi}(x) & \coloneqq\begin{cases}
\xi(x), & x\in D;\\
0, & x\notin D,
\end{cases}\label{eq:xi bar}
\end{align}
and the constant $C$ depends only on $N_{0}$, $K$, $d$, $T$ and
$|D|$.
\end{enumerate}
\end{thm}

\begin{rem}
\label{rem:different star property}Since the positive part function
is not an even function, it requires us to prove estimates with the $(\star)$-property
of different entropy solutions $u$ or $\tilde{u}$. Using this fact,
we obtain the $L_{1}$-estimates with only one of the entropy solutions
has the $(\star)$-property, which is a key point in proving the uniqueness
in the proof of \cref{thm:main theorem}.
\end{rem}

\begin{proof}
We first assume that $\tilde{u}$ has the $(\star)$-property. Let
$0<s<\tau<T$ be Lebesgue points of the function
\[
t\rightarrow\mathbb{E}\int_{x,y}(\tilde{u}(t,x)-u(t,y))^{+}\psi(x)\varrho_{\varepsilon}(x-y),
\]
and fix a constant $\gamma\in(0,(\tau-s)\lor(T-\tau))$. Choose a
sequence of functions $\{\phi_{n}\}_{n\in\mathbb{N}}$ and the limit
$V_{(\gamma)}$ as in the proof of \cref{prop:non-negative},
take $\varphi=\phi_{n}$ in \cref{eq:intermitineq-1} (using the $(\star)$-property of $\tilde{u}$) and pass to the limit $n\rightarrow\infty$. Then, taking $\gamma\rightarrow0^{+}$
and using \cref{eq:the uniform estimates of u_n and Phi_n}, we have
\begin{align*}
 & -\mathbb{E}\int_{x,y}(\tilde{u}(s,x)-u(s,y))^{+}\varrho_{\varepsilon}(x-y)\psi(x)\\
 & +\mathbb{E}\int_{x,y}(\tilde{u}(\tau,x)-u(\tau,y))^{+}\varrho_{\varepsilon}(x-y)\psi(x)\\
 & \leq\mathbb{E}\int_{s}^{\tau}\int_{x}\Delta_{x}\psi(x)(\tilde{\Phi}(\tilde{u}(t,x))-\Phi(u(t,x)))^{+}\mathrm{d}t+M+C\mathbb{E}\int_{0}^{\tau}\int_{x,y}\mathbf{1}_{B\cap\overline{D}}(x)\\
 & \quad\cdot\Big(\varepsilon^{2}\sum_{i,j}|\partial_{x_{i}y_{j}}\varrho_{\varepsilon}(x-y)|+\varepsilon\sum_{i}|\partial_{x_{i}}\varrho_{\varepsilon}(x-y)|+\varrho_{\varepsilon}(x-y)\Big)(\tilde{u}(t,x)-u(t,y))^{+}\mathrm{d}t
\end{align*}
holds for almost all $s\in(0,\tau)$, where
\begin{align*}
M &\coloneqq C\varepsilon^{\frac{1}{m+1}}\mathbb{E}\Vert\nabla\llbracket\mathfrak{a}\rrbracket(u)\Vert_{L_{2}(D_{T})}^{2}
+C\mathfrak{C}(\varepsilon,\delta,\lambda,\alpha)\mathbb{E}(1+\Vert u\Vert_{L_{m+1}(D_{T})}^{m+1}+\Vert\tilde{u}\Vert_{L_{m+1}(D_{T})}^{m+1})\\
 & \quad+C\varepsilon^{-2}\mathbb{E}\Big(\big\Vert\mathbf{1}_{|u|\geq R_{\lambda}}(1+u)\big\Vert_{L_{m}(D_{T})}^{m}+\big\Vert\mathbf{1}_{|\tilde{u}|\geq R_{\lambda}}(1+\tilde{u})\big\Vert_{L_{m}(D_{T})}^{m}\Big).
\end{align*}
Then, for $\tilde{\gamma}\in(0,\tau)$, by averaging over $s\in(0,\tilde{\gamma})$,
setting $\tilde{\gamma}\rightarrow0^{+}$ and using 
\cref{lem:initial value}, we have
\begin{align}
 & \mathbb{E}\int_{x,y}(\tilde{u}(\tau,x)-u(\tau,y))^{+}\varrho_{\varepsilon}(x-y)\psi(x)\label{eq:L1 0-tau}\\
 & \leq\mathbb{E}\int_{x,y}(\tilde{\xi}(x)-\xi(y))^{+}\varrho_{\varepsilon}(x-y)\psi(x)\nonumber \\
 & \quad+\mathbb{E}\int_{0}^{\tau}\int_{x}\Delta_{x}\psi(x)(\tilde{\Phi}(\tilde{u}(t,x))-\Phi(u(t,x)))^{+}\mathrm{d}t+M+C\mathbb{E}\int_{0}^{\tau}\int_{x,y}\mathbf{1}_{B\cap\overline{D}}(x)\nonumber \\
 & \quad\cdot\Big(\varepsilon^{2}\sum_{i,j}|\partial_{x_{i}y_{j}}\varrho_{\varepsilon}(x-y)|+\varepsilon\sum_{i}|\partial_{x_{i}}\varrho_{\varepsilon}(x-y)|+\varrho_{\varepsilon}(x-y)\Big)(\tilde{u}(t,x)-u(t,y))^{+}\mathrm{d}t.\nonumber 
\end{align}
To prove \cref{eq:The first part of esssup}, taking $\lambda=0$
and $R_{\lambda}=\infty$, we have
\[
\mathbb{E}\Big(\big\Vert\mathbf{1}_{|u|\geq R_{\lambda}}(1+u)\big\Vert_{L_{m}(D_{T})}^{m}+\big\Vert\mathbf{1}_{|\tilde{u}|\geq R_{\lambda}}(1+\tilde{u})\big\Vert_{L_{m}(D_{T})}^{m}\Big)=0,
\]
and $\mathfrak{C}(\varepsilon,\delta,\lambda,\alpha)$ becomes
\[
\mathfrak{C}(\varepsilon,\delta,\alpha)\coloneqq\varepsilon^{-2}\delta^{2\beta}+\delta^{\beta}\varepsilon^{-1}+\varepsilon^{\bar{\kappa}}+\varepsilon^{\tilde{\beta}}+\varepsilon^{1/(m+1)}+\varepsilon^{2\bar{\kappa}}\delta^{-1}+\varepsilon^{-2}\delta^{2\alpha}.
\]
Since $\beta\in((2\bar{\kappa})^{-1},1]$, we can choose $\vartheta\in((m\land2)^{-1}\vee(2\beta)^{-1},\bar{\kappa})$
and $\alpha\in((2\vartheta)^{-1},1\land(m/2))$. Let $\delta=\varepsilon^{2\vartheta}$
and $\varepsilon\rightarrow0^+$, we have $\mathfrak{C}(\varepsilon,\delta,\alpha)\rightarrow0^+$.
Notice that $\varepsilon|\partial_{x_{i}}\varrho_{\varepsilon}|$
and $\varepsilon^{2}|\partial_{x_{i}x_{j}}\varrho_{\varepsilon}|$
are all approximations of the identity up to a constant. Adding over
$\psi_{i}$ from partition of unity, with the continuity of translations
in $L_{1}$, we have
\[
\mathbb{E}\int_{x}(\tilde{u}(\tau,x)-u(\tau,x))^{+}\leq\mathbb{E}\int_{x}(\tilde{\xi}(x)-\xi(x))^{+}+C\int_{0}^{\tau}\mathbb{E}\int_{x}(\tilde{u}(t,x)-u(t,x))^{+}\mathrm{d}t
\]
holds for almost all $\tau\in[0,T]$. Hence, \cref{eq:The first part of esssup}
follows from Gronwall's inequality.

Now, we prove \cref{eq:the first part estimation}. Notice
that
\begin{align*}
\mathbb{E}\int_{x,y}(\xi(x)-\xi(y))^{+}\varrho_{\varepsilon}(x-y)\psi(x) & \leq\int_{\mathbb{R}^{d}}\varrho_{\varepsilon}(h)\cdot\mathbb{E}\int_{B\cap D}(\xi(x)-\xi(x-h))^{+}\mathrm{d}x\mathrm{d}h\\
 & \leq\sup_{|h|\leq\varepsilon}\mathbb{E}\int_{B\cap D}(\bar{\xi}(x)-\bar{\xi}(x-h))^{+}\mathrm{d}x.
\end{align*}
Fixing $s_{1}\in(0,T]$ and integrating \cref{eq:L1 0-tau} over
$\tau\in(0,s_{1})$, we have
\begin{align*}
 & \mathbb{E}\int_{0}^{s_{1}}\int_{x,y}(\tilde{u}(\tau,x)-u(\tau,y))^{+}\varrho_{\varepsilon}(x-y)\psi(x)\mathrm{d}\tau\\
 & \leq T\mathbb{E}\int_{x}(\tilde{\xi}(x)-\xi(x))^{+}\psi(x)+T\sup_{|h|\leq2\varepsilon}\mathbb{E}\int_{B\cap D}(\bar{\xi}(x)-\bar{\xi}(x-h))^{+}\mathrm{d}x\\
 & \quad+\mathbb{E}\int_{0}^{s_{1}}\int_{0}^{\tau}\int_{x}\Delta_{x}\psi(x)(\tilde{\Phi}(\tilde{u}(t,x))-\Phi(u(t,x)))^{+}\mathrm{d}t\mathrm{d}\tau+TM\\
 &\quad+C\mathbb{E}\int_{0}^{s_{1}}\int_{0}^{\tau}\int_{x,y}\bigg[\mathbf{1}_{B\cap\overline{D}}(x)\cdot\Big(\varepsilon^{2}\sum_{i,j}|\partial_{x_{i}y_{j}}\varrho_{\varepsilon}(x-y)|\\
 & \quad+\varepsilon\sum_{i}|\partial_{x_{i}}\varrho_{\varepsilon}(x-y)|+\varrho_{\varepsilon}(x-y)\Big)(\tilde{u}(t,x)-u(t,y))^{+}\bigg]\mathrm{d}t\mathrm{d}\tau.
\end{align*}
Notice that $\varepsilon|\partial_{x_{i}}\varrho_{\varepsilon}|$
and $\varepsilon^{2}|\partial_{x_{i}x_{j}}\varrho_{\varepsilon}|$
are approximations of the identity up to a constant. Taking $\varepsilon\rightarrow0$
and adding with different $\psi_{i}$, from the partition of unity,
we have
\begin{align*}
  &\mathbb{E}\int_{0}^{s_{1}}\int_{x}(\tilde{u}(\tau,x)-u(\tau,x))^{+}\mathrm{d}\tau\\
 & \lesssim\mathbb{E}\int_{x}(\tilde{\xi}(x)-\xi(x))^{+}+\sup_{|h|\leq2\varepsilon}\mathbb{E}\int_{D}(\bar{\xi}(x)-\bar{\xi}(x-h))^{+}\mathrm{d}x\\
 & \quad+M+\mathbb{E}\int_{0}^{s_{1}}\int_{0}^{t}\int_{x}(\tilde{u}(\tau,x)-u(\tau,x))^{+}\mathrm{d}\tau\mathrm{d}t.
\end{align*}
Using Gronwall's inequality, we acquire \cref{eq:the first part estimation}.

For the case that $u$ in $(\tilde{u}-u)^+$ has the $(\star)$-property, using \cref{eq:intermitineq-1} with the $(\star)$-property of $u$ and following the proceeding method, we obtain the desired estimates.
\end{proof}

\section{Approximation\label{sec:Approximation}}

We approximate the function $\Phi$ to make the approximate equations
non-degenerate. The following proposition is taken from \cite{dareiotis2019entropy,dareiotis2020nonlinear}
and we refer to \cite[Proposition 5.1]{dareiotis2019entropy} for
the proof.
\begin{prop}
\label{prop:Phi_n}Let $\Phi$ satisfy \cref{assu:=00005CPhi}
with a constant $K>1$. Then, for all $n$ there exists an increasing
function $\Phi_{n}\in C^{\infty}(\mathbb{R})$ with bounded derivatives,
satisfying \cref{assu:=00005CPhi} with constant $3K$,
such that $\mathfrak{a}_{n}(r)\geq2/n$, and 
\begin{equation}
\sup_{|r|\leq n}|\mathfrak{a}(r)-\mathfrak{a}_{n}(r)|\leq4/n.\label{eq:the approximation of Phi}
\end{equation}
\end{prop}

Define $\xi_{n}\coloneqq\xi\wedge n$. Denote by $(\cdot,\cdot)$
the inner product in $L_{2}(D)$.
\begin{defn}
An $L_{2}$-solution  $u_{n}$ to $\Pi(\Phi_{n},\xi_{n})$ is a continuous $L_{2}(D)$-valued
process, such that $u_{n}\in L_{2}(\Omega_{T};H_{0}^{1}(D))$,
$\nabla\Phi_{n}(u_{n})\in L_{2}(\Omega_{T};L_{2}(D))$, and the equality
\begin{align*}
(u_{n}(t,\cdot),\phi) & =(\xi_{n},\phi)-\int_{0}^{t}\Big((\nabla\Phi_{n}(u_{n}),\nabla\phi)+(a^{ij}(\cdot,u_{n})\partial_{x_{j}}u_n+b^{i}(\cdot,u_{n})\\
 & \quad+f^{i}(\cdot,u_{n}),\partial_{x_{i}}\phi)+(F(\cdot,u_{n}),\phi)\Big)\mathrm{d}s-\int_{0}^{t}(\sigma^{k}(\cdot,u_{n}),\nabla\phi)\mathrm{d}W^{k}(s)
\end{align*}
holds for all $\phi\in C_{c}^{\infty}(D)$, almost surely for all
$t\in[0,T]$.
\end{defn}

Differing from \cref{def:entropy-solution}, with a strong
regularity of the solution $u_{n}$, we can consider the Dirichlet
boundary condition in the sense of trace. From \cref{eq:sigma x},
\cref{eq:f x} and \cref{eq:F} in \cref{assu:Coefficients},
we have the linear growth of $\sigma_{x_{i}}^{i}$, $f_{x_{i}}^{i}$
and $F$ in $r$. Combining with \cref{eq:a K} in \cref{assu:=00005CPhi},
for all $p\geq2$, the $L_{2}$-solution $u_{n}$ to $\Pi(\Phi_{n},\xi_{n})$
has the following a priori estimates
\begin{align}
\mathbb{E}\sup_{t\leq T}\Vert u_{n}\Vert_{L_{2}(D)}^{p}+\mathbb{E}\Vert\nabla\llbracket\mathfrak{a}_{n}\rrbracket(u_{n})\Vert_{L_{2}(D_{T})}^{p} & \leq C(1+\mathbb{E}\Vert\xi_{n}\Vert_{L_{2}(D)}^{p}),\label{eq:the uniform estimates of u_n and a_n}\\
\mathbb{E}\sup_{t\leq T}\Vert u_{n}\Vert_{L_{m+1}(D)}^{m+1}+\mathbb{E}\Vert\nabla\Phi_{n}(u_{n})\Vert_{L_{2}(D_{T})}^{2} & \leq C(1+\mathbb{E}\Vert\xi_{n}\Vert_{L_{m+1}(D)}^{m+1}).\label{eq:the uniform estimates of u_n and Phi_n}
\end{align}
The proof of the above two estimates is almost the same as \cite[Lemma A.1]{dareiotis2020nonlinear}
with the Dirichlet boundary condition of $u_{n}$, and we omit it
here. It is worth noting that the fact $\mathfrak{a}_{n}\geq2/n>0$
and \cref{eq:the uniform estimates of u_n and a_n} indicate
\begin{equation}
\mathbb{E}\Vert\nabla u_{n}\Vert_{L_{2}(D_{T})}^{p}\leq C(n)(1+\mathbb{E}\Vert\xi_{n}\Vert_{L_{2}(D)}^{p}).\label{eq:u_n one order}
\end{equation}

\begin{rem}
\label{rem:L2=00003Dentropy}Note that for all $(\eta,\varphi,\varrho)\in\mathcal{E}\times C_{c}^{\infty}([0,T))\times C_{c}^{\infty}(D)$
or $(\eta,\varphi,\varrho)\in\mathcal{E}_{0}\times C_{c}^{\infty}([0,T))\times C^{\infty}(\overline{D})$
and $\phi\coloneqq\varphi\times\varrho\geq0$, applying It\^{o}'s formula
(see e.g. \cite{krylov2013relatively}) to $u_{n}\mapsto\int_{D}\eta(u_{n})\phi\mathrm{d}x$,
with \cref{eq:the uniform estimates of u_n and a_n}-\cref{eq:the uniform estimates of u_n and Phi_n}
and \cref{assu:=00005CPhi}, we have that the $L_{2}$-solution
$u_{n}$ is also an entropy solution to $\Pi(\Phi_{n},\xi_{n})$.
Using \cref{prop:non-negative}, when $0\leq\xi\in L_{m+1}(\Omega,\mathcal{F}_{0};L_{m+1}(D))$,
we have $u_{n}\geq0$ for almost all $(\omega,t,x)\in\Omega_{T}\times D$.
\end{rem}

\begin{prop}
\label{prop:u_n star property}Let $0\leq\xi\in L_{m+1}(\Omega,\mathcal{F}_{0};L_{m+1}(D))$ and Assumptions \ref{assu:=00005CPhi}, \ref{assu:Coefficients} and \ref{assu:zero condition} hold.
 For each $n\in\mathbb{N}$, let $u_{n}$
be an $L_{2}$-solution of $\Pi(\Phi_{n},\xi_{n})$. Then, $u_{n}$
has the $(\star)$-property. If in addition $\mathbb{E}\Vert\xi\Vert_{L_{2}(D)}^{2(m+1)/m}<\infty$,
the constants $C$ in \cref{def:star property} are independent of $n$.
\end{prop}

With the help of \cref{rem:L2=00003Dentropy}, \cref{prop:u_n star property},
and \cref{thm:the first part} (i), following almost the same
argument as \cite[Proposition 5.4]{dareiotis2020nonlinear}, we have
the existence and uniqueness of the $L_{2}$-solution $u_{n}$.
Here, we omit the proof.
\begin{prop}
\label{prop:L2 solution}Let $0\leq\xi\in L_{m+1}(\Omega,\mathcal{F}_{0};L_{m+1}(D))$
and Assumptions \ref{assu:=00005CPhi}, \ref{assu:Coefficients} and \ref{assu:zero condition} hold. Then, for each $n\in\mathbb{N}$,
$\Pi(\Phi_{n},\xi_{n})$ admits a unique $L_{2}$-solution $u_{n}$.
\end{prop}

\begin{proof}[Proof of \cref{prop:u_n star property}]
We first prove \Cref{def:star property} (i). Fix $i\in\{0,1,\ldots,N\}$. 
For the sake of brevity, we define $B\coloneqq B_{i}$,
$\psi\coloneqq\psi_{i}$ and $\varrho_{\varepsilon}(x-y)\coloneqq\varrho_{\varepsilon,i}(x-y)$ which are introduced in
the definition of the spatial mollifier in \Cref{sec:-property-and--esitmate}.
 Fix sufficiently small $\gamma>0$. Since $y\mapsto\varrho_{\gamma}(x-y)\in C_{c}^{\infty}(D)$
for all $(x,\gamma)\in(B\cap\overline{D})\times(0,\bar{\varepsilon})$,
for a function $f\in L_{2}(D)$, let $f^{(\gamma)}(x)\coloneqq\int_{D}f(z)\varrho_{\gamma}(x-z)\mathrm{d}z$.
Then, on $B\cap\overline{D}$, $u_{n}^{(\gamma)}$ satisfies (pointwise) the
equation
\begin{align*}
\mathrm{d}u_{n}^{(\gamma)} & =\Big[\Delta(\Phi_{n}(u_{n}))^{(\gamma)}+\partial_{x_{i}}\big(a^{ij}(\cdot,u_{n})\partial_{x_{j}}u_{n}+b^{i}(\cdot,u_{n})+f^{i}(\cdot,u_{n})\big)^{(\gamma)}\\
 & \quad+\big(F(\cdot,u_{n})\big)^{(\gamma)}\Big]\mathrm{d}t+\partial_{x_{i}}\big(\sigma^{ik}(\cdot,u_{n})\big)^{(\gamma)}\mathrm{d}W^{k}(t).
\end{align*}
Note that
\[
\mathbb{E}\int_{s,x}H_{\theta}(s,x,u_{n}(s,x))=\lim_{\lambda\rightarrow0}\mathbb{E}\int_{s,x,z}H_{\theta}(s,x,z)\mathbf{1}_{B\cap\overline{D}}(x)\rho_{\lambda}(u_{n}(s,x)-z),
\]
and
\begin{align*}
 & \Big|\mathbb{E}\int_{s,x,z}H_{\theta}(s,x,z)\mathbf{1}_{B\cap \overline{D}}(x)\big(\rho_{\lambda}(u_{n}(s,x)-z)-\rho_{\lambda}(u_{n}^{(\gamma)}(s,x)-z)\big)\Big|\\
 & \leq C\Big(\mathbb{E}\Vert u_{n}-u_{n}^{(\gamma)}\Vert_{L_{1}([0,T]\times(B\cap\overline{D}))}^{2}\Big)^{\frac{1}{2}}\Big(\mathbb{E}\Vert\partial_{z}H_{\theta}\Vert_{L_{\infty}(D_{T}\times\mathbb{R})}^{2}\Big)^{\frac{1}{2}}\xrightarrow{\gamma\rightarrow0}0.
\end{align*}
With \cref{rem:Htheta format}, we have $\mathbb{E}H_{\theta}(s,x,z)X=0$
for any $\mathcal{F}_{s-\theta}$-measurable bounded random variable
$X$. Then,
\begin{align*}
 & \mathbb{E}\int_{s,x,z}H_{\theta}(s,x,z)\mathbf{1}_{B\cap \overline{D}}(x)\rho_{\lambda}(u_{n}^{(\gamma)}(s,x)-z)\\
 & =\mathbb{E}\int_{s,x,z}H_{\theta}(s,x,z)\mathbf{1}_{B\cap \overline{D}}(x)\Big(\rho_{\lambda}(u_{n}^{(\gamma)}(s,x)-z)-\rho_{\lambda}(u_{n}^{(\gamma)}(s-\theta,x)-z)\Big).
\end{align*}
Using It\^{o}'s formula, we have
\begin{align*}
 & \int_{s,x,z}H_{\theta}(s,x,z)\Big(\rho_{\lambda}(u_{n}^{(\gamma)}(s,x)-z)-\rho_{\lambda}(u_{n}^{(\gamma)}(s-\theta,x)-z)\Big)=\sum_{i=1}^{5}N_{\lambda,\gamma}^{(i)},
\end{align*}
where
\begin{align*}
N_{\lambda,\gamma}^{(1)}&\coloneqq\int_{s,x,z}H_{\theta}(s,x,z)\int_{s-\theta}^{s}\rho_{\lambda}^{\prime}(u_{n}^{(\gamma)}(t,x)-z)\Delta_{x}(\Phi_{n}(u_{n}))^{(\gamma)}\mathrm{d}t,
\\
N_{\lambda,\gamma}^{(2)}&\coloneqq\int_{s,x,z}H_{\theta}(s,x,z)\int_{s-\theta}^{s}\rho_{\lambda}^{\prime}(u_{n}^{(\gamma)}(t,x)-z)\partial_{x_{i}}\big(a^{ij}(\cdot,u_{n})\partial_{x_{j}}u_{n}+b^{i}(\cdot,u_{n})\big)^{(\gamma)}\mathrm{d}t,
\\
N_{\lambda,\gamma}^{(3)}&\coloneqq\int_{s,x,z}H_{\theta}(s,x,z)\int_{s-\theta}^{s}\rho_{\lambda}^{\prime}(u_{n}^{(\gamma)}(t,x)-z)\big(\partial_{x_{i}}f^{i}(\cdot,u_{n})+F(\cdot,u_{n})\big)^{(\gamma)}\mathrm{d}t,
\\
N_{\lambda,\gamma}^{(4)}&\coloneqq\int_{s,x,z}H_{\theta}(s,x,z)\int_{s-\theta}^{s}\rho_{\lambda}^{\prime}(u_{n}^{(\gamma)}(t,x)-z)\partial_{x_{i}}\big(\sigma^{ik}(\cdot,u_{n})\big)^{(\gamma)}\mathrm{d}W^{k}(t),
\\
N_{\lambda,\gamma}^{(5)}&\coloneqq\frac{1}{2}\int_{s,x,z}H_{\theta}(s,x,z)\int_{s-\theta}^{s}\rho_{\lambda}^{\prime\prime}(u_{n}^{(\gamma)}(t,x)-z)\sum_{k=1}^{\infty}|\partial_{x_{i}}\big(\sigma^{ik}(\cdot,u_{n})\big)^{(\gamma)}|^{2}\mathrm{d}t.
\end{align*}
Using the divergence theorem in $x$, we have $N_{\lambda,\gamma}^{(1)}=\sum_{i=1}^{3}N_{\lambda,\gamma}^{(1,i)}$,
where
\begin{align*}
N_{\lambda,\gamma}^{(1,1)}&\coloneqq-\int_{s,x,z}\nabla_{x}H_{\theta}(s,x,z)\int_{s-\theta}^{s}\rho_{\lambda}^{\prime}(u_{n}^{(\gamma)}(t,x)-z)\nabla_{x}(\Phi_{n}(u_{n}))^{(\gamma)}\mathrm{d}t,
\\
N_{\lambda,\gamma}^{(1,2)}&\coloneqq-\int_{s,x,z}H_{\theta}(s,x,z)\int_{s-\theta}^{s}\rho_{\lambda}^{\prime\prime}(u_{n}^{(\gamma)}(t,x)-z)\nabla_{x}u_{n}^{(\gamma)}(t,x)\nabla_{x}(\Phi_{n}(u_{n}))^{(\gamma)}\mathrm{d}t,
\\
N_{\lambda,\gamma}^{(1,3)}&\coloneqq\int_{s,z}\int_{\partial D}H_{\theta}(s,x,z)\int_{s-\theta}^{s}\rho_{\lambda}^{\prime}(u_{n}^{(\gamma)}(t,x)-z)\nabla_{x}(\Phi_{n}(u_{n}))^{(\gamma)}\cdot\nu\mathrm{d}t\mathrm{d}S,
\end{align*}
and $\nu(x)$ is the unit normal vector of $\partial D$ at $x$. Using
integration by parts in $z$ and \cref{eq:the uniform estimates of u_n and Phi_n},
we have
\begin{align*}
\mathbb{E}|N_{\lambda,\gamma}^{(1,1)}| & =\mathbb{E}\Big|\int_{s,x,z}\mathbf{1}_{\{s>\theta\}}\nabla_{x}\partial_{z}H_{\theta}(s,x,z)\int_{s-\theta}^{s}\rho_{\lambda}(u_{n}^{(\gamma)}(t,x)-z)\nabla_{x}(\Phi_{n}(u_{n}))^{(\gamma)}\mathrm{d}t\Big|\\
 & \leq C\theta(\mathbb{E}\Vert\nabla_{x}\partial_{z}H_{\theta}\Vert_{L_{\infty}(D_{T}\times\mathbb{R})}^{2})^{\frac{1}{2}}(\mathbb{E}\Vert\nabla_{x}(\Phi_{n}(u_{n}))^{(\gamma)}\Vert_{L_{1}(D_{T})}^{2})^{\frac{1}{2}}\leq C\theta^{1-\mu}.
\end{align*}
Similarly, integrating by parts twice in $z$ on $N_{\lambda,\gamma}^{(1,2)}$,
we have
\begin{align*}
\lim_{\gamma\rightarrow0}\mathbb{E}|N_{\lambda,\gamma}^{(1,2)}| & \leq C\theta^{1-\mu}\lim_{\gamma\rightarrow0}(\mathbb{E}\Vert\nabla_{x}u_{n}^{(\gamma)}(t,x)\nabla_{x}(\Phi_{n}(u_{n}))^{(\gamma)}\Vert_{L_{1}(D_{T})}^{\frac{m+1}{m}})^{\frac{m}{m+1}}\\
 & \leq C\theta^{1-\mu}(\mathbb{E}\Vert\nabla_{x}u_{n}\nabla_{x}(\Phi_{n}(u_{n}))\Vert_{L_{1}(D_{T})}^{\frac{m+1}{m}})^{\frac{m}{m+1}}.\\
 & =C\theta^{1-\mu}(\mathbb{E}\Vert\nabla\llbracket\mathfrak{a}_{n}\rrbracket(u_{n})\Vert_{L_{2}(D_{T})}^{\frac{2(m+1)}{m}})^{\frac{m}{m+1}}\leq C(n)\theta^{1-\mu}.
\end{align*}
From the Dirichlet boundary condition of $u_{n}$, we have
\begin{align*}
\lim_{\gamma\rightarrow0}N_{\lambda,\gamma}^{(1,3)} & =\int_{s,z}\int_{\partial D}\mathbf{1}_{\{s>\theta\}}H_{\theta}(s,x,z)\int_{s-\theta}^{s}\rho_{\lambda}^{\prime}(0-z)\nabla_{x}(\Phi_{n}(u_{n}))^{(\gamma)}\cdot\nu\mathrm{d}t\mathrm{d}S\\
 & =\int_{s,z}\mathbf{1}_{\{z<0\}}\int_{\partial D}\mathbf{1}_{\{s>\theta\}}H_{\theta}(s,x,z)\int_{s-\theta}^{s}\rho_{\lambda}^{\prime}(0-z)\nabla_{x}(\Phi_{n}(u_{n}))^{(\gamma)}\cdot\nu\mathrm{d}t\mathrm{d}S.
\end{align*}
Since $\text{supp}\,\rho_{\lambda}^{\prime}\subset[0,\infty)$, the
integrand only acts when $z\in(-\infty,0]$. However, based on the non-negativity
of $u$ and the definition of $\mathcal{C}^{-}$, we have for all
$(h,z)\in\mathcal{C}^{-}\times(-\infty,0]$,
\[
h(u(t,y)-z)=0,\quad\text{a.s.}\ (t,\omega,y)\in\Omega_{T}\times D.
\]
which indicates $\lim_{\gamma\rightarrow0}N_{\lambda,\gamma}^{(1,3)}=0$.
Therefore,
\[
\underset{\gamma\rightarrow0}{\lim\sup}\mathbb{E}|N_{\lambda,\gamma}^{(1)}|\leq C(n)\theta^{1-\mu}.
\]
Now we estimate $N_{\lambda,\gamma}^{(2)}+N_{\lambda,\gamma}^{(5)}$.
As the estimate of $N_{\lambda,\gamma}^{(1,3)}$, using the divergence
theorem in $x$ and combining with the definition of $a^{ij}$ and
$b^{i}$, we have 
\begin{align*}
 & \underset{\gamma\rightarrow0}{\lim\sup}\,\mathbb{E}|N_{\lambda,\gamma}^{(2)}+N_{\lambda,\gamma}^{(5)}|\\
 & \leq\mathbb{E}\Big|\int_{s,x,z}\partial_{x_{i}}H_{\theta}(s,x,z)\int_{s-\theta}^{s}\rho_{\lambda}^{\prime}(u_{n}(t,x)-z)\big(a^{ij}(\cdot,u_{n})\partial_{x_{j}}u_{n}\big)\mathrm{d}t\Big|\\
 & \quad+\mathbb{E}\Big|\int_{s,x,z}\partial_{x_{i}}H_{\theta}(s,x,z)\int_{s-\theta}^{s}\rho_{\lambda}^{\prime}(u_{n}(t,x)-z)\big(b^{i}(\cdot,u_{n})\big)\mathrm{d}t\Big|\\
 & \quad+\mathbb{E}\Big|\frac{1}{2}\int_{s,x,z}H_{\theta}(s,x,z)\int_{s-\theta}^{s}\rho_{\lambda}^{\prime\prime}(u_{n}(t,x)-z)\sum_{k=1}^{\infty}|\big(\sigma_{x_{i}}^{ik}(\cdot,u_{n})\big)|^{2}\mathrm{d}t\Big|.
\end{align*}
Using the identity
\begin{align*}
\rho_{\lambda}^{\prime}(u_{n}(t,x)-z)a^{ij}(x,u_{n})\partial_{x_{j}}u_{n} & =\partial_{x_{j}}\llbracket a^{ij}\rho_{\lambda}^{\prime}(\cdot-z)\rrbracket(x,u_{n}(t,x))\\
 & \quad-\llbracket a_{x_{j}}^{ij}\rho_{\lambda}^{\prime}(\cdot-z)\rrbracket(x,u_{n}(t,x)),
\end{align*}
and 
\begin{align*}
 & \mathbb{E}\Big|\int_{s,x,z}\partial_{x_{i}}H_{\theta}(s,x,z)\int_{s-\theta}^{s}\partial_{x_{j}}\llbracket a^{ij}\rho_{\lambda}^{\prime}(\cdot-z)\rrbracket(x,u_{n}(t,x))\mathrm{d}t\Big|\\
 & \leq C\theta(\mathbb{E}\Vert\partial_{z}\partial_{x_{i}}\partial_{x_{j}}H_{\theta}\Vert_{L_{\infty}(D_{T}\times\mathbb{R})}^{2})^{\frac{1}{2}}(\mathbb{E}\Vert\llbracket a^{ij}\rho_{\lambda}(\cdot-z)\rrbracket(\cdot,u_{n})\Vert_{L_{1}(D_{T})}^{2})^{\frac{1}{2}},
\end{align*}
with the linear growth of $\sigma_{x_{i}}^{i}$, $b^{i}$ and and
the boundness of $a_{x_{j}}^{ij}$ and $a^{ij}$ derived from \cref{eq:sigma x}-\cref{eq:sigma r}
and \cref{eq:f r} in \cref{assu:Coefficients}, we have
\[
\underset{\gamma\rightarrow0}{\lim\sup}\mathbb{E}|N_{\lambda,\gamma}^{(2)}+N_{\lambda,\gamma}^{(5)}|\leq C\theta^{1-\mu}(1+\mathbb{E}\Vert u_{n}\Vert_{L_{2}(D_{T})}^{\frac{2(m+1)}{m}})^{\frac{m}{m+1}}\leq C(n)\theta^{1-\mu}.
\]
For $N_{\lambda,\gamma}^{(3)}$, with the linear growth of $f_{x_{i}}^{i}$,
$F$ and the boundness of $f_{r}^{i}$ and $f_{rx_{i}}^{i}$ derived
from \cref{eq:f r}-\cref{eq:F}
in \cref{assu:Coefficients}, we similarly obtain
\[
\underset{\gamma\rightarrow0}{\lim\sup}\mathbb{E}|N_{\lambda,\gamma}^{(3)}|\leq C\theta^{1-\mu}(1+\mathbb{E}\Vert u_{n}\Vert_{L_{2}(D_{T})}^{2})^{2}\leq C(n)\theta^{1-\mu}.
\]
Using It\^{o}'s isometry and taking $\gamma\rightarrow0$, since
\begin{align}
\quad & \int_{s,x,z,y}\int_{s-\theta}^{s}\int_{0}^{z}\sigma_{ry_{i}}^{ik}(y,r)h(r-z)\mathrm{d}r\phi_{\theta}\rho_{\lambda}^{\prime}(u_{n}(t,x)-z)\partial_{x_{j}}\sigma^{jk}(x,u_{n}(t,x))\mathrm{d}t\label{eq:divergence}\\
\quad & =-\int_{s,x,z,y}\int_{s-\theta}^{s}\int_{0}^{z}\sigma_{r}^{ik}(y,r)h(r-z)\mathrm{d}r\partial_{y_{i}}\phi_{\theta}\rho_{\lambda}^{\prime}(u_{n}(t,x)-z)\partial_{x_{j}}\sigma^{jk}(x,u_{n}(t,x))\mathrm{d}t,\nonumber 
\end{align}
we have
\[
\lim_{\gamma\rightarrow0}\mathbb{E}N_{\lambda,\gamma}^{(4)}=\sum_{i=1}^{6}I_{i},
\]
where
\begin{align*}
I_{1}&\coloneqq-\int_{s,x,z,y}\int_{s-\theta}^{s}\Big(\int_{z}^{u(t,y)}\sigma_{ry_{i}}^{ik}(y,r)h(r-z)\mathrm{d}r\phi_{\theta}\rho_{\lambda}^{\prime}(u_{n}(t,x)-z)\\
&\quad\cdot\sigma_{x_{j}}^{jk}(x,u_{n}(t,x))\Big)\mathrm{d}t,
\\
I_{2}&\coloneqq-\int_{s,x,z,y}\int_{s-\theta}^{s}\Big(\int_{z}^{u(t,y)}\sigma_{ry_{i}}^{ik}(y,r)h(r-z)\mathrm{d}r\phi_{\theta}\rho_{\lambda}^{\prime}(u_{n}(t,x)-z)\\
&\quad\cdot\sigma_{r}^{jk}(x,u_{n}(t,x))\partial_{x_{j}}u_{n}(t,x)\Big)\mathrm{d}t,
\\
I_{3}&\coloneqq-\int_{s,x,z,y}\int_{s-\theta}^{s}\Big(\int_{z}^{u(t,y)}\sigma_{r}^{ik}(y,r)h(r-z)\mathrm{d}r\partial_{y_{i}}\phi_{\theta}\rho_{\lambda}^{\prime}(u_{n}(t,x)-z)\\
&\quad\cdot\sigma_{x_{j}}^{jk}(x,u_{n}(t,x))\Big)\mathrm{d}t,
\\
I_{4}&\coloneqq-\int_{s,x,z,y}\int_{s-\theta}^{s}\Big(\int_{z}^{u(t,y)}\sigma_{r}^{ik}(y,r)h(r-z)\mathrm{d}r\partial_{y_{i}}\phi_{\theta}\rho_{\lambda}^{\prime}(u_{n}(t,x)-z)\\
&\quad\cdot\sigma_{r}^{jk}(x,u_{n}(t,x))\partial_{x_{j}}u_{n}(t,x)\Big)\mathrm{d}t,
\\
I_{5}&\coloneqq\int_{s,x,z,y}\int_{s-\theta}^{s}h(u(t,y)-z)\phi_{\theta}\sigma_{y_{i}}^{ik}(y,u(t,y))\rho_{\lambda}^{\prime}(u_{n}(t,x)-z)\sigma_{x_{j}}^{jk}(x,u_{n}(t,x))\mathrm{d}t,
\\
I_{6}&\coloneqq\int_{s,x,z,y}\int_{s-\theta}^{s}\Big(h(u(t,y)-z)\phi_{\theta}\sigma_{y_{i}}^{ik}(y,u(t,y))\rho_{\lambda}^{\prime}(u_{n}(t,x)-z)\\
&\quad\cdot\sigma_{r}^{jk}(x,u_{n}(t,x))\partial_{x_{j}}u_{n}(t,x)\Big)\mathrm{d}t.
\end{align*}
For $I_{2}+I_{4}$, notice that
\begin{align*}
\rho_{\lambda}^{\prime}(u_{n}(t,x)-z)\sigma_{r}^{jk}(x,u_{n}(t,x))\partial_{x_{j}}u_{n}(t,x) & =\partial_{x_{j}}\int_{0}^{u_{n}(t,x)}\rho_{\lambda}^{\prime}(\tilde{r}-z)\sigma_{r}^{jk}(x,\tilde{r})\mathrm{d}\tilde{r}\\
 & \quad-\int_{0}^{u_{n}(t,x)}\rho_{\lambda}^{\prime}(\tilde{r}-z)\sigma_{rx_{j}}^{jk}(x,\tilde{r})\mathrm{d}\tilde{r}.
\end{align*}
We can apply the divergence theorem in $x$ and the Dirichlet boundary
condition and integrate by parts in $z$. Moreover, Since $h\in\mathcal{C}^{-}$,
the integrand is non-zero only when $u(t,y)<z$. Using $\text{supp}\,\rho_{\lambda}\subset\mathbb{R}_{+}$
and the non-negativity of $u$, we have
\begin{align}
 & I_{2}+I_{4}\label{eq:instead D2 D4}\\
 & =-\int_{s,x,z,y}\int_{s-\theta}^{s}\int_{z}^{u(t,y)}\sigma_{ry_{i}}^{ik}(y,r)h^{\prime}(r-z)\mathrm{d}r\partial_{x_{j}}\phi_{\theta}\int_{u(t,y)}^{u_{n}(t,x)}\rho_{\lambda}(\tilde{r}-z)\sigma_{r}^{jk}(x,\tilde{r})\mathrm{d}\tilde{r}\mathrm{d}t\nonumber \\
 & \quad-\int_{s,x,z,y}\int_{s-\theta}^{s}\int_{z}^{u(t,y)}\sigma_{ry_{i}}^{ik}(y,r)h^{\prime}(r-z)\mathrm{d}r\phi_{\theta}\int_{u(t,y)}^{u_{n}(t,x)}\rho_{\lambda}(\tilde{r}-z)\sigma_{rx_{j}}^{jk}(x,\tilde{r})\mathrm{d}\tilde{r}\mathrm{d}t\nonumber \\
 & \quad-\int_{s,x,z,y}\int_{s-\theta}^{s}\int_{z}^{u(t,y)}\sigma_{r}^{ik}(y,r)h^{\prime}(r-z)\mathrm{d}r\partial_{y_{i}x_{j}}\phi_{\theta}\int_{u(t,y)}^{u_{n}(t,x)}\rho_{\lambda}(\tilde{r}-z)\sigma_{r}^{jk}(x,\tilde{r})\mathrm{d}\tilde{r}\mathrm{d}t\nonumber \\
 & \quad-\int_{s,x,z,y}\int_{s-\theta}^{s}\int_{z}^{u(t,y)}\sigma_{r}^{ik}(y,r)h^{\prime}(r-z)\mathrm{d}r\partial_{y_{i}}\phi_{\theta}\int_{u(t,y)}^{u_{n}(t,x)}\rho_{\lambda}(\tilde{r}-z)\sigma_{rx_{j}}^{jk}(x,\tilde{r})\mathrm{d}\tilde{r}\mathrm{d}t.\nonumber 
\end{align}
Similarly, we have
\begin{align*}
I_{6} & =-\int_{s,x,z,y}\int_{s-\theta}^{s}h^{\prime}(u(t,y)-z)\phi_{\theta}\sigma_{y_{i}}^{ik}(y,u(t,y))\partial_{x_{j}}\int_{0}^{u_{n}(t,x)}\rho_{\lambda}(\tilde{r}-z)\sigma_{r}^{jk}(x,\tilde{r})\mathrm{d}\tilde{r}\mathrm{d}t\\
 & \quad+\int_{s,x,z,y}\int_{s-\theta}^{s}h^{\prime}(u(t,y)-z)\phi_{\theta}\sigma_{y_{i}}^{ik}(y,u(t,y))\int_{0}^{u_{n}(t,x)}\rho_{\lambda}(\tilde{r}-z)\sigma_{rx_{j}}^{jk}(x,\tilde{r})\mathrm{d}\tilde{r}\mathrm{d}t\\
 & =\int_{s,x,z,y}\int_{s-\theta}^{s}h^{\prime}(u(t,y)-z)\partial_{x_{j}}\phi_{\theta}\sigma_{y_{i}}^{ik}(y,u(t,y))\int_{u(t,y)}^{u_{n}(t,x)}\rho_{\lambda}(\tilde{r}-z)\sigma_{r}^{jk}(x,\tilde{r})\mathrm{d}\tilde{r}\mathrm{d}t\\
 & \quad+\int_{s,x,z,y}\int_{s-\theta}^{s}h^{\prime}(u(t,y)-z)\phi_{\theta}\sigma_{y_{i}}^{ik}(y,u(t,y))\int_{u(t,y)}^{u_{n}(t,x)}\rho_{\lambda}(\tilde{r}-z)\sigma_{rx_{j}}^{jk}(x,\tilde{r})\mathrm{d}\tilde{r}\mathrm{d}t.
\end{align*}
For $I_{1}$, $I_{3}$ and $I_{5}$, we integrate by parts in $z$.
Therefore, we have
\[
\lim_{\lambda\rightarrow0}\lim_{\gamma\rightarrow0}\mathbb{E}N_{\lambda,\gamma}^{(4)}=\mathcal{E}(u,u_{n},\theta).
\]
Hence,
\[
\mathbb{E}\int_{s,x}H_{\theta}(s,x,u_{n}(s,x))\leq C(n)\theta^{1-\mu}+\mathcal{E}(u,u_{n},\theta).
\]
Furthermore, inspired by \cref{eq:the uniform estimates of u_n and a_n},
if $\mathbb{E}\Vert\xi\Vert_{L_{2}(D)}^{\frac{2(m+1)}{m}}<\infty$,
we can choose $C$ independent of $n$.

To prove (ii) in \Cref{def:star property}, note that for all $g\in\Gamma_{B}^{+}$, we have $g(\cdot,y)\in C_{c}(D)$
for all $y\in B\cap\overline{D}$. It is easy to prove following the
proceeding proof. The reason is that the boundary terms vanish using
the divergence theorem in $x$, and other differences in the proof
are \cref{eq:divergence}, $I_{2}+I_{4}$ and $I_{6}$. For \cref{eq:divergence},
with $h\in\mathcal{C}^{+}$, we have for all $(y,z)\in D\times[0,\infty),$
\[
\int_{0}^{z}\sigma_{ry_{i}}^{ik}(y,r)h(r-z)\mathrm{d}r=\int_{0}^{z}\sigma_{r}^{ik}(y,r)h(r-z)\mathrm{d}r=0.
\]
For $z<0$, using \cref{eq:sigma x} in \cref{assu:Coefficients},
\cref{eq:the uniform estimates of u_n and a_n}, \cref{rem:L2=00003Dentropy}
and the definition of $\mathcal{C}^{+}$, the boundary term arising
in the divergence theorem in $y$ will disappear when $\lambda\rightarrow0$.
Therefore, we have \cref{eq:divergence}. 

As for $I_{2}+I_{4}$ and $I_{6}$, we can apply
\[
\partial_{x_{j}}\int_{0}^{u(t,y)}\rho_{\lambda}(\tilde{r}-z)\sigma_{r}^{jk}(x,\tilde{r})\mathrm{d}\tilde{r}\mathrm{d}t-\int_{0}^{u(t,y)}\rho_{\lambda}(\tilde{r}-z)\sigma_{rx_{j}}^{jk}(x,\tilde{r})\mathrm{d}\tilde{r}\mathrm{d}t=0
\]
instead of using the support of $h$. 
Therefore, this proposition is proved.
\end{proof}

\section{\label{sec:Existence-and-Uniqueness}Existence and Uniqueness}
Now, we give the proof of our main theorem.
\begin{proof}[Proof of \cref{thm:main theorem}]
In this part, we will use the $L_{1}$-estimates to prove that $\{u_{n}\}_{n\in\mathbb{N}}$
constructed in \Cref{sec:Approximation} is a Cauchy sequence. 

For $n,n^{\prime}\geq1$, let $u_{n}$ and $u_{n^{\prime}}$ be the
$L_{2}$-solutions of $\Pi(\Phi_{n},\xi_{n})$ and $\Pi(\Phi_{n^{\prime}},\xi_{n^{\prime}})$,
respectively. \Cref{rem:L2=00003Dentropy} shows that $u_{n}$
and $u_{n^{\prime}}$ are also entropy solutions. \Cref{prop:u_n star property}
indicates that $\{u_{n}\}_{n\in\mathbb{N}}$ has the $(\star)$-property. Without loss of generality, we assume that
$n\leq n^{\prime}$. Since $\beta\in((2\bar{\kappa})^{-1},1]$, we
can choose $\vartheta\in((m\land2)^{-1}\vee(2\beta)^{-1},\bar{\kappa})$
and $\alpha\in((2\vartheta)^{-1},1\land(m/2))$. Let $\delta=\varepsilon^{2\vartheta}$
and $\lambda=8/n$. Using \cref{eq:the approximation of Phi}, we
have $R_{\lambda}\geq n$. Applying \cref{thm:the first part}
with $u_{n}$ and $u_{n^{\prime}}$ and using \cref{eq:the uniform estimates of u_n and Phi_n}
and the triangle inequality 
\begin{align*}
 & \mathbb{E}\Vert\xi_{n}(\cdot)-\bar{\xi}_{n}(\cdot+h)\Vert_{L_{1}(D)}\\
 & \leq\mathbb{E}\Vert\xi(\cdot)-\bar{\xi}(\cdot+h)\Vert_{L_{1}(D)}+2\mathbb{E}\Vert\xi-\xi_{n}\Vert_{L_{1}(D)},\quad\forall n\in\mathbb{N},
\end{align*}
we have
\begin{align*}
 & \mathbb{E}\int_{0}^{T}\int_{D}|u_{n^{\prime}}(\tau,x)-u_{n}(\tau,x)|\mathrm{d}x\mathrm{d}\tau\\
 & \leq M(\varepsilon)+C\mathbb{E}\Vert\xi-\xi_{n^{\prime}}\Vert_{L_{1}(D)}+C\mathbb{E}\Vert\xi-\xi_{n}\Vert_{L_{1}(D)}+C\varepsilon^{-2}n^{-2}+C\varepsilon^{-1}n^{-1}\\
 & \quad+C\varepsilon^{-2}\mathbb{E}\Big(\big\Vert\mathbf{1}_{|u_{n}|\geq n}(1+u_{n})\big\Vert_{L_{m}(D_{T})}^{m}+\big\Vert\mathbf{1}_{|u_{n^{\prime}}|\geq n}(1+u_{n^{\prime}})\big\Vert_{L_{m}(D_{T})}^{m}\Big),
\end{align*}
where $M(\varepsilon)\rightarrow0$ as $\varepsilon\rightarrow0^+$.
For any $\varepsilon_{0}>0$, we select sufficiently small $\varepsilon\in(0,\bar{\varepsilon})$
 such that $M(\varepsilon)\leq\varepsilon_{0}$. Then,
using \cref{eq:the uniform estimates of u_n and Phi_n},
we can choose $n_{0}$ sufficiently large so that for $n_{0}\leq n\leq n^{\prime}$,
we have
\begin{align*}
 & C\mathbb{E}\Vert\xi_{n^{\prime}}-\xi\Vert_{L_{1}(D)}+C\mathbb{E}\Vert\xi-\xi_{n}\Vert_{L_{1}(D)}+C\varepsilon^{-2}n^{-2}+C\varepsilon^{-1}n^{-1}\\
 & +C\varepsilon^{-2}\mathbb{E}\Big(\big\Vert\mathbf{1}_{|u_{n}|\geq n}(1+u_{n})\big\Vert_{L_{m}(D_{T})}^{m}+\big\Vert\mathbf{1}_{|u_{n^{\prime}}|\geq n}(1+u_{n^{\prime}})\big\Vert_{L_{m}(D_{T})}^{m}\Big)\leq\varepsilon_{0}.
\end{align*}
Therefore, we have
\[
\lim_{n,n^{\prime}\rightarrow\infty}\Vert u_{n^{\prime}}(t,x)-u_{n}(t,x)\Vert_{L_1(\Omega_T\times D)}=0.
\]
Moreover, by taking a subsequence, we may assume
\begin{equation}
\lim_{n\rightarrow\infty}u_{n}=u,\quad\mbox{a.s.}\ (\omega,t,x)\in\Omega_{T}\times D.\label{eq:converge almost all}
\end{equation}
In addition, the sequence $\{|u_{n}(t,x)|^{q}\}_{n\in\mathbb{N}}$
is uniformly integrable on $\Omega_{T}\times D$ for all $q\in(0,m+1)$.
Now, we verify that $u$ is an entropy solution to $\Pi(\Phi,\xi)$
under \cref{def:entropy-solution}.

Firstly, with the definition of $\xi_{n}$, we have that $\{u_{n}\}_{n\in\mathbb{N}}$
is weak convergence in the Banach space $L_{m+1}(\Omega_{T};L_{m+1}(D))$.
Applying the Banach-Saks Theorem, taking a subsequence and using \cref{eq:the uniform estimates of u_n and Phi_n}, we have
\begin{align*}
\mathbb{E}\Vert u\Vert_{L_{m+1}(\Omega_{T};L_{m+1}(D))} & \leq\liminf_{n\rightarrow\infty}\Vert u_{n}\Vert_{L_{m+1}(\Omega_{T};L_{m+1}(D))}\\
 & \leq C(1+\liminf_{n\rightarrow\infty}\Vert\xi_{n}\Vert_{L_{m+1}(\Omega;L_{m+1}(D))})\\
 & \leq C(1+\Vert\xi\Vert_{L_{m+1}(\Omega;L_{m+1}(D))}).
\end{align*}
To prove \cref{def:entropy-solution} (ii) of $u$, let
$f\in C_{b}(\mathbb{R})$. From \cref{assu:=00005CPhi}
and \cref{eq:the uniform estimates of u_n and Phi_n}, we
have 
\begin{align*}
\sup_{n}\mathbb{E}\int_{0}^{T}\int_{D}|\llbracket\mathfrak{a}_{n}f\rrbracket(u_{n})|^{2}\mathrm{d}x\mathrm{d}t & \lesssim\sup_{n}\mathbb{E}\int_{0}^{T}\int_{D}(|u_{n}|+|u_{n}|^{\frac{m+1}{2}})^{2}\mathrm{d}x\mathrm{d}t<\infty.
\end{align*}
Combining \cref{eq:the uniform estimates of u_n and a_n}
with the fact that $\llbracket\mathfrak{a}_{n}f\rrbracket(u_{n})\in L_{2}(\Omega_{T};H_{0}^{1}(D))$,
we have
\[
\sup_{n}\mathbb{E}\int_{t}\Vert\llbracket\mathfrak{a}_{n}f\rrbracket(u_{n})\Vert_{H^{1}(D)}^{2}<\infty.
\]
With the pointwise convergence and uniform integrability of $u_{n}$
and \cref{prop:Phi_n}, by taking a subsequence, we obtain
the weak convergence of $\llbracket\mathfrak{a}_{n}f\rrbracket(u_{n})$
and $\llbracket\mathfrak{a}_{n}\rrbracket(u_{n})$ in $L_{2}(\Omega_{T};H_{0}^{1}(D))$
as $n\rightarrow\infty$, and the limits are $\llbracket\mathfrak{a}f\rrbracket(u)$
and $\llbracket\mathfrak{a}\rrbracket(u)$, respectively. On the other
hand, for all $\phi\in C_{c}^{\infty}([0,T)\times D)$ and $A\in\mathcal{F}$,
based on the strong convergence of $f(u_{n})\phi$ and weak convergence
of $\llbracket\mathfrak{a}_{n}\rrbracket(u_{n})$, we have
\begin{align*}
\mathbb{E}\bigg[\mathbf{1}_{A}\int_{t,x}\partial_{x_{i}}\llbracket\mathfrak{a}f\rrbracket(u)\phi\bigg] & =\lim_{n\rightarrow\infty}\mathbb{E}\bigg[\mathbf{1}_{A}\int_{t,x}\partial_{x_{i}}\llbracket\mathfrak{a}_{n}f\rrbracket(u_{n})\phi\bigg]\\
 & =\lim_{n\rightarrow\infty}\mathbb{E}\bigg[\mathbf{1}_{A}\int_{t,x}f(u_{n})\partial_{x_{i}}\llbracket\mathfrak{a}_{n}\rrbracket(u_{n})\phi\bigg]\\
 & =\mathbb{E}\bigg[\mathbf{1}_{A}\int_{t,x}f(u)\partial_{x_{i}}\llbracket\mathfrak{a}\rrbracket(u)\phi\bigg].
\end{align*}

For \cref{def:entropy-solution} (iii), let $A\in\mathcal{F}$.
Combining \cref{rem:L2=00003Dentropy} with It\^{o}'s product
rule, we have
\begin{align}
 & -\mathbb{E}\bigg[\mathbf{1}_{A}\int_{0}^{T}\int_{D}\eta(u_{n})\partial_{t}\phi\mathrm{d}x\mathrm{d}t\bigg]\label{eq:ito entorpy equality}\\
 & =\mathbb{E}\mathbf{1}_{A}\Big[\int_{D}\eta(\xi_{n})\phi(0)\mathrm{d}x+\int_{0}^{T}\int_{D}\big(\llbracket\mathfrak{a}_{n}^{2}\eta^{\prime}\rrbracket(u_{n})\Delta\phi+\llbracket a^{ij}\eta^{\prime}\rrbracket(x,u_{n})\partial_{x_{i}x_{j}}\phi\big)\mathrm{d}x\mathrm{d}t\nonumber \\
 & \quad+\int_{0}^{T}\int_{D}\big(\llbracket a_{x_{j}}^{ij}\eta^{\prime}-f_{r}^{i}\eta^{\prime}\rrbracket(x,u_{n})-\eta^{\prime}(u_{n})b^{i}(x,u_{n})\big)\partial_{x_{i}}\phi\mathrm{d}x\mathrm{d}t\nonumber \\
 & \quad+\int_{0}^{T}\int_{D}\big(-\llbracket f_{rx_{i}}^{i}\eta^{\prime}\rrbracket(x,u_{n})+\eta^{\prime}(u_{n})f_{x_{i}}^{i}(x,u_{n})+\eta^{\prime}(u_{n})F(x,u_{n})\big)\phi\mathrm{d}x\mathrm{d}t\nonumber \\
 & \quad+\int_{0}^{T}\int_{D}\Big(\frac{1}{2}\eta^{\prime\prime}(u_{n})\sum_{k=1}^{\infty}|\sigma_{x_{i}}^{ik}(x,u_{n})|^{2}\phi-\eta^{\prime\prime}(u_{n})|\nabla\llbracket\mathfrak{a}_{n}\rrbracket(u_{n})|^{2}\phi\Big)\mathrm{d}x\mathrm{d}t\nonumber \\
 & \quad+\int_{0}^{T}\int_{D}\Big(\eta^{\prime}(u_{n})\phi\sigma_{x_{i}}^{ik}(x,u_{n})-\llbracket\sigma_{rx_{i}}^{ik}\eta^{\prime}\rrbracket(x,u_{n})\phi-\llbracket\sigma_{r}^{ik}\eta^{\prime}\rrbracket(x,u_{n})\partial_{x_{i}}\phi\Big)\mathrm{d}x\mathrm{d}W^{k}(t)\Big].\nonumber 
\end{align}
Since $(\eta^{\prime\prime})^{1/2}\in C_{b}(\mathbb{R})$. As the
proof in checking (ii), we have 
\[
\partial_{x_{i}}\llbracket(\eta^{\prime\prime})^{1/2}\mathfrak{a}_{n}\rrbracket(u_{n})=(\eta^{\prime\prime}(u_{n}))^{1/2}\partial_{x_{i}}\llbracket\mathfrak{a}_{n}\rrbracket(u_{n}),
\]
and can assume that $\partial_{x_{i}}\llbracket(\eta^{\prime\prime})^{1/2}\mathfrak{a}_{n}\rrbracket(u_{n})$
converges weakly to $\partial_{x_{i}}\llbracket(\eta^{\prime\prime})^{1/2}\mathfrak{a}\rrbracket(u)$
in $L_{2}(\Omega_{T};L_{2}(D))$. Then, we also have the weakly convergence
in $L_{2}(\Omega_{T}\times D,\bar{\mu})$, where $\mathrm{d}\bar{\mu}\coloneqq\mathbf{1}_{B}\phi\mathrm{d}\mathbb{P}\otimes\mathrm{d}t\otimes\mathrm{d}x$,
which indicates
\[
\mathbb{E}\mathbf{1}_{B}\int_{t,x}\phi\eta^{\prime\prime}(u)|\nabla\llbracket\mathfrak{a}\rrbracket(u)|^{2}\leq\liminf_{n\rightarrow\infty}\mathbb{E}\mathbf{1}_{B}\int_{t,x}\phi\eta^{\prime\prime}(u_{n})|\nabla\llbracket\mathfrak{a}_{n}\rrbracket(u_{n})|^{2}.
\]
Therefore, taking inferior limit on \cref{eq:ito entorpy equality}
and using \cref{prop:Phi_n} and \cref{assu:Coefficients},
with the almost sure convergence and uniformly integrability of
$u_{n}$, we have that $u$ satisfies \cref{eq:entropy formulation}
almost surely. Therefore, $u$ is actually an entropy solution.

Now, we focus on the uniqueness. For $\bar{n}\in\mathbb{N}$, define
$\xi_{\bar{n}}\coloneqq\xi\wedge\bar{n}$ and denote by $u_{\bar{n}}$
the entropy solution of $\Pi(\Phi,\xi_{\bar{n}})$ constructed in
the proof of existence. From the construction of $u_{\bar{n}}$, with
 \cref{lem:Limit star property} and \cref{prop:u_n star property},
the entropy solution $u_{\bar{n}}$ has the $(\star)$-property. Let
$\tilde{u}$ be an entropy solution of $\Pi(\Phi,\tilde{\xi})$. Note that
\begin{align*}
&\mathbb{E}\Vert\tilde{u}(t,\cdot)-u_{\bar{n}}(t,\cdot)\Vert_{L_1(D)}\\
&=\mathbb{E}\Vert(\tilde{u}(t,\cdot)-u_{\bar{n}}(t,\cdot))^+\Vert_{L_1(D)}+\mathbb{E}\Vert(u_{\bar{n}}(t,\cdot)-\tilde{u}(t,\cdot))^+\Vert_{L_1(D)}.
\end{align*}
The estimate of the first part on the right hand side is obtained using \cref{thm:the first part} in the case that $u$ in $(\tilde{u}-u)^+$ has the $(\star)$-property. Similarly, the second part can be estimated using \cref{thm:the first part} in the case that $\tilde{u}$ in $(\tilde{u}-u)^+$ has the $(\star)$-property. Combining these two estimates, we have
\[
\underset{t\in[0,T]}{\mathrm{ess\,sup}}\,\mathbb{E}\Vert\tilde{u}(t,\cdot)-u_{\bar{n}}(t,\cdot)\Vert_{L_{1}(D)}\leq C\mathbb{E}\Vert\tilde{\xi}-\xi_{\bar{n}}\Vert_{L_{1}(D)},
\]
where the constant $C$ depends only on $N_{0}$, $K$, $d$, $T$
and $|D|$. Taking the limit $\bar{n}\rightarrow\infty$, we obtain \cref{thm:main theorem} (i).
 Following this method and using \cref{thm:the first part}
(i), we complete the proof of \cref{thm:main theorem} (iii).
\end{proof}

\section{Some auxiliary estimates}
\begin{lem}
\label{lem:initial value}Suppose $u$ is an entropy solution to the
Dirichlet problem $\Pi(\Phi,\xi)$. Under Assumptions \ref{assu:=00005CPhi}, \ref{assu:Coefficients} and \ref{assu:zero condition},
 if $\xi\in L_{m+1}(\Omega,\mathcal{F}_{0};L_{m+1}(D))$,
we have
\begin{align}
\lim_{\gamma\rightarrow0^+}\frac{1}{\gamma}\mathbb{E}\int_{0}^{\gamma}\int_{x}|u(t,x)-\xi(x)|^{2} & \mathrm{d}t=0.\label{eq:initial value}
\end{align}
\end{lem}

\begin{proof}
From the partition of unity in \Cref{sec:-property-and--esitmate}, we
can fix $i\in\{0,1,\ldots,N\}$ and define $B\coloneqq B_{i}$, $\psi\coloneqq\psi_{i}$
and $\varrho_{\varepsilon}\coloneqq\varrho_{\varepsilon,i}$. Notice
that $\text{dist}(\text{supp}\,\psi,\partial B)>0$. When $\varepsilon$
is small enough, we have $\text{supp}\,\varrho_{\varepsilon}(\cdot-x)\subset B$
for all $x\in\text{supp}\,\psi$. Then, from the definition of $\varrho_{\varepsilon}$,
we have $\psi(\cdot)\varrho_{\varepsilon}(y-\cdot)\in C_{c}^{\infty}(D)$
for all $y\in\overline{D}$ and sufficient small $\varepsilon>0$.
Now, we only need to prove
\[
\lim_{\gamma\rightarrow0^+}\frac{1}{\gamma}\mathbb{E}\int_{0}^{\gamma}\int_{x}|u(t,x)-\xi(x)|^{2}\psi(x)\mathrm{d}t=0.
\]
We split it into three parts
\begin{align}
 & \frac{1}{\gamma}\mathbb{E}\int_{0}^{\gamma}\int_{x}|u(t,x)-\xi(x)|^{2}\psi(x)\mathrm{d}t\label{eq:split into three}\\
 & \leq2\mathbb{E}\int_{D}\int_{\mathbb{R}^{d}}|\bar{\xi}(y)-\xi(x)|^{2}\psi(x)\varrho_{\varepsilon}(y-x)\mathrm{d}y\mathrm{d}x\nonumber \\
 & \quad+\frac{2}{\gamma}\mathbb{E}\int_{0}^{\gamma}\int_{D}\int_{D}|u(t,x)-\bar{\xi}(y)|^{2}\psi(x)\varrho_{\varepsilon}(y-x)\mathrm{d}y\mathrm{d}x\mathrm{d}t\nonumber \\
 & \quad+\frac{2}{\gamma}\mathbb{E}\int_{0}^{\gamma}\int_{D}\int_{D_{y,\varepsilon}^{1}}|u(t,x)|^{2}\psi(x)\varrho_{\varepsilon}(y-x)\mathrm{d}y\mathrm{d}x\mathrm{d}t,\nonumber 
\end{align}
where $\gamma\in[0,T)$ and $D_{y,\varepsilon}^{1}\coloneqq\{y\in\mathbb{R}^{d}:\varrho_{\varepsilon}(y-x)>0,\ \exists x\in D\}\setminus D$.
See \cref{eq:xi bar} for the definition of $\bar{\xi}$. Notice
that $|D_{y,\varepsilon}^{1}|=O(\varepsilon).$ We first estimate
the third term on the right hand side as
\begin{align}
 & \frac{2}{\gamma}\mathbb{E}\int_{0}^{\gamma}\int_{D}\int_{D_{y,\varepsilon}^{1}}|u(t,x)|^{2}\psi(x)\varrho_{\varepsilon}(y-x)\mathrm{d}y\mathrm{d}x\mathrm{d}t\label{eq:u^2 in small domain}\\
 & =\frac{2}{\gamma}\mathbb{E}\int_{0}^{\gamma}\int_{D_{x,\varepsilon}^{2}}\int_{D_{y,\varepsilon}^{1}}|u(t,x)|^{2}\psi(x)\varrho_{\varepsilon}(y-x)\mathrm{d}y\mathrm{d}x\mathrm{d}t\nonumber \\
 & \leq C\underset{t\in(0,\gamma)}{\mathrm{ess\,sup}}\,\mathbb{E}\int_{D_{x,\varepsilon}^{2}}|u(t,x)|^{2}\mathrm{d}x,\nonumber 
\end{align}
where $D_{x,\varepsilon}^{2}\coloneqq\{x\in D:\varrho_{\varepsilon}(y-x)>0,\ \exists y\in D_{y,\varepsilon}^{1}\}$.
We also have $|D_{x,\varepsilon}^{2}|=O(\varepsilon)$. Now, we choose
a non-negative function $w_{\varepsilon}\in C_{c}^{\infty}(\mathbb{R}^{d})$
such that $w_{\varepsilon}(x)=1$ for all $x\in D_{x,\varepsilon}^{2}$
and $|\text{supp}\,w_{\varepsilon}|\leq C\varepsilon$. Suppose that
$s\in(0,\gamma)$ is a Lebesgue point of the function 
\[
t\mapsto\mathbb{E}\int_{x}|u(t,x)|^{2}w_{\varepsilon}(x),
\]
and $\theta\in(0,T-\gamma)$. Define $V_{(\theta)}:[0,T]\rightarrow\mathbb{R}^{+}$
by $V_{(\theta)}(0)\coloneqq1$ and $V_{(\theta)}^{\prime}\coloneqq-\theta^{-1}\mathbf{1}_{[s,s+\theta)}$.
We take a sequence of non-negative functions $\varphi_{\theta,n}\in C_{c}^{\infty}([0,T))$
satisfying 
\[
\Vert\varphi_{\theta,n}\Vert_{L_{\infty}(0,T)}\vee\Vert\partial_{t}\varphi_{\theta,n}\Vert_{L_{1}(0,T)}\leq1,
\]
such that
\[
\lim_{n\rightarrow\infty}\Vert\varphi_{\theta,n}-V_{(\theta)}\Vert_{H^{1}(0,T)}=0.
\]
For each $\delta>0$, define $\eta_{\delta}\in C^{2}(\mathbb{R})$
by
\begin{equation}
\eta_{\delta}(0)=\eta_{\delta}^{\prime}(0)=0,\quad\eta_{\delta}^{\prime\prime}(r)\coloneqq2\cdot\mathbf{1}_{[0,\delta^{-1})}(|r|)+(-|r|+\delta^{-1}+2)\cdot\mathbf{1}_{[\delta^{-1},\delta^{-1}+2)}(|r|).\label{eq:new eta}
\end{equation}
It is easy to find that $\eta_{\delta}(r)\rightarrow r^{2}$ as $\delta\rightarrow0^{+}$.
Using the entropy inequality \cref{eq:entropy formulation} for $u$
with $\eta=\eta_{\delta}\in\mathcal{E}_{0}$ and $\phi_{n}(t,x)=\varphi_{\theta,n}(t)w_{\varepsilon}(x),$
letting $n\rightarrow\infty$ and taking expectations, with 
\cref{assu:=00005CPhi} and \cref{eq:sigma x}-\cref{eq:F} in 
\cref{assu:Coefficients}, we have
\begin{align*}
 & \frac{1}{\theta}\int_{s}^{s+\theta}\mathbb{E}\int_{D}\eta_{\delta}(u)w_{\varepsilon}(x)\mathrm{d}x\mathrm{d}t\\
 & \leq\mathbb{E}\int_{D}\eta_{\delta}(\xi)w_{\varepsilon}(x)\mathrm{d}x+\mathbb{E}\int_{0}^{s+\theta}\int_{D}\eta_{\delta}^{\prime\prime}(u)|\nabla\llbracket\mathfrak{a}\rrbracket(u)|^{2}w_{\varepsilon}(x)\mathrm{d}x\mathrm{d}t\\
 & \quad+C\mathbb{E}\int_{0}^{s+\theta}\int_{D}(1+|u|^{m+1})\Big(\sum_{i,j}|\partial_{x_{i}x_{j}}w_{\varepsilon}|+\sum_{i}|\partial_{x_{i}}w_{\varepsilon}|+w_{\varepsilon}\Big)\mathrm{d}x\mathrm{d}t.
\end{align*}
Taking $\delta\rightarrow0^{+}$ and then letting $\theta\rightarrow0^{+}$,
we have
\begin{align*}
 & \mathbb{E}\int_{D}|u(s,x)|^{2}w_{\varepsilon}(x)\mathrm{d}x\mathrm{d}t\\
 & \leq\mathbb{E}\int_{D}|\xi(x)|^{2}w_{\varepsilon}(x)\mathrm{d}x+2\mathbb{E}\int_{0}^{s}\int_{D}|\nabla\llbracket\mathfrak{a}\rrbracket(u)|^{2}w_{\varepsilon}(x)\mathrm{d}x\mathrm{d}t\\
 & \quad+C\mathbb{E}\int_{0}^{s}\int_{D}(1+|u|^{m+1})\Big(\sum_{i,j}|\partial_{x_{i}x_{j}}w_{\varepsilon}|+\sum_{i}|\partial_{x_{i}}w_{\varepsilon}|+w_{\varepsilon}\Big)\mathrm{d}x\mathrm{d}t
\end{align*}
hold for almost all $s\in(0,\gamma)$. Combining \cref{eq:u^2 in small domain}
and taking $\gamma\rightarrow0^+$, we have
\begin{align}
 & \limsup_{\gamma\rightarrow0^+}\frac{2}{\gamma}\mathbb{E}\int_{0}^{\gamma}\int_{D}\int_{D_{y,\varepsilon}^{1}}|u(t,x)-\bar{\xi}(y)|^{2}\psi(x)\varrho_{\varepsilon}(y-x)\mathrm{d}y\mathrm{d}x\mathrm{d}t\label{eq:second}\\
 & \leq C\limsup_{\gamma\rightarrow0^+}\underset{t\in(0,\gamma)}{\mathrm{ess\,sup}}\,\mathbb{E}\int_{D_{x,\varepsilon}^{2}}|u(t,x)|^{2}\mathrm{d}x\leq C\mathbb{E}\int_{D\cap\text{supp}\,w_{\varepsilon}}|\xi|^{2}\mathrm{d}x.\nonumber 
\end{align}

Now, we estimate the second term on the right hand side of \cref{eq:split into three}.
For $\gamma\in[0,T]$, choose a decreasing, non-negative function
$\varpi\in C^{\infty}([0,T])$ such that
\[
\varpi(0)=2,\quad\varpi\leq2\cdot\mathbf{1}_{[0,2\gamma]},\quad\partial_{t}\varpi\leq-\frac{1}{\gamma}\cdot\mathbf{1}_{[0,\gamma]}.
\]
Note that $\psi(\cdot)\varrho_{\varepsilon}(y-\cdot)\in C_{c}^{\infty}(D)$
for all $y\in\overline{D}$ and sufficient small $\varepsilon>0$.
For fixed $(y,z)\in D\times\mathbb{R}$, using the entropy inequality
\cref{eq:entropy formulation} with $\phi(t,x)=\varpi(t)\psi(x)\varrho_{\varepsilon}(y-x)$
and $\eta(r)=\eta_{\delta}(r-z)$ defined in \cref{eq:new eta},
with Assumptions \ref{assu:=00005CPhi} and \ref{assu:Coefficients},
we have
\begin{align*}
 & -\int_{t,x}\eta_{\delta}(u-z)\partial_{t}\varpi(t)\psi(x)\varrho_{\varepsilon}(y-x)\\
 & \leq2\int_{x}\eta_{\delta}(\xi-z)\psi(x)\varrho_{\varepsilon}(y-x)+\frac{C}{\varepsilon^{2}}\int_{t,x}(1+|u|^{m+1}+|z|^{m+1})\varpi(t)\\
 & \quad-\int_{t,x}\eta_{\delta}^{\prime\prime}(u-z)|\nabla_{x}\llbracket\mathfrak{a}\rrbracket(u)|^{2}\phi+\int_{0}^{T}\int_{x}\Big(\eta_{\delta}^{\prime}(u-z)\phi\sigma_{x_{i}}^{ik}(x,u)\\
 & \quad-\llbracket\sigma_{rx_{i}}^{ik}\eta_{\delta}^{\prime}(\cdot-z)\rrbracket(x,u)\phi-\llbracket\sigma_{r}^{ik}\eta_{\delta}^{\prime}(\cdot-z)\rrbracket(x,u)\partial_{x_{i}}\phi\Big)\mathrm{d}W^{k}(t).
\end{align*}
Notice that all the terms are continuous in $z\in\mathbb{R}$. Replacing
$z$ by $\xi(y)$ by convolution, taking expectations, integrating
over $y\in D$ and taking $\delta\rightarrow0^{+},\gamma\rightarrow0^+$,
with the definition of $\varpi$, we have
\begin{align}
 & \limsup_{\gamma\rightarrow0^+}\frac{1}{\gamma}\int_{0}^{\gamma}\mathbb{E}\int_{x,y}|u(t,x)-\xi(y)|^{2}\psi(x)\varrho_{\varepsilon}(y-x)\mathrm{d}t\label{eq:third}\\
 & \leq2\mathbb{E}\int_{x,y}|\xi(x)-\xi(y)|^{2}\psi(x)\varrho_{\varepsilon}(x-y).\nonumber 
\end{align}
Combining \cref{eq:split into three} with \cref{eq:second}-\cref{eq:third},
we have
\begin{align*}
 & \limsup_{\gamma\rightarrow0^+}\frac{1}{\gamma}\mathbb{E}\int_{0}^{\gamma}\int_{x}|u(t,x)-\xi(x)|^{2}\psi(x)\mathrm{d}t\\
 & \leq2\mathbb{E}\int_{D}\int_{\mathbb{R}^{d}}|\bar{\xi}(y)-\xi(x)|^{2}\psi(x)\varrho_{\varepsilon}(y-x)\mathrm{d}y\mathrm{d}x\\
 & \quad+4\mathbb{E}\int_{x,y}|\xi(x)-\xi(y)|^{2}\psi(x)\varrho_{\varepsilon}(x-y)+C\mathbb{E}\int_{D\cap\text{supp}\,w_{\varepsilon}}|\xi|^{2}\mathrm{d}x.
\end{align*}
Since $\xi\in L_{m+1}(\Omega\times D)$, the right hand side goes
to $0$ as $\varepsilon\rightarrow0^+$. The proof is complete.
\end{proof}
\begin{lem}
\label{lem:approximation}Let $u\in L_{m+1}(\Omega\times D_{T})$
and for some $\varepsilon\in(0,1)$. Denote by $\tilde{K}$ the constant in
 \cref{rem:support of rho}. Let $\varrho:\mathbb{R}^{d}\rightarrow\mathbb{R}$
be a non-negative function, integrating to $1$ and supported on a
ball of radius $\tilde{K}\varepsilon$ centered at the origin. Under 
\cref{assu:=00005CPhi}, one has
\begin{align}
 & \mathbb{E}\int_{t,x,y}|u(t,x)-u(t,y)|\varrho(x-y)\label{eq:approximation in u}\\
 & \leq C\varepsilon^{\frac{1}{m+1}}\mathbb{E}(1+\Vert u\Vert_{L_{m+1}(D_{T})}^{m+1}+\Vert\nabla\llbracket\mathfrak{a}\rrbracket(u)\Vert_{L_{1}(D_{T})}),\nonumber 
\end{align}
\begin{align}
 & \mathbb{E}\int_{t,x,y}|\Phi(u(t,x))-\Phi(u(t,y))|\varrho(x-y)\label{eq:approximation in Phi}\\
 & \leq C\varepsilon^{\frac{1}{m+1}}\mathbb{E}(1+\Vert u\Vert_{L_{m+1}(D_{T})}^{m+1}+\Vert\nabla\llbracket\mathfrak{a}\rrbracket(u)\Vert_{L_{2}(D_{T})}^{2}),\nonumber 
\end{align}
where $C$ depends on $d$, $K$, $\tilde{K}$ and $T$.
\end{lem}

\begin{proof}
We first prove \cref{eq:approximation in Phi}. Define a set
\[
D_{-\varepsilon}\coloneqq\{y\in D:y+z\in D,\ \text{for all }z\in\mathbb{R}^{d}\ \text{with }|z|\leq \tilde{K}\varepsilon\}.
\]
Notice that $|D\setminus D_{-\varepsilon}|=O(\varepsilon)$, combining
with \cref{assu:=00005CPhi} and H\"{o}lder's inequality,
we have
\begin{align}
 & \mathbb{E}\int_{t,x,y}|\Phi(u(t,x))-\Phi(u(t,y))|\varrho(x-y)\label{eq:Phi-Phi with varrho}\\
 & \leq\mathbb{E}\int_{t}\int_{D_{-\varepsilon}}\int_{\mathbb{R}^{d}}|\Phi(u(t,y+z))-\Phi(u(t,y))|\varrho(z)\mathrm{d}z\mathrm{d}y\nonumber\\
 & \quad+\mathbb{E}\int_{t,x}\int_{D\setminus D_{-\varepsilon}}|\Phi(u(t,x))-\Phi(u(t,y))|\varrho(x-y)\mathrm{d}y\nonumber\\
 & \leq C\varepsilon\mathbb{E}\int_{t}\int_{D_{-\varepsilon}}\int_{\mathbb{R}^{d}}\varrho(z)\int_{0}^{1}|\mathfrak{a}(u)\nabla\llbracket\mathfrak{a}\rrbracket(u)|(y+\lambda z)\mathrm{d}\lambda\mathrm{d}z\mathrm{d}y\nonumber\\
 & \quad+\Big(\mathbb{E}\int_{t,x}\int_{D\setminus D_{-\varepsilon}}|\Phi(u(t,x))-\Phi(u(t,y))|^{\frac{m+1}{m}}\varrho(x-y)\mathrm{d}y\Big)^{\frac{m}{m+1}}\nonumber\\
 & \quad\cdot\Big(\mathbb{E}\int_{t,x}\int_{D\setminus D_{-\varepsilon}}\varrho(x-y)\mathrm{d}y\Big)^{\frac{1}{m+1}}\nonumber\\
 & \leq C\varepsilon^{\frac{1}{m+1}}\mathbb{E}(1+\Vert u\Vert_{L_{m+1}(D_{T})}^{m+1}+\Vert\nabla\llbracket\mathfrak{a}\rrbracket(u)\Vert_{L_{2}(D_{T})}^{2}).\nonumber
\end{align}

Moreover, using \cref{assu:=00005CPhi} as in the proof of
\cite[Lemma 3.1]{dareiotis2019entropy},
we have 
\begin{align}
&\mathbb{E}\int_{t,x,y}|u(t,x)-u(t,y)|\varrho(x-y)\label{eq:u-u varrho to a-a}\\
&\leq C\mathbb{E}\int_{t,x,y}\mathbf{1}_{|u(t,x)|\vee|u(t,y)|\geq 1}|\llbracket\mathfrak{a}\rrbracket(u(t,x))-\llbracket\mathfrak{a}\rrbracket(u(t,y))|\varrho(x-y)\nonumber\\
&\quad+C\Big(\mathbb{E}\int_{t,x,y}\mathbf{1}_{|u(t,x)|\vee|u(t,y)|\leq 1}|\llbracket\mathfrak{a}\rrbracket(u(t,x))-\llbracket\mathfrak{a}\rrbracket(u(t,y))|\varrho(x-y)\Big)^\frac{2}{m+1}.\nonumber
\end{align}
Besides, as \cref{eq:Phi-Phi with varrho}, we have
\begin{align*}
&\mathbb{E}\int_{t,x,y}|\llbracket\mathfrak{a}\rrbracket(u(t,x))-\llbracket\mathfrak{a}\rrbracket(u(t,y))|\varrho(x-y)\\
&\leq C\varepsilon\mathbb{E}\int_{t}\int_{D_{-\varepsilon}}\int_{\mathbb{R}^d}\varrho(z)\int_{0}^{1}|\nabla\llbracket\mathfrak{a}\rrbracket(u)(y+\lambda z)|\mathrm{d}\lambda\mathrm{d}z\mathrm{d}y\nonumber\\
 & \quad+\Big(\mathbb{E}\int_{t,x}\int_{D\setminus D_{-\varepsilon}}|\llbracket\mathfrak{a}\rrbracket(u(t,x))-\llbracket\mathfrak{a}\rrbracket(u(t,y))|^2\varrho(x-y)\mathrm{d}y\Big)^{\frac{1}{2}}\nonumber\\
 & \quad\cdot\Big(\mathbb{E}\int_{t,x}\int_{D\setminus D_{-\varepsilon}}\varrho(x-y)\mathrm{d}y\Big)^{\frac{1}{2}}\nonumber\\
 & \leq C\varepsilon^{\frac{1}{2}}\mathbb{E}(1+\Vert u\Vert_{L_{m+1}(D_{T})}^{m+1}+\Vert\nabla\llbracket\mathfrak{a}\rrbracket(u)\Vert_{L_{2}(D_{T})}^{2}).\nonumber
\end{align*}
Combining with \cref{eq:u-u varrho to a-a}, we obtain inequality \cref{eq:approximation in u}.
\end{proof}

\begin{acknowledgement*}
	The research of K.~Du was partially supported by National Natural Science
	Foundation of China (12222103), by National Key R{\&}D Program of China (2018YFA0703900), 
	and by Natural Science Foundation of Shanghai (20ZR1403600).
\end{acknowledgement*}
\bibliographystyle{amsalpha}
\bibliography{bi}

\providecommand{\bysame}{\leavevmode\hbox to3em{\hrulefill}\thinspace}
\providecommand{\MR}{\relax\ifhmode\unskip\space\fi MR }
\providecommand{\MRhref}[2]{%
  \href{http://www.ams.org/mathscinet-getitem?mr=#1}{#2}
}
\providecommand{\href}[2]{#2}
\begin{thebibliography}{BDPR16}

\bibitem[ABK00]{andreianov2000l1}
Boris~P. Andreianov, Philippe B\'{e}nilan, and Stanislav~N. Kruzhkov,
  \emph{{$L^1$}-theory of scalar conservation law with continuous flux
  function}, J. Funct. Anal. \textbf{171} (2000), no.~1, 15--33. \MR{1742856}

\bibitem[BDPR08]{barbu2008existence}
Viorel Barbu, Giuseppe Da~Prato, and Michael R\"{o}ckner, \emph{Existence and
  uniqueness of nonnegative solutions to the stochastic porous media equation},
  Indiana Univ. Math. J. \textbf{57} (2008), no.~1, 187--211. \MR{2400255}

\bibitem[BDPR16]{barbu2016stochastic}
\bysame, \emph{Stochastic porous media equations}, Lecture Notes in
  Mathematics, vol. 2163, Springer, [Cham], 2016. \MR{3560817}

\bibitem[BGV20]{bavnas2020numerical}
L'ubom{\'\i}r Ba{\v{n}}as, Benjamin Gess, and Christian Vieth, \emph{Numerical
  approximation of singular-degenerate parabolic stochastic {PDE}s}, arXiv
  preprint arXiv:2012.12150 (2020).

\bibitem[Bre11]{brezis2011functional}
Haim Brezis, \emph{Functional analysis, {S}obolev spaces and partial
  differential equations}, Universitext, Springer, New York, 2011. \MR{2759829}

\bibitem[BVW12]{bauzet2012cauchy}
Caroline Bauzet, Guy Vallet, and Petra Wittbold, \emph{The {C}auchy problem for
  conservation laws with a multiplicative stochastic perturbation}, J.
  Hyperbolic Differ. Equ. \textbf{9} (2012), no.~4, 661--709. \MR{3021756}

\bibitem[BVW14]{bauzet2014dirichlet}
\bysame, \emph{The {D}irichlet problem for a conservation law with a
  multiplicative stochastic perturbation}, J. Funct. Anal. \textbf{266} (2014),
  no.~4, 2503--2545. \MR{3150169}

\bibitem[BVW15]{bauzet2015degenerate}
\bysame, \emph{A degenerate parabolic-hyperbolic {C}auchy problem with a
  stochastic force}, J. Hyperbolic Differ. Equ. \textbf{12} (2015), no.~3,
  501--533. \MR{3401975}

\bibitem[Car99]{carrillo1999entropy}
Jos\'{e} Carrillo, \emph{Entropy solutions for nonlinear degenerate problems},
  Arch. Ration. Mech. Anal. \textbf{147} (1999), no.~4, 269--361. \MR{1709116}

\bibitem[CDK12]{chen2012nonlinear}
Gui-Qiang Chen, Qian Ding, and Kenneth~H. Karlsen, \emph{On nonlinear
  stochastic balance laws}, Arch. Ration. Mech. Anal. \textbf{204} (2012),
  no.~3, 707--743. \MR{2917120}

\bibitem[Cio20]{ciotir2020stochastic}
Ioana Ciotir, \emph{Stochastic porous media equations with divergence {I}t\^{o}
  noise}, Evol. Equ. Control Theory \textbf{9} (2020), no.~2, 375--398.
  \MR{4097646}

\bibitem[Cli23]{clini2023porous}
Andrea Clini, \emph{Porous media equations with nonlinear gradient noise and
  {D}irichlet boundary conditions}, Stochastic Process. Appl. \textbf{159}
  (2023), 428--498. \MR{4555282}

\bibitem[Daf05]{dafermos2005hyperbolic}
Constantine~M. Dafermos, \emph{Hyperbolic conservation laws in continuum
  physics}, second ed., Grundlehren der mathematischen Wissenschaften
  [Fundamental Principles of Mathematical Sciences], vol. 325, Springer-Verlag,
  Berlin, 2005. \MR{2169977}

\bibitem[DG19]{dareiotis2019supremum}
Konstantinos Dareiotis and Benjamin Gess, \emph{Supremum estimates for
  degenerate, quasilinear stochastic partial differential equations}, Ann.
  Inst. Henri Poincar\'{e} Probab. Stat. \textbf{55} (2019), no.~3, 1765--1796.
  \MR{4010951}

\bibitem[DG20]{dareiotis2020nonlinear}
\bysame, \emph{Nonlinear diffusion equations with nonlinear gradient noise},
  Electron. J. Probab. \textbf{25} (2020), Paper No. 35, 43. \MR{4089785}

\bibitem[DGG19]{dareiotis2019entropy}
K.~Dareiotis, M.~Gerencs\'{e}r, and B.~Gess, \emph{Entropy solutions for
  stochastic porous media equations}, J. Differential Equations \textbf{266}
  (2019), no.~6, 3732--3763. \MR{3912697}

\bibitem[DGT20]{dareiotis2020ergodicity}
Konstantinos Dareiotis, Benjamin Gess, and Pavlos Tsatsoulis, \emph{Ergodicity
  for stochastic porous media equations with multiplicative noise}, SIAM J.
  Math. Anal. \textbf{52} (2020), no.~5, 4524--4564. \MR{4154312}

\bibitem[DHV16]{debussche2016degenerate}
Arnaud Debussche, Martina Hofmanov\'{a}, and Julien Vovelle, \emph{Degenerate
  parabolic stochastic partial differential equations: quasilinear case}, Ann.
  Probab. \textbf{44} (2016), no.~3, 1916--1955. \MR{3502597}

\bibitem[DSZ16]{dirr2016entropic}
Nicolas Dirr, Marios Stamatakis, and Johannes Zimmer, \emph{Entropic and
  gradient flow formulations for nonlinear diffusion}, J. Math. Phys.
  \textbf{57} (2016), no.~8, 081505, 13. \MR{3534824}

\bibitem[FG19]{fehrman2019well}
Benjamin Fehrman and Benjamin Gess, \emph{Well-posedness of nonlinear diffusion
  equations with nonlinear, conservative noise}, Arch. Ration. Mech. Anal.
  \textbf{233} (2019), no.~1, 249--322. \MR{3974641}

\bibitem[FG21a]{fehrman2021path}
\bysame, \emph{Path-by-path well-posedness of nonlinear diffusion equations
  with multiplicative noise}, J. Math. Pures Appl. (9) \textbf{148} (2021),
  221--266. \MR{4223353}

\bibitem[FG21b]{fehrman2021well}
\bysame, \emph{Well-posedness of the {D}ean-{K}awasaki and the nonlinear
  {D}awson-{W}atanabe equation with correlated noise}, arXiv preprint
  arXiv:2108.08858 (2021).

\bibitem[FN08]{feng2008stochastic}
Jin Feng and David Nualart, \emph{Stochastic scalar conservation laws}, J.
  Funct. Anal. \textbf{255} (2008), no.~2, 313--373. \MR{2419964}

\bibitem[GH18]{gess2018well}
Benjamin Gess and Martina Hofmanov\'{a}, \emph{Well-posedness and regularity
  for quasilinear degenerate parabolic-hyperbolic {SPDE}}, Ann. Probab.
  \textbf{46} (2018), no.~5, 2495--2544. \MR{3846832}

\bibitem[GS15]{gess2015scalar}
Benjamin Gess and Panagiotis~E. Souganidis, \emph{Scalar conservation laws with
  multiple rough fluxes}, Commun. Math. Sci. \textbf{13} (2015), no.~6,
  1569--1597. \MR{3351442}

\bibitem[GS17a]{gess2017long}
\bysame, \emph{Long-time behavior, invariant measures, and regularizing effects
  for stochastic scalar conservation laws}, Comm. Pure Appl. Math. \textbf{70}
  (2017), no.~8, 1562--1597. \MR{3666564}

\bibitem[GS17b]{gess2017stochastic}
\bysame, \emph{Stochastic non-isotropic degenerate parabolic-hyperbolic
  equations}, Stochastic Process. Appl. \textbf{127} (2017), no.~9, 2961--3004.
  \MR{3682120}

\bibitem[Hen21]{hensel2021finite}
Sebastian Hensel, \emph{Finite time extinction for the 1{D} stochastic porous
  medium equation with transport noise}, Stoch. Partial Differ. Equ. Anal.
  Comput. \textbf{9} (2021), no.~4, 892--939. \MR{4333506}

\bibitem[Kim03]{kim2003stochastic}
Jong~Uhn Kim, \emph{On a stochastic scalar conservation law}, Indiana Univ.
  Math. J. \textbf{52} (2003), no.~1, 227--256. \MR{1970028}

\bibitem[Kru70]{kruvzkov1970first}
S.~N. Kru\v{z}kov, \emph{First order quasilinear equations with several
  independent variables}, Mat. Sb. (N.S.) (1970), 228--255. \MR{267257}

\bibitem[Kry13]{krylov2013relatively}
N.~V. Krylov, \emph{A relatively short proof of {I}t\^{o}'s formula for {SPDE}s
  and its applications}, Stoch. Partial Differ. Equ. Anal. Comput. \textbf{1}
  (2013), no.~1, 152--174. \MR{3327504}

\bibitem[Kun97]{kunita1997stochastic}
Hiroshi Kunita, \emph{Stochastic flows and stochastic differential equations},
  Cambridge Studies in Advanced Mathematics, vol.~24, Cambridge University
  Press, Cambridge, 1997, Reprint of the 1990 original. \MR{1472487}

\bibitem[LL06a]{lasry2006jeuxa}
Jean-Michel Lasry and Pierre-Louis Lions, \emph{Jeux {\`a} champ moyen. {I} -
  {L}e cas stationnaire}, Comptes Rendus Math{\'e}matique \textbf{343} (2006),
  no.~9, 619--625.

\bibitem[LL06b]{lasry2006jeux}
\bysame, \emph{Jeux {\`a} champ moyen. {II} - {Horizon} fini et contr{\^o}le
  optimal}, Comptes Rendus Math{\'e}matique \textbf{343} (2006), no.~10,
  679--684.

\bibitem[LL07]{lasry2007mean}
\bysame, \emph{Mean field games}, Jpn. J. Math. \textbf{2} (2007), no.~1,
  229--260. \MR{2295621}

\bibitem[LPS13]{lions2013scalar}
Pierre-Louis Lions, Beno\^{i}t Perthame, and Panagiotis~E. Souganidis,
  \emph{Scalar conservation laws with rough (stochastic) fluxes}, Stoch.
  Partial Differ. Equ. Anal. Comput. \textbf{1} (2013), no.~4, 664--686.
  \MR{3327520}

\bibitem[LPS14]{lions2014scalar}
\bysame, \emph{Scalar conservation laws with rough (stochastic) fluxes: the
  spatially dependent case}, Stoch. Partial Differ. Equ. Anal. Comput.
  \textbf{2} (2014), no.~4, 517--538. \MR{3274890}

\bibitem[LT21]{liu2021obstacle}
Ruoyang Liu and Shanjian Tang, \emph{The obstacle problem for stochastic porous
  media equations}, arXiv preprint arXiv:2111.10736 (2021).

\bibitem[LW12]{li2012homogeneous}
Yachun Li and Qin Wang, \emph{Homogeneous {D}irichlet problems for quasilinear
  anisotropic degenerate parabolic-hyperbolic equations}, J. Differential
  Equations \textbf{252} (2012), no.~9, 4719--4741. \MR{2891345}

\bibitem[RRW07]{ren2007stochastic}
Jiagang Ren, Michael R\"{o}ckner, and Feng-Yu Wang, \emph{Stochastic
  generalized porous media and fast diffusion equations}, J. Differential
  Equations \textbf{238} (2007), no.~1, 118--152. \MR{2334594}

\bibitem[Ser99]{serre1999systems}
Denis Serre, \emph{Systems of conservation laws. 1}, Cambridge University
  Press, Cambridge, 1999, Hyperbolicity, entropies, shock waves, Translated
  from the 1996 French original by I. N. Sneddon. \MR{1707279}

\bibitem[T\"20]{tolle2020stochastic}
Jonas~M. T\"{o}lle, \emph{Stochastic evolution equations with singular drift
  and gradient noise via curvature and commutation conditions}, Stochastic
  Process. Appl. \textbf{130} (2020), no.~5, 3220--3248. \MR{4080744}

\bibitem[V\'07]{vazquez2007porous}
Juan~Luis V\'{a}zquez, \emph{The porous medium equation}, Oxford Mathematical
  Monographs, The Clarendon Press, Oxford University Press, Oxford, 2007,
  Mathematical theory. \MR{2286292}

\bibitem[VW09]{vallet2009stochastic}
Guy Vallet and Petra Wittbold, \emph{On a stochastic first-order hyperbolic
  equation in a bounded domain}, Infin. Dimens. Anal. Quantum Probab. Relat.
  Top. \textbf{12} (2009), no.~4, 613--651. \MR{2590159}

\end{thebibliography}

\end{document}